\documentclass[final]{siamltex}
\usepackage{amsfonts}
\usepackage{graphics}
\usepackage{caption2}
\usepackage{amssymb}
\usepackage{psfrag}
\usepackage{amsmath}
\usepackage{graphicx}
\usepackage{multirow}
\usepackage{cite}
\usepackage{subfigure}
\usepackage[normalem]{ulem}
\usepackage{color}

\allowdisplaybreaks[4]

\newcommand{\bm}[1]{\mbox{\boldmath{$#1$}}}

\def\x{{\bf   x}}

\def\u{{\bf   u}}

\def\v{{\bf   v}}

\def\F{{\bf   F}}
\def\I{{\bf   I}}
\def\J{{\bm   J}}
\def\n{{\bf   n}}

\def\({\left(}
\def\[{\left[}
\def\){\right)}
\def\]{\right]}
\def\div{\nabla\cdot }
\def\grad{\nabla }

\newtheorem{thm}{Theorem}

\numberwithin{equation}{section}
\numberwithin{thm}{section}
\numberwithin{lem}{section}

\begin{document}
\title{Thermodynamically consistent modeling and  simulation of  multi-component two-phase  flow  model with partial miscibility\thanks{This work is  supported by   National Natural Science Foundation of China (No.11301163),  and KAUST research fund to the
Computational Transport Phenomena Laboratory.}
}

\author{Jisheng Kou\thanks{School of Mathematics and Statistics, Hubei Engineering  University, Xiaogan 432000, Hubei, China. } \and Shuyu Sun\thanks{Corresponding author. Computational Transport Phenomena Laboratory, Division of Physical Science and Engineering,
King Abdullah University of Science and
Technology, Thuwal 23955-6900, Kingdom of Saudi Arabia.   Email: {\tt shuyu.sun@kaust.edu.sa}.}
}

 \maketitle

\begin{abstract}
  A general diffuse interface model with a realistic equation of state (e.g. Peng-Robinson equation of state)  is proposed to describe the multi-component two-phase fluid  flow   based on   the principles of  the NVT-based framework which is  a latest alternative  over  the NPT-based framework to model the realistic fluids. 
The proposed model  uses  the Helmholtz free energy rather than Gibbs free energy  in the NPT-based framework.  Different from the classical  routines,   we combine  the first law of thermodynamics and  related thermodynamical relations to derive   the entropy balance equation, and then we derive a transport equation of the Helmholtz free energy density.  Furthermore, by using the second law of thermodynamics, we derive a set of unified equations for both interfaces and bulk phases that can  describe the   partial miscibility of two fluids. 
A relation between the pressure gradient  and     chemical potential gradients  is   established, and this relation leads to a new formulation of the  momentum balance equation, which demonstrates that     chemical potential gradients become the primary driving force of  fluid motion.
Moreover, we prove that the proposed model satisfies  the total (free) energy dissipation with time.  For numerical simulation  of the proposed model,  the key difficulties   result from the strong nonlinearity of Helmholtz free energy density and tight coupling relations between molar densities and velocity.  To resolve these problems,  we propose a novel convex-concave  splitting of Helmholtz free energy density and deal well with the coupling relations between molar densities and velocity through very careful physical observations with a mathematical rigor.  We prove that the proposed numerical scheme can preserve  the discrete (free) energy dissipation.  Numerical tests are carried out to verify the effectiveness of the proposed method.

\end{abstract}

\section{Introduction}

Modeling and simulation of multiphase fluid systems with  a realistic equation of state (e.g. Peng-Robinson equation of state \cite{Peng1976EOS})    has become an attractive and challenging research topic   in the chemical  and reservoir engineering   \cite{kousun2015CMA,qiaosun2014,kousun2015SISC,kousun2015CHE,jindrova2013fast,mikyvska2015General,mikyvska2011new,mikyvska2012investigation,polivka2014compositional}.  
It   plays very important role in the pore scale modeling and simulation of subsurface fluid flow, especially  shale gas reservoir that  has become an increasingly important source of natural gas in the recent years. 

The mathematical models of multiphase fluids are often formulated by a set of thermodynamic state variables and fluid velocity. In the traditional framework of modeling multiphase fluids, the thermodynamic state variables are  the  pressure, temperature, and chemical composition (the so-called NPT-based framework). This framework has been extensively used in many applications \cite{firoozabadi1999thermodynamics,Moortgat2013compositional,chen2006multiphase,kouandsun2014dgtwophase}. 
 However, the NPT-based framework  has some essential   limitations  \cite{jindrova2013fast, mikyvska2011new,mikyvska2012investigation, polivka2014compositional,mikyvska2015General}.  First,  a realistic equation of state  (e.g. Peng-Robinson equation of state)  is a cubic equation with respect to the density, so   the density   might  not  be uniquely determined  for  the specified  pressure, temperature, and  molar fractions. Second,  \cite{polivka2014compositional} states that the specification of pressure, temperature, and   mole fractions does not always determine the equilibrium state of the system uniquely. In addition, in compositional fluid simulation,  the pressure is not known a-priori and there exists no intrinsic  equation for the pressure. This leads to the complication of constructing the pressure evolution equation   for  application of  the NPT-based framework  \cite{polivka2014compositional}.  
 
 In order to resolve the issues of the NPT-based framework,    an alternative  framework \cite{mikyvska2011new} has been developed, and  it uses the moles, volume, and temperature  (the so-called  NVT-based framework)  as the primal state variables. The NVT-based framework  is initially applied to the phase-split computations  of multi-component fluids at the constant moles, volume and temperature (also called NVT flash computation) in \cite{mikyvska2011new}, and for  the  subsequent research progress of the NVT flash computation, we refer to \cite{jindrova2013fast,mikyvska2015General,mikyvska2012investigation,kousun2016Flash}.
 Recently, the NVT-based framework has been successfully applied in the numerical simulation of compositional two-phase flow in porous media \cite{polivka2014compositional}.  Very recently, in \cite{kouandsun2016multiscale}, we have extended the NVT-based framework to the pore scale and developed a   diffuse interface model to  simulate   multi-component two-phase  flow  with partial miscibility, but  the rigorous mathematical analysis is absent to demonstrate the consistency with the thermodynamical laws.  In this work, we will propose a general mathematical model for multi-component two-phase  flow  with partial miscibility, which is viewed as an extension of the NVT-based framework.  Moreover, the proposed model is derived with the   mathematical rigor based on  the thermodynamic laws and a realistic equation of state (e.g. Peng-Robinson equation of state). 
 
 The diffuse-interface models for  two-phase  incompressible immiscible fluids  have been developed in the literature, \cite{Abels2012TwoPhaseModel} for instance.  
 The proposed model in this work can characterize the compressibility and partial miscibility.   A significant contribution in  \cite{Abels2012TwoPhaseModel} is that  an additional term involving mass diffusions is introduced in the momentum equation to ensure   consistency with thermodynamics for the case of non-matched densities. As shown in our derivations of Section \ref{secModel},  we find that an analogical term that also results from  mass diffusions is crucial to establish the thermodynamically consistent model of  compressible and partially  miscible multi-component two-phase flow.  This term occurs  in the momentum equation through a different but natural derived form, which satisfies the thermodynamic laws.

A major distinction is that the NVT-based framework  uses  the Helmholtz free energy rather than Gibbs free energy used in the NPT-based framework. The first challenging problem caused by this transition is that the traditional techniques \cite{Groot2015NET,Lebon2008NET} can no longer    work well   to derive the two-phase flow model. It is well known that the entropy balance equation  plays a fundamental role in the derivations of two-phase flow model.  In the classical non-equilibrium thermodynamics \cite{Groot2015NET,Lebon2008NET},   the variation of entropy $dS$ exists,  so the Gibbs relation can be used to derive this key equation. However, when  the Helmholtz free energy is introduced in the Gibbs relation, the variation of entropy $dS$ is eliminated. As a result,   the Gibbs relation cannot  be used to derive the entropy balance equation as in the classical  thermodynamics.  In this work, we resolve this problem by combing  the first law of thermodynamics and the related physical relations, and from this, we derive   consistent entropy balance equations. 

 
Different from the NPT-based framework,  we use a  thermodynamic pressure, which  is no longer an independent state variable in the NVT-based framework, and indeed, it is a function of the molar density and temperature.  Consequently, it is never necessary to construct the pressure evolution equation although the velocity field  is  no longer divergence-free.  The proposed model is  related to the dynamic  van der Waals theory of two-phase fluid flow for a pure substance with the thermodynamic pressure, which is developed from physical  point of view      \cite{Onuki2005PRL,Onuki2007PRE}. Such model  is first proposed in \cite{Onuki2005PRL,Onuki2007PRE}, and it has  successful applications \cite{Qian2016heatflow} for example.  But the proposed model  has  substantial differences from \cite{Onuki2005PRL,Onuki2007PRE} that we consider the fluids composed of multiple components and introduce an additional term involving mass diffusions  in the momentum equation by a natural way. Moreover,  we establish the relation between the pressure gradient  and    gradients of the chemical potentials. 
Based on these relations, we derive   a new formulation of the  momentum conservation equation, which shows that the gradients of  chemical potentials  become the primary driving force of the fluid motion.
This formulation allows us  to conveniently prove that the total (free) energy of the proposed model is dissipated with time.


For numerical simulation, a main challenge   of diffuse interface models is  to design efficient numerical schemes that preserve  the discrete (free) energy dissipation \cite{shen2015SIAM,shen2016JCP,Bao2012FEM}.  There are a lot of efforts on the developments of such schemes in the literature; in particular, a notable  progress is  that   efficient numerical simulation methods for simulating the model of  \cite{Abels2012TwoPhaseModel}  have been proposed in \cite{shen2015SIAM,shen2016JCP}.  The key difficulties encountered in constructing  energy-stable numerical simulation for the proposed model are the strong nonlinearity  of Helmholtz free energy density and tight coupling relations between molar densities and velocity.  To resolve these problems,  we propose a novel convex-concave  splitting of Helmholtz free energy density and deal well with the coupling relations between molar densities and velocity through very careful physical observations with a mathematical rigor.  The proposed numerical scheme is proved to preserve  the discrete (free) energy dissipation.

The key contributions of our work are listed as below:

(1)  A general diffuse interface model with a realistic equation of state (e.g. Peng-Robinson equation of state)  is proposed to describe the multi-component two-phase fluid  flow   based on   the principles of  the NVT-based framework. 

(2) We combine  the first law of thermodynamics and the related physical relations to derive   the entropy balance equations, and then we derive a transport equation of the Helmholtz free energy density.  Finally, using the second law of thermodynamics, we derive a set of unified equations for both interfaces and bulk phases that can  describe the   partial miscibility of two fluids. 

(3) A  term involving mass diffusions is naturally included in the momentum equation to ensure  consistency with thermodynamics.

(4) We prove a relation between the pressure gradient  and   the gradients of the chemical potentials, and from this, we derive   a new formulation of the  momentum balance equation, which demonstrates that  chemical potential gradients  become the primary driving force of the fluid motion.

(5) We prove that the total (free) energy of the proposed model is dissipated with time.

(6) An energy-dissipation  numerical scheme is proposed based on a convex-concave  splitting of Helmholtz free energy density and a careful treatment of the coupling relations between molar densities and velocity.  Numerical tests are carried out to verify the effectiveness of the proposed method.

The rest of this paper is organized as   follows. In Section 2, we will  introduce the related thermodynamic relations.
 In Section 3, we derive a general model for multi-component two-phase fluid flow.  In Section 4, we prove the total  (free) energy dissipation law of the proposed model.   In Section 5, we propose an efficient energy-stable numerical method, and in   Section 6 numerical tests are carried out  to verify  effectiveness of the proposed method.   Finally,   concluding remarks are provided in Section 7.

\section{Thermodynamic relations}

\subsection{Primal thermodynamic relations}

We consider a fluid mixture  composed of $M$ components under a constant temperature ($T$).  We denote the molar density vector by $\n = [n_1,n_2,\cdots,n_M]^T$, where  $n_i$ is the molar density of the $i$th component.  Furthermore, we  denote  the overall volume by $V$, and then the moles of component $i$ is $N_i=n_iV$. The  NVT-based framework uses the moles of each component ($N_i$) and volume ($V$) together with the given constant temperature as the primal state variables. With the relation $n_i=N_i/V$, the primal state variables can be reduced into   the molar density ($\n$).

By fundamental laws of thermodynamics \cite{firoozabadi1999thermodynamics}, we have the Gibbs relation \cite{Groot2015NET}
\begin{eqnarray}\label{eqThermodynamicsLaw01}
    dU=TdS-pdV+\sum_{i=1}^M\mu_idN_i,
\end{eqnarray}
where $U$ is the internal energy, $T$ is the temperature, $S$ is the entropy,   $p$ is the pressure, and $\mu_i$ is the  chemical potential   of  component $i$.

We define the Helmholtz free energy as usual:
\begin{eqnarray}\label{eqHelmholtzenergy01}
    F=U-TS.
\end{eqnarray}
With the definition of the Helmholtz free energy, we get $dU=dF+TdS$, and thus, the Gibbs relation \eqref{eqThermodynamicsLaw01} becomes
\begin{eqnarray}\label{eqThermodynamicsLaw03}
    dF =-pdV+\sum_{i=1}^M\mu_idN_i.
\end{eqnarray}

We first note that the entropy balance equation plays a fundamental role in the derivations of two-phase flow model. In the classical  derivations \cite{Groot2015NET,Lebon2008NET},   the Gibbs relation is used to derive this key equation because it contains the variation of entropy $dS$. However, we can see from \eqref{eqThermodynamicsLaw03}  that after  the Helmholtz free energy is introduced in the Gibbs relation, the variation of entropy $dS$ is eliminated. As a result,   the Gibbs relation cannot  be used to derive the entropy equation as in the classical  theory, and it is clearly in need of  the other alternative relations.  In order to resolve this problem, we will use 
the first law of thermodynamics and the entropy structure rather than the Gibbs relation to derive the entropy balance equation.  

We now define the Helmholtz free energy density as $ f=\frac{F}{V}$, and then have its form
\begin{eqnarray}\label{eqgeneralHelmholtzenergydensity}
    f=-p+\sum_{i=1}^M\mu_in_i.
\end{eqnarray}
In general, the Helmholtz free energy density is a function of $\n$  and  temperature.

\subsection{Helmholtz free energy  and a realistic equation of state}
We introduce briefly  the formulations of Helmholtz free energy density $f_{b}(\n)$ of a homogeneous bulk  fluid determined by Peng-Robinson equation of state, which is   widely  used in the oil reservoir and chemical engineering. The proposed model can also apply the other realistic cubic  equations of state, for instance, the van der Waals equation of state \cite{Onuki2005PRL,Onuki2007PRE}, which is popularly used in physics. 

The Helmholtz free energy density $f_{b}(\n)$ of a bulk fluid is calculated by a thermodynamic model as
\begin{eqnarray}\label{eqHelmholtzEnergy_a0}
    f_b(\n)= f_b^{\textnormal{ideal}}(\n) +f_b^{\textnormal{repulsion}}(\n)+f_b^{\textnormal{attraction}}(\n),
\end{eqnarray}
where $f_b^{\textnormal{ideal}}$, $f_b^{\textnormal{repulsion}}$ and $f_b^{\textnormal{attraction}}$ are formulated in the appendix. 
 The mole-pressure-temperature form of Peng-Robinson equation of state  \cite{Peng1976EOS}  is
\begin{eqnarray}\label{eqPREOSByMolarDens}
    p = \frac{nRT}{1-bn}-\frac{an^2}{1+2bn -b^2n^2 },
\end{eqnarray}
which is a cubic equation of the overall molar density $n$.   In the mathematical sense, \eqref{eqPREOSByMolarDens}    might  have not  a unique   solution   for  the specified  pressure, temperature, and  molar fractions (that is, $p,T,y_i$ are given and thus $a,b$ are known). It is  a  well-known drawback of the NPT-Based framework. However, under the NVT-based framework, \eqref{eqPREOSByMolarDens} can always provide a unique and explicit pressure   for given  molar density and temperature. We note that for bulk fluids, the PR-EOS formulation \eqref{eqPREOSByMolarDens} is equivalent to  the pressure equation (see the details in the appendix)
\begin{eqnarray}\label{eqgeneralpressure}
    p=\sum_{i=1}^M\mu_in_i-f,
\end{eqnarray}
which results from \eqref{eqgeneralHelmholtzenergydensity}. 
However, the pressure equation \eqref{eqgeneralpressure} is more general and it can be also used to define and calculate  the pressure of an inhomogeneous fluid.

For  realistic fluids, the diffuse interfaces always exist between two phases. The interfacial partial miscibility is a phenomenon that the two-phase fluids   behave on the interfaces.  To model this feature, a local density gradient  contribution is introduced into the Helmholtz free energy density of inhomogeneous fluids. The  general form of Helmholtz free energy density (denoted by $f$) is then the sum of two contributions: Helmholtz free energy density of bulk homogeneous fluid and a local density gradient  contribution:
\begin{eqnarray}\label{eqHelmholtzfreeenergydensity}
 f=f_b+f_\grad, 
\end{eqnarray}
where  
\begin{eqnarray}\label{eqHelmholtzfreeenergydensityDensityGradient}
f_\grad=\frac{1}{2}\sum_{i,j=1}^Mc_{ij}\grad n_i\cdot\grad n_j.
 \end{eqnarray}
 Here, $c_{ij}$ is the cross influence parameter.   The density gradient  contribution accounts for the phase transition by the gradual density changes of each component  on the interfaces.  But for a bulk phase, the molar density of each component  has uniform distribution in space, and in this case,  $f_\grad$ vanishes. 
 For the  influence parameters, it is generally  considered  constant in this paper (see the appendix). 

The chemical potential of component $i$ is defined as 
  \begin{eqnarray}\label{eqgeneralchemicalpotential}
   \mu_i=\(\frac{\delta  f(\n,T)}{\delta n_i}\)_{T}=\mu_i^b-\sum_{j=1}^M\div\(c_{ij}\grad{n_j}\),
   \end{eqnarray}
  where $\mu_i^b=\(\frac{\partial  f_b(\n,T)}{\partial n_i}\)_T$ and $\frac{\delta}{\delta n_i}$ is the variational  derivative. Moreover, the general  pressure can be formulated as
\begin{eqnarray}\label{eqMultiComponentDefPresGeneralB}
    p &=& \sum_{i=1}^Mn_i \mu_i- f\nonumber\\
   &=&p_b-  \sum_{i,j=1}^Mn_i\div\(c_{ij}\grad{n_j}\)-\frac{1}{2} \sum_{i,j=1}^Mc_{ij}\nabla n_i\cdot\nabla n_j,
\end{eqnarray}
where $p_b=\sum_{i=1}^Mn_i\mu_i^b-f_b.$
 From \eqref{eqMultiComponentDefPresGeneralB}, the pressure is a function of $\n$ and $T$ in the NVT-based framework, but it is no longer an independent state variable as it is so in the NPT-based framework.

\section{Mathematical modeling of multi-component two-phase flow}\label{secModel}
In this section, we will derive the general model for the multi-component two-phase fluid flow,  in which the viscosity and density gradient contribution to free energy  are under  consideration.  First, we use the first law of thermodynamics and entropy splitting structure to derive an entropy equation, and then we derive a transport equation of Helmholtz free energy density to further reduce the entropy equation.  From the reduced entropy equation, we derive a general model of multi-component two-phase flow, which obeys the second law of thermodynamics. 

\subsection{Entropy equation}

The first law of thermodynamics states
\begin{eqnarray}\label{eqFirstThermodynamicsLawUE}
    \frac{d(U+E)}{dt}=\frac{dW}{dt}+\frac{dQ}{dt},
\end{eqnarray}
where $t$ is the time, $E$ is the  kinetic  energy, $W$ is the work done by the   face force $\F_t$, and $Q$ stands for the heat transfer from the surrounding that occurs  to keep the system temperature constant.  We split  the total entropy $S$ into a summation of two contributions.  One is the entropy of the system, denoted by $S_{\textnormal{sys}}$. The other is  the entropy of the surrounding, denoted by $S_{\textnormal{surr}}$,  which has the relation with $Q$ as 
\begin{eqnarray}\label{eqEntropySurr}
    dS_{\textnormal{surr}}=-\frac{dQ}{T}.
\end{eqnarray}
Taking into account    the relation $U=F+TS_{\textnormal{sys}}$, and using \eqref{eqFirstThermodynamicsLawUE} and \eqref{eqEntropySurr}, we have 
\begin{eqnarray}\label{eqFirstThermodynamicsLawFE}
    \frac{dS}{dt}&=&\frac{dS_{\textnormal{sys}}}{dt}+\frac{dS_{\textnormal{surr}}}{dt}\nonumber\\
    &=& \frac{dS_{\textnormal{sys}}}{dt}-\frac{1}{T}\frac{dQ}{dt}\nonumber\\
    &=& \frac{dS_{\textnormal{sys}}}{dt}-\frac{1}{T}\( \frac{d(U+E)}{dt}-\frac{dW}{dt}\)\nonumber\\
    &=& -\frac{1}{T}\frac{d(F+E)}{dt}+\frac{1}{T}\frac{dW}{dt}.
\end{eqnarray}

 We denote by  $M_{w,i}$ the molar weight of component $i$, and  define the   mass density
 of the mixture  as
\begin{eqnarray}\label{eqMultiCompononentTotalDensity}
  \rho=\sum_{i=1}^Mn_iM_{w,i}.
\end{eqnarray}
In a  time-dependent volume $V(t)$,  we define the entropy, Helmholtz free energy and  kinetic energy within $V(t)$ as
\begin{eqnarray}\label{eqDefKineticEnergy}
  S=\int_{V(t)} sdV,~~F=\int_{V(t)} fdV,~~E=\frac{1}{2}\int_{V(t)}\rho|\u|^2dV ,
\end{eqnarray}
 where   $s$ is the  entropy density and $\u$ is the mass-averaged velocity.
Applying the Reynolds transport theorem and the Gauss divergence theorem,  we deduce that
 \begin{eqnarray}\label{eqReynoldsMultiComp01}
  \frac{dS}{dt}&=&\int_{V(t)}\frac{\partial s}{\partial t}dV+\int_{V(t)}\div\(\u s\)dV,
\end{eqnarray}
\begin{eqnarray}\label{eqReynoldsMultiComp02}
 \frac{dF}{dt}=\int_{V(t)}\frac{\partial f}{\partial t}dV+\int_{V(t)}\div\(\u f\)dV,
\end{eqnarray}
and
\begin{eqnarray}\label{eqReynoldsMultiComp03a}
  \frac{dE}{dt}&=&\frac{1}{2}\int_{V(t)}\frac{\partial \(\rho \u\cdot\u\)}{\partial t}dV+\frac{1}{2}\int_{V(t)}\div\(\u \(\rho \u\cdot\u\)\)dV\nonumber\\
   &=&\int_{V(t)}\(\rho \u\cdot\frac{\partial \u}{\partial t}+\frac{1}{2}\u\cdot\u\frac{\partial\rho}{\partial t}\)dV\nonumber\\
   &&+\frac{1}{2}\int_{V(t)}\big(\(\rho \u\cdot\u\)\div\u+\(\u\cdot\u\)\u\cdot\grad\rho+ \rho\u\cdot\grad\( \u\cdot\u\)\big)dV\nonumber\\
   &=&\int_{V(t)}\(\rho \u\cdot\frac{\partial \u}{\partial t}+\frac{1}{2}\u\cdot\u\frac{\partial\rho}{\partial t}\)dV\nonumber\\
   &&+\frac{1}{2}\int_{V(t)}\big(\(\u\cdot\u\)\div\(\rho\u\)+ 2\rho\u\cdot\(\u\cdot\grad\u\)\big)dV\nonumber\\
   &=&\int_{V(t)}\rho\u\cdot\(\frac{\partial \u}{\partial t}+\u\cdot\grad\u\)dV+\frac{1}{2}\int_{V(t)}\u\cdot\u\(\frac{\partial\rho}{\partial t}+\div\(\rho\u\)\)dV.
\end{eqnarray}
 
 In presence of a fluid velocity field, the mass transfer in fluids  takes place through    the convection in addition to the diffusion  of each component.  Thus, the mass balance law   for component  $i$    gives us
\begin{eqnarray}\label{eqGeneralNSEQMass}
\frac{\partial n_i}{\partial t}+\div\(\u n_i\)+\div \J_i=0,
\end{eqnarray}
where $\J_i$ is the diffusion flux of component $i$ and its formulation will be discussed in the latter of this section.  Multiplying \eqref{eqGeneralNSEQMass} by $M_{w,i}$ and summing them from $i=1$ to $M$, we obtain  the mass balance equation
\begin{eqnarray}\label{eqMultiCompononentMassConserveMass}
\frac{\partial \rho}{\partial t}+\div(\rho\u)+\sum_{i=1}^MM_{w,i}\div\J_i=0.
\end{eqnarray} 
Substituting  \eqref{eqMultiCompononentMassConserveMass}   into \eqref{eqReynoldsMultiComp03a},  we have 
\begin{eqnarray}\label{eqReynoldsMultiComp03}
  \frac{dE}{dt}&=&\int_{V(t)}\rho\u\cdot\frac{d \u}{d t}dV-\frac{1}{2}\int_{V(t)}\sum_{i=1}^MM_{w,i}\(\div\J_i\)\( \u\cdot\u\) dV\nonumber\\
  &=&\int_{V(t)}\u\cdot\(\rho\frac{d \u}{d t}+\sum_{i=1}^MM_{w,i}\J_i\cdot\grad\u\)dV\nonumber\\
  &&-\frac{1}{2}\int_{V(t)}\sum_{i=1}^MM_{w,i}\div\( (\u\cdot\u)\J_i\) dV,
\end{eqnarray}
where  $\frac{d \u}{d t}=\frac{\partial  \u}{\partial t}+ \u\cdot\grad{ \u}.$

The work done by $\F_{t}$     is expressed as
\begin{eqnarray*}\label{eqMultiCompononentFirstThermodynamicsLaw07}
  \frac{dW}{dt}=   \int_{\partial V(t)} \F_{t}\cdot\u d\bm s.
\end{eqnarray*}
Cauchy's relation between face force $\F_{t}$ and the stress tensor $\bm\sigma$ of  component $i$  gives $\F_{t}=-\bm\sigma\cdot\bm\nu$, and as a result, 
\begin{eqnarray}\label{eqReynoldsMultiCompWork}
  \frac{dW}{dt}&=& - \int_{\partial V(t)} \(\bm\sigma\cdot\bm\nu\)\cdot\u d\bm s\nonumber\\
  &=& -\int_{V(t)} \( \bm\sigma^T:\grad\u +\u\cdot(\div\bm\sigma)\)dV,
\end{eqnarray}
where $\bm\nu$ is the unit normal vector towards  the outside  of  $V(t)$.  We note that  the other external forces including gravity force is ignored in this work, but the model derivations can be easily  extended to the cases in the presence of external forces.

Substituting \eqref{eqReynoldsMultiComp01}, \eqref{eqReynoldsMultiComp02}, \eqref{eqReynoldsMultiComp03} and \eqref{eqReynoldsMultiCompWork}  into \eqref{eqFirstThermodynamicsLawFE}, and taking into account the arbitrariness of $V(t)$, we obtain the entropy balance equation
\begin{eqnarray}\label{eqEntropyFirst}
 T\(\frac{\partial s}{\partial t}+\div(\u s)\)&=&\frac{1}{2} \sum_{i=1}^MM_{w,i}\div\( (\u\cdot\u)\J_i\)-\frac{\partial  f}{\partial t}-\div\(\u  f\)- \bm\sigma^T:\grad\u\nonumber\\
&&~~
 -\u\cdot\(\rho\frac{d \u}{d t}+\sum_{i=1}^MM_{w,i}\J_i\cdot\grad\u+\div\bm\sigma\)  .
\end{eqnarray}
In the following subsections, we will derive the transport equation of Helmholtz free energy density to reduce the entropy equation.

 \subsection{Transport equation of Helmholtz free energy density}
 
 From the definition of $p_b$, we have
 \begin{eqnarray}\label{eqMultiCompononentMultiPresChptl01b}
\grad p_b=\grad \(\sum_{i=1}^Mn_i\mu_i^b-f_b\)=\sum_{i=1}^M\(n_i\grad\mu_i^b+\mu_i^b\grad n_i-\mu_i^b\grad n_i\)=\sum_{i=1}^Mn_i\grad\mu_i^b.
\end{eqnarray}
Using the relation \eqref{eqMultiCompononentMultiPresChptl01b} and component mass balance equation \eqref{eqGeneralNSEQMass}, we  derive the transport equation of Helmholtz free energy density $f_b$ as
\begin{eqnarray}\label{eqMultiCompononentHelmholtzMultiTransport01}
 \frac{\partial f_b}{\partial t}
 &=&\sum_{i=1}^M\mu_i^b\frac{\partial n_i}{\partial t}=-\sum_{i=1}^M\mu_i^b\(\div(n_i\u)+\div\J_i\)\nonumber\\
 &=&-\div\(\sum_{i=1}^Mn_i\mu_i^b\u-p_b\u\)-\div(p_b\u)+\sum_{i=1}^Mn_i\u\cdot\grad\mu_i^b-\sum_{i=1}^M\mu_i^b\div\J_i\nonumber\\
 &=&-\div(f_b\u)-\div(p_b\u)+\u\cdot\grad p_b-\sum_{i=1}^M\mu_i^b\div\J_i\nonumber\\
 &=&-\div(f_b\u)-p_b\div\u-\sum_{i=1}^M\mu_i^b\div\J_i.
\end{eqnarray}
The   gradient contribution of Helmholtz free energy density can be formulated as
\begin{eqnarray}\label{eqMultiCompononentHelmholtzMultiTransport02}
\frac{\partial f_\grad}{\partial t}
 &=&\frac{1}{2}\frac{\partial\(\sum_{i,j=1}^M c_{ij} \nabla n_i\cdot\nabla n_j\)}{\partial{t}}=\sum_{i,j=1}^M c_{ij} \nabla n_i\cdot\nabla\frac{\partial n_j}{\partial{t}}\nonumber\\
 &=&-\sum_{i,j=1}^M c_{ij} \nabla n_i\cdot\nabla\big(\div(\u n_j)+\div\J_j\big)\nonumber\\
&=&-\sum_{i,j=1}^M\div\big(\(\div\(\u n_j\)\)c_{ij}\grad{n_i}\big)-\sum_{i,j=1}^M  \div\big(\(\div\J_j\)c_{ij}\grad{n_i}\big)\nonumber\\
 &&+\sum_{i,j=1}^Mn_j\(\div\u \)\div\(c_{ij}\grad{n_i}\)
+\sum_{i,j=1}^M\(\u\cdot\grad{n_j}\)\div\(c_{ij}\grad{n_i}\)\nonumber\\
 &&+\sum_{i,j=1}^M  \(\div\J_j\)\div\(c_{ij}\grad{n_i}\),
 \end{eqnarray}
 and
 \begin{eqnarray}\label{eqMultiCompononentHelmholtzMultiTransport03}
\div(f_\grad\u)
 &=&\frac{1}{2}\div\(\u\sum_{i,j=1}^M c_{ij} \nabla n_i\cdot\nabla n_j\)\nonumber\\
 &=&\frac{1}{2}\(\sum_{i,j=1}^M c_{ij} \nabla n_i\cdot\nabla n_j\)\div\u+\frac{1}{2}\u\cdot\grad\(\sum_{i,j=1}^M c_{ij} \nabla n_i\cdot\nabla n_j\).
 \end{eqnarray}
 Combining  \eqref{eqMultiCompononentHelmholtzMultiTransport01}-\eqref{eqMultiCompononentHelmholtzMultiTransport03}, we deduce the transport equation of the   Helmholtz free energy $ f$  as
\begin{eqnarray}\label{eqMultiCompononentHelmholtzMultiTransport05}
 \frac{\partial  f}{\partial t}+\div( f\u)
 &=&\frac{\partial f_b}{\partial t}+\div(f_b\u)+\frac{\partial f_\grad}{\partial t}+\div(f_\grad\u)\nonumber\\
 &=&-p_b\div\u-\sum_{i=1}^M\mu_i^b\div\J_i-\sum_{i,j=1}^M\div\big(\(\div\(\u n_j\)\)c_{ij}\grad{n_i}\big)\nonumber\\
 &&+\sum_{i,j=1}^M\(\div\u \)n_i\div\(c_{ij}\grad{n_j}\)
+\sum_{i,j=1}^M\(\u\cdot\grad{n_j}\)\div\(c_{ij}\grad{n_i}\)\nonumber\\
 &&-\sum_{i,j=1}^M  \div\big(\(\div\J_j\)c_{ij}\grad{n_i}\big)+\sum_{i,j=1}^M  \(\div\J_j\)\div\(c_{ij}\grad{n_i}\)\nonumber\\
 &&+\frac{1}{2}\(\sum_{i,j=1}^M c_{ij} \nabla n_i\cdot\nabla n_j\)\div\u+\frac{1}{2}\u\cdot\sum_{i,j=1}^M\grad\( c_{ij} \nabla n_i\cdot\nabla n_j\)\nonumber\\
 &=&-\(p_b-  \sum_{i,j=1}^Mn_i\div\(c_{ij}\grad{n_j}\)-\frac{1}{2} \sum_{i,j=1}^Mc_{ij}\nabla n_i\cdot\nabla n_j\)\div\u\nonumber\\
 &&-\sum_{i=1}^M\mu_i^b\div\J_i-\sum_{i,j=1}^M\div\big(\(\div\(\u n_j\)\)c_{ij}\grad{n_i}\big)\nonumber\\
 &&+\sum_{i,j=1}^M\(\u\cdot\grad{n_j}\)\div\(c_{ij}\grad{n_i}\)-\sum_{i,j=1}^M  \div\big(\(\div\J_j\)c_{ij}\grad{n_i}\big)\nonumber\\
 &&+\sum_{i,j=1}^M  \(\div\J_i\)\div\(c_{ij}\grad{n_j}\)+\frac{1}{2}\u\cdot\sum_{i,j=1}^M\grad\( c_{ij} \nabla n_i\cdot\nabla n_j\)\nonumber\\
 &=&- p ~\div\u-\sum_{i=1}^M\mu_i\div\J_i+\sum_{i,j=1}^M\(\u\cdot\grad{n_i}\)\div\(c_{ij}\grad{n_j}\)\nonumber\\
 &&-\sum_{i,j=1}^M\div\big(\(\div\(\u n_j\)\)c_{ij}\grad{n_i}\big)-\sum_{i,j=1}^M  \div\big(\(\div\J_j\)c_{ij}\grad{n_i}\big)\nonumber\\
 &&+\frac{1}{2}\u\cdot\sum_{i,j=1}^M\grad\( c_{ij} \nabla n_i\cdot\nabla n_j\).
\end{eqnarray}
Taking into account the identity 
\begin{eqnarray*}\label{eqMultiCompononentHelmholtzMultiTransport04}
\sum_{i,j=1}^M\(\grad{n_i}\)\div\(c_{ij}\grad{n_j}\) +\frac{1}{2}\sum_{i,j=1}^M \grad\(c_{ij} \nabla n_i\cdot\nabla n_j\) =\sum_{i,j=1}^M\div \(c_{ij}\nabla n_i\otimes\nabla n_j\),
  \end{eqnarray*}
 we can  reformulate \eqref{eqMultiCompononentHelmholtzMultiTransport05} as
  \begin{eqnarray}\label{eqMultiCompononentHelmholtzMultiTransport05b}
 \frac{\partial  f}{\partial t}+\div( f\u)
 &=&- p ~\div\u-\sum_{i=1}^M\mu_i\div\J_i+\u\cdot\(\sum_{i,j=1}^M\div c_{ij}\(\nabla n_i\otimes\nabla n_j\)\)\nonumber\\
 &&-\sum_{i,j=1}^M\div\big(\(\div\(\u n_j\)\)c_{ij}\grad{n_i}\big)-\sum_{i,j=1}^M  \div\big(\(\div\J_j\)c_{ij}\grad{n_i}\big)\nonumber\\
 &=&- p ~\div\u+\div\(\sum_{i,j=1}^Mc_{ij}\(\grad n_i\otimes\grad n_j\)\cdot\u\)-\sum_{i=1}^M\div\( \mu_i\J_i\)\nonumber\\
 &&-\sum_{i,j=1}^M\div\big(\(\div\(\u n_j\)\)c_{ij}\grad{n_i}\big)-\sum_{i,j=1}^M  \div\big(\(\div\J_j\)c_{ij}\grad{n_i}\big)\nonumber\\
 &&-\(\sum_{i,j=1}^Mc_{ij}\(\grad n_i\otimes\grad n_j\)\):\grad\u+\sum_{i=1}^M\J_i\cdot\grad \mu_i.
\end{eqnarray}

\subsection{Model equations}
Substituting \eqref{eqMultiCompononentHelmholtzMultiTransport05b} into \eqref{eqEntropyFirst}, we reformulate   the entropy equation as
\begin{eqnarray}\label{eqEntropySecond}
 T\(\frac{\partial s}{\partial t}+\div(\u s)\)&=&\frac{1}{2} \sum_{i=1}^MM_{w,i}\div\( (\u\cdot\u)\J_i\)-\div\(\sum_{i,j=1}^Mc_{ij}\(\grad n_i\otimes\grad n_j\)\cdot\u\)\nonumber\\
 &&+\sum_{i=1}^M\div\( \mu_i\J_i\)+\sum_{i,j=1}^M\div\big(\(\div\(\u n_j\)\)c_{ij}\grad{n_i}\big)\nonumber\\
 &&+\sum_{i,j=1}^M  \div\big(\(\div\J_j\)c_{ij}\grad{n_i}\big)-\sum_{i=1}^M\J_i\cdot\grad \mu_i
 - \bm\sigma^T:\grad\u\nonumber\\
 &&+\(p\I+\sum_{i,j=1}^Mc_{ij}\(\grad n_i\otimes\grad n_j\)\):\grad\u\nonumber\\
 &&
 -\u\cdot\(\rho\frac{d \u}{d t}+\sum_{i=1}^MM_{w,i}\J_i\cdot\grad\u+\div\bm\sigma\)  ,
\end{eqnarray}
where $\I$ is the second-order identity tensor. 
We consider the fluid mixture in a closed system with   the fixed total moles for each component.  Let $\Omega$ denote the domain  with the fixed volume.  The natural boundary conditions can be formulated as
\begin{eqnarray}\label{eqEntropySecondb}
 \u\cdot\bm\gamma_{\partial\Omega} =0,~~\J_i\cdot\bm\gamma_{\partial\Omega} =0,~~\grad{n}_i\cdot\bm\gamma_{\partial\Omega} =0.
\end{eqnarray}
where $\bm\gamma_{\partial\Omega}$ denotes a normal unit outward vector  to the boundary $\partial \Omega$. 
Integrating \eqref{eqEntropySecond} over the entire domain, we obtain the change of total entropy $S$ with time
\begin{eqnarray}\label{eqEntropySecondb}
 T\frac{\partial S}{\partial t}&=&-\int_\Omega\sum_{i=1}^M\J_i\cdot\grad \mu_id\x-\int_\Omega\(\bm\sigma^T-p\I-\sum_{i,j=1}^Mc_{ij}\(\grad n_i\otimes\grad n_j\)\):\grad\u d\x\nonumber\\
 &&
 -\int_\Omega\u\cdot\(\rho\frac{d \u}{d t}+\sum_{i=1}^MM_{w,i}\J_i\cdot\grad\u+\div\bm\sigma\) d\x .
\end{eqnarray}
where $\x\in\Omega$.
According to the second law of thermodynamics, the total entropy shall not decrease with time. Using this principle, we can determine the complete forms of multi-component two-phase flow model. 
First, we consider  an ideal reversible  process   to get the form of the reversible stress, and the reversibility implies that there exist no  effects of viscosity and friction.    In this case,  the  entropy shall be conserved, so the diffusions vanish, i.e. $\J_i=0$, and the total tress $\bm\sigma$ becomes equal to the reversible stress,   denoted by $\bm\sigma_{\textnormal{rev}}$, which must have the form
\begin{eqnarray}\label{eqMulticomponentTotalStressA}
 \bm\sigma_{\textnormal{rev}}=  p \I+\sum_{i,j=1}^Mc_{ij}\(\nabla n_i\otimes\nabla n_j\).
\end{eqnarray}
 The last term on the right-hand side of \eqref{eqEntropySecondb} shall also be zero as 
\begin{eqnarray}\label{eqGeneralNSEQA}
\rho\frac{d \u}{d t}+\div\bm\sigma_{\textnormal{rev}}=0.
\end{eqnarray}

For the realistic irreversible   multi-component chemical systems, the driving force for diffusion of each component is the gradient of chemical potentials,  so we express the diffusion flux  for each component   as   \cite{Cogswell2010thesis,kouandsun2016multiscale}
\begin{eqnarray}\label{eqMultiCompononentMassConserveDiffusion}
  \bm J_{i}=-\sum_{j=1}^M\mathcal{M}_{ij}\nabla \mu_{j},~~i=1,\cdots, M,
\end{eqnarray}
where   $\bm{\mathcal{M}}=\(\mathcal{M}_{ij}\)_{i,j=1}^M$ is the mobility. Onsager's reciprocal principle \cite{Groot2015NET}  requires the  symmetry  of $\bm{\mathcal{M}}$, and moreover, the second law of thermodynamics demands that $\bm{\mathcal{M}}$ is  positive semidefinite or positive definite, i.e.
\begin{eqnarray}\label{eqMultiCompononentMassConserveDiffusionA}
 \sum_{i=1}^M\J_i\cdot\grad \mu_i= -\sum_{i,j=1}^M\mathcal{M}_{ij}\nabla \mu_{i}\cdot\nabla \mu_{j}\leq0,
\end{eqnarray}
which ensures the non-negativity of the first term on the right-hand side of \eqref{eqEntropySecondb}.
 In  general principle, $\bm J_{i}$ may   depend on the chemical potential gradients of other components except for component $i$. In numerical tests of this paper, we take a special case of $\bm{\mathcal{M}}$ that is a diagonal positive definite matrix; indeed, we use the following diffusion flux  \cite{Cogswell2010thesis,kousun2016Flash}
\begin{eqnarray}\label{eqMultiCompononentMassConserveDiffusionB}
\J_i  = -\frac{D_in_i}{RT}\grad\mu_i,
\end{eqnarray}
where $D_i>0$ is the diffusion coefficient of component $i$. 

For the realistic viscous  flow, the total stress can be split into two parts:  reversible part (i.e. $\bm\sigma_{\textnormal{rev}}$ given in \eqref{eqMulticomponentTotalStressA}) and irreversible part  (denoted by $\bm\sigma_{\textnormal{irrev}}$); that is,
\begin{eqnarray}\label{eqTotalStress}
   \bm\sigma=\bm\sigma_{\textnormal{rev}}+\bm\sigma_{\textnormal{irrev}}.
\end{eqnarray}
Newtonian fluid theory suggests
\begin{eqnarray}\label{eqTotalStressB}
   \bm\sigma_{\textnormal{irrev}}=-\eta\(\nabla\u+\nabla\u^T\)-\(\(\xi-\frac{2}{3}\eta\)\div\u\) \I,
\end{eqnarray}
where  $\xi$ is the volumetric  viscosity and   $\eta$ is the shear viscosity. We assume that $\eta>0$ and $\xi>\frac{2}{3}\eta$.
So the second term on the right-hand side of \eqref{eqEntropySecondb} is   non-negative.

The non-negativity of the last term on the right-hand side of \eqref{eqEntropySecondb} requires that
\begin{eqnarray}\label{eqGeneralNSEQA}
\rho\frac{d \u}{d t}+\sum_{i=1}^MM_{w,i}\J_i\cdot\grad\u+\div\bm\sigma=0.
\end{eqnarray}

Substituting \eqref{eqMulticomponentTotalStressA}, \eqref{eqTotalStress} and \eqref{eqTotalStressB} into \eqref{eqGeneralNSEQA},  we obtain the complete momentum balance equation as
\begin{eqnarray}\label{eqGeneralNSEQ}
&&\rho\(\frac{\partial  \u}{\partial t}+ \u\cdot\grad{ \u}\)+\sum_{i=1}^MM_{w,i}\J_i\cdot\grad\u=\nabla\(\(\xi-\frac{2}{3}\eta\)\div\u-p\)\nonumber\\
&&~~+\div\eta\(\nabla\u+\nabla\u^T\)
-\sum_{i,j=1}^M\div\(c_{ij}\nabla n_i\otimes\nabla n_j\),
\end{eqnarray}
which is also equivalent to 
\begin{eqnarray}\label{eqGeneralNSEQMC}
&&\frac{\partial \(\rho \u\)}{\partial t}+ \div{ \(\rho\u\otimes\u\)}+\sum_{i=1}^MM_{w,i}\div{ \(\u\otimes\J_i\)}=\nabla\(\(\xi-\frac{2}{3}\eta\)\div\u-p\)\nonumber\\
&&~~+\div\eta\(\nabla\u+\nabla\u^T\)
-\sum_{i,j=1}^M\div\(c_{ij}\nabla n_i\otimes\nabla n_j\).
\end{eqnarray}
If  $\J_i$ is taken such that $\sum_{i=1}^MM_{w,i}\J_i=0$, then \eqref{eqGeneralNSEQMC} is reduced into the form in \cite{kouandsun2016multiscale}, but for the general formulations of $\J_i$, the term $\sum_{i=1}^MM_{w,i}\J_i\cdot\grad\u$ or $\sum_{i=1}^MM_{w,i}\div{ \(\u\otimes\J_i\)}$ is essential to ensure the thermodynamical consistency   \cite{Abels2012TwoPhaseModel}.
 
In summary, the proposed mathematical model  of multi-component two-phase flow is formulated by a nonlinear and fully coupled  system of equations including the mass conservation equation \eqref{eqGeneralNSEQMass} coupling with the diffusion flux  \eqref{eqMultiCompononentMassConserveDiffusion} and the chemical potential \eqref{eqgeneralchemicalpotential},  and the momentum balance equation \eqref{eqGeneralNSEQ} or \eqref{eqGeneralNSEQMC} coupling with the pressure formulation \eqref{eqMultiComponentDefPresGeneralB}.

\section{Energy dissipation}

In this section, we will demonstrate the total (free) energy dissipation property of the proposed model, which seems not an obvious  conclusion from its derivation, but it is required for thermodynamical consistency and it is essential to design the stable and efficient numerical methods in the next section. To do this, we first reformulate the  momentum conservation equation by a simpler form. 
This requires the following theorem regarding the relations between  the gradients of the pressure and chemical potentials.
 \begin{thm}\label{lemPresChptlEqnMultiGrad}
The gradients of  the pressure and chemical potentials  have the following relation 
\begin{eqnarray}\label{eqPresChptlMultiGrad01}
 \sum_{i=1}^Mn_i\grad \mu_i 
 =\grad p+\sum_{i,j=1}^M\div\(c_{ij}\grad n_i\otimes\grad n_j\).
\end{eqnarray}
\end{thm}
\begin{proof}
We recall the relation $\grad p_b=\sum_{i=1}^Mn_i\grad\mu^b_i$ proved in \eqref{eqMultiCompononentMultiPresChptl01b}, and then 
taking into account    chemical potential formulation \eqref{eqgeneralchemicalpotential} and pressure formulation \eqref{eqMultiComponentDefPresGeneralB}, we obtain
\begin{eqnarray*}\label{eqPresChptlMultiGradproof01}
\sum_{i=1}^Mn_i\grad \mu_i-\grad p 
 &=&\sum_{i=1}^Mn_i\grad\(\mu_i^b-\sum_{j=1}^M\div\(c_{ij}\grad{n_j}\)\)\nonumber\\
 &&-\grad\(p_b- \sum_{i,j=1}^Mn_i\div\(c_{ij}\grad{n_j}\)-\frac{1}{2} \sum_{i,j=1}^Mc_{ij} \nabla n_i\cdot\nabla n_j\)\nonumber\\
 &=&\sum_{i=1}^Mn_i\grad\mu_i^b-\grad p_b-\sum_{i=1}^Mn_i\grad\(\sum_{j=1}^M\div\(c_{ij}\grad{n_j}\)\)\nonumber\\
 &&+\sum_{i,j=1}^M\grad\big(n_i\div\(c_{ij}\grad{n_j}\)\big)+\frac{1}{2} \sum_{i,j=1}^M\grad\(c_{ij} \nabla n_i\cdot\nabla n_j\)\nonumber\\
 &=&\sum_{i,j=1}^M\(\div\(c_{ij}\grad{n_j}\)\)\grad n_i+\frac{1}{2} \sum_{i,j=1}^M\grad\(c_{ij} \nabla n_i\cdot\nabla n_j\)\nonumber\\
 &=&\sum_{i,j=1}^M\div\(c_{ij}\grad n_i\otimes\grad n_j\).
\end{eqnarray*}
Thus,  \eqref{eqPresChptlMultiGrad01} is obtained. 
\end{proof}

As a direct application  of \eqref{eqPresChptlMultiGrad01}, the momentum conservation equation \eqref{eqGeneralNSEQ} is  reformulated as
\begin{eqnarray}\label{eqGeneralNSEQb}
&&\rho\(\frac{\partial  \u}{\partial t}+ \u\cdot\grad{ \u}\)+\sum_{i=1}^MM_{w,i}\J_i\cdot\grad\u=-\sum_{i=1}^Mn_i\grad \mu_i\nonumber\\
   &&~~+\div\(\eta\(\nabla\u+\nabla\u^T\)+\(\xi-\frac{2}{3}\eta\)\(\div\u\) \I\) ,
\end{eqnarray}
which indicates that the  fluid motion is indeed driven by the gradients of chemical potentials.

Applying the formulation of momentum conservation equation given in \eqref{eqGeneralNSEQb}, we can conveniently  prove the total (free) energy dissipation property of the proposed model.  We use $\(\cdot,\cdot\)$ and $\|\cdot\|$ to represent the $L^2\(\Omega\)$, $\(L^2\(\Omega\)\)^d$ or $\(L^2\(\Omega\)\)^{d\times d}$ inner product and norm respectively. 
We define  the Helmholtz free energy and  kinetic energy within the entire domain $\Omega$ as
\begin{eqnarray}\label{eqPDETotalEnergy}
F=\int_{\Omega} fd\x,~~~~~E=\frac{1}{2}\int_{\Omega}\rho|\u|^2d\x.
\end{eqnarray}
\begin{thm}
The sum  of    the Helmholtz free energy and  kinetic energy   is dissipated with time, i.e.
\begin{eqnarray}\label{eqPDETotalEnergyDecay}
\frac{\partial (F+E)}{\partial t} \leq 0.
\end{eqnarray}
\end{thm}
\begin{proof}
Taking into account the mass balance equation \eqref{eqMultiCompononentMassConserveMass},
we obtain
\begin{eqnarray}\label{eqMultiComponentMomentumConserveProof01}
\frac{\partial E}{\partial t}&=&\(\rho\frac{\partial \u}{\partial t},\u\)+\frac{1}{2}\(\frac{\partial \rho}{\partial t},|\u|^2\)\nonumber\\
&=&\(\rho\frac{\partial \u}{\partial t},\u\)-\frac{1}{2}\(\div(\rho\u)+\sum_{i=1}^MM_{w,i}\div\J_i,|\u|^2\)\nonumber\\
&=&\(\rho\frac{\partial \u}{\partial t}+\rho\u\cdot\grad\u+\sum_{i=1}^MM_{w,i}\J_i\cdot\grad\u,\u\)\nonumber\\
&=&-\sum_{i=1}^M\(n_i\grad \mu_i,\u\)+\(\div\(\eta\(\nabla\u+\nabla\u^T\)+\(\xi-\frac{2}{3}\eta\)\(\div\u\) \I\), \u\)\nonumber\\
&=&-\sum_{i=1}^M\(n_i\grad \mu_i,\u\) -\(\eta\(\nabla\u+\nabla\u^T\)+\(\xi-\frac{2}{3}\eta\)\(\div\u\) \I, \grad\u\)\nonumber\\
&=&-\sum_{i=1}^M\(n_i\grad \mu_i,\u\) -\(\eta\(\nabla\u+\nabla\u^T\), \grad\u\)-\(\(\xi-\frac{2}{3}\eta\)\div\u, \div\u\)\nonumber\\
&=&-\sum_{i=1}^M\(n_i\grad \mu_i,\u\) -\frac{1}{2}\left\|\eta^{1/2}\(\nabla\u+\nabla\u^T\)\right\|^2-\left\|\(\xi-\frac{2}{3}\eta\)^{1/2}\div\u\right\|^2.
\end{eqnarray}
Multiplying both sides of \eqref{eqGeneralNSEQMass} by $\mu_i$ and integrating it over $\Omega$,  we obtain
\begin{eqnarray}\label{eqMultiCompononentMassConserveProofC01}
\(\frac{\partial n_i}{\partial t},\mu_i\)+\(\div(n_i\u),\mu_i\)=\(\J_i,\grad\mu_i\).
\end{eqnarray}
Applying the formulation  of $\mu_i$, we derive the summation of the first term on the left-hand side of \eqref{eqMultiCompononentMassConserveProofC01} as
\begin{eqnarray}\label{eqMultiCompononentMassConserveProofC02}
\sum_{i=1}^M\(\frac{\partial n_i}{\partial t},\mu_i\)&=&\sum_{i=1}^M\(\frac{\partial n_i}{\partial t},\mu_i^b-\sum_{j=1}^M\div\(c_{ij}\grad{n_j}\)\)\nonumber\\
&=&\(\frac{\partial f_b}{\partial t},1\)-\sum_{i,j=1}^M\(\frac{\partial n_i}{\partial t},\div\(c_{ij}\grad{n_j}\)\)\nonumber\\
&=&\(\frac{\partial f_b}{\partial t},1\)+\sum_{i,j=1}^M\(\grad\frac{\partial n_i}{\partial t},c_{ij}\grad{n_j}\)\nonumber\\
&=&\(\frac{\partial f_b}{\partial t},1\)+\frac{1}{2}\frac{\partial}{\partial t}\sum_{i,j=1}^M\(c_{ij}\grad n_i,\grad{n_j}\)\nonumber\\
&=&\frac{\partial F}{\partial t}.
\end{eqnarray}
We denote $\bm\mu=[\mu_1,\cdots,\mu_M]^T$, and further define the norm $$\big\|\grad\bm\mu\big\|_{M}^2=\sum_{i,j=1}^M\big(\mathcal{M}_{ij}\nabla \mu_{i},\grad\mu_j\big).$$ Using \eqref{eqMultiCompononentMassConserveProofC02} and \eqref{eqMultiCompononentMassConserveDiffusion}, we have the summation of \eqref{eqMultiCompononentMassConserveProofC01}  from $i=1$ to $M$  as
\begin{eqnarray}\label{eqMultiCompononentMassConserveProofC03}
\frac{\partial F}{\partial t}=-\sum_{i=1}^M\(\div(n_i\u),\mu_i\)-\big\|\grad\bm\mu\big\|_{M}^2.
\end{eqnarray}
We combine  \eqref{eqMultiComponentMomentumConserveProof01} and \eqref{eqMultiCompononentMassConserveProofC03}
\begin{eqnarray}\label{eqMultiComponentMomentumConserveMassConserv01}
\frac{\partial \(F+E\)}{\partial t}&=& -\big\|\grad\bm\mu\big\|_{M}^2 -\frac{1}{2}\left\|\eta^{1/2}\(\nabla\u+\nabla\u^T\)\right\|^2 -\left\|\(\xi-\frac{2}{3}\eta\)^{1/2}\div\u\right\|^2.
\end{eqnarray}
where we have also used $$\(n_i\grad \mu_i,\u\)+\(\div(n_i\u),\mu_i\)=\(\div(n_i\mu_i\u),1\)=0.$$
 Therefore,  the energy dissipation \eqref{eqPDETotalEnergyDecay} can be obtained from \eqref{eqMultiComponentMomentumConserveMassConserv01} and \eqref{eqMultiCompononentMassConserveDiffusionA}. 
\end{proof}

\section{Energy-stable numerical method}
In this section,  we will design an energy-dissipated semi-implicit time marching scheme for simulating  the proposed multi-component fluid model. The key difficulties  result from the complication of Helmholtz free energy density and fully coupling relations between molar densities and velocity.  We resolve these problems  through very careful physical observations with a mathematical rigor.

The gradient contribution to the Helmholtz free energy density is always convex with respect to  molar densities, so it shall be  treated implicitly.  The challenge   is to deal with the bulk Helmholtz free energy density $f_b$.  For the pure substance,  the ideal and repulsion terms, i.e. $f_b^{\textnormal{ideal}}$ and $f_b^{\textnormal{repulsion}}$,   result in convex contributions to the  Helmholtz free energy density \cite{qiaosun2014}.   For the multi-component mixtures ($M\geq2$), we can prove the ideal contribution $f_b^{\textnormal{ideal}}$ is still convex with respect to  molar densities since its Hessian matrix is a diagonal positive definite matrix.  However, the repulsion term $f_b^{\textnormal{repulsion}}$ for multi-component mixtures is never convex due to cross products of multiple components, and it can be easily  verified for the binary mixtures through  eigenvalue calculations (we omit  the details here).  
For the pure substance, 
the attraction term $f_b^{\textnormal{attraction}}$ is proved to result in a concave contribution to the Helmholtz free energy density  \cite{qiaosun2014}. But it may be not true for multi-component mixture; in fact,  numerical tests show that the maximum  eigenvalues of its Hessian matrix may be slightly larger than zero in some cases (not presented here).  

In order to construct a strict convex-concave splitting for the case of multi-component mixture, we impose two auxiliary terms   on the bulk Helmholtz free energy density. One term is the additional  ideal term for the reason that the ideal term is a good approximation of the behavior of many gases and always convex. The other is a separate repulsion term,  denoted by $ f_b^{\textnormal{SR}}$,  which reduces the cross products of multiple components in the original repulsion term and is expressed as
 \begin{eqnarray}\label{eqHelmholtzEnergy_a0_02R}
    f_b^{\textnormal{SR}}(\n)=-RT\sum_{i=1}^Mn_i\ln\(1-b_in_i\).
\end{eqnarray}
We note that $bn=\sum_{i=1}^Mb_in_i<1$ in terms of its physical meaning, so $f_b^{\textnormal{SR}}$ is well defined. We define the summation of the above two auxiliary terms as
 \begin{eqnarray}\label{eqHelmholtzEnergy_a0_02R}
 f_b^{\textnormal{auxiliary}}(\n)&=&f_b^{\textnormal{ideal}}(\n)+   f_b^{\textnormal{SR}}(\n)\nonumber\\
 &=&RT\sum_{i=1}^{M}n_i\(\ln n_i-1\)-RT\sum_{i=1}^Mn_i\ln\(1-b_in_i\),
\end{eqnarray}
which has a diagonal positive definite Hessian matrix
 \begin{eqnarray}\label{eqHessianMat01}
\frac{\partial^2f_b^{\textnormal{auxiliary}}(\n) }{\partial n_i\partial n_i} =\frac{RT}{n_i} +\frac{RTb_i}{1-b_in_i}+\frac{RTb_i}{\(1-b_in_i\)^2} ,~~\frac{\partial^2f_b^{\textnormal{auxiliary}} }{\partial n_i\partial n_j} =0, ~i\neq j.
\end{eqnarray}

We now state a convex-concave splitting of the bulk Helmholtz free energy density based on the above auxiliary terms. Let us define a parameter $\lambda>0$, and then we reformulate the bulk Helmholtz free energy density   as
\begin{eqnarray}\label{eqConvexConcaveHelmholtzEnergy}
    f_b(\n)= f_b^{\textnormal{convex}}(\n) +f_b^{\textnormal{concave}}(\n),
\end{eqnarray}
where
\begin{eqnarray}\label{eqConvexHelmholtzEnergy}
    f_b^{\textnormal{convex}}(\n)= f_b^{\textnormal{ideal}}(\n) +f_b^{\textnormal{repulsion}}(\n)+\lambda f_b^{\textnormal{auxiliary}}(\n),
\end{eqnarray}
\begin{eqnarray}\label{eqConcaveHelmholtzEnergy}
    f_b^{\textnormal{concave}}(\n)= f_b^{\textnormal{attraction}}(\n)-\lambda f_b^{\textnormal{auxiliary}}(\n) .
\end{eqnarray}
The chemical potentials can be reformulated accordingly
\begin{eqnarray}\label{eqConvexChPtl}
    \mu_i^b(\n)= \mu_i^{b,\textnormal{convex}}(\n) +\mu_i^{b,\textnormal{concave}}(\n).
\end{eqnarray}
We make some remarks on the convex-concave splitting given in \eqref{eqConvexConcaveHelmholtzEnergy}-\eqref{eqConcaveHelmholtzEnergy}:\\
(a)  The reason that $f_b^{\textnormal{repulsion}}$ is considered in the convex part is an observation in numerical tests  that  positive eigenvalues of $f_b^{\textnormal{repulsion}}$ are usually  large, while  its negative eigenvalues are only slightly less than zero.  In contrast, $f_b^{\textnormal{attraction}}$ is included in the concave part because numerical tests show that its negative  eigenvalues are usually  far from zero, while  its positive  eigenvalues are only slightly larger than zero. Moreover, the attraction force shall result in a concave contribution from  the point view of physics \cite{qiaosun2014}. \\
(b) Because of strong convexity of $f_b^{\textnormal{auxiliary}}$,  if we choose a sufficiently large $\lambda$, the strict convex-concave splitting can be achieved for the   bulk Helmholtz free energy density of  multi-component mixture. But a very large value for $\lambda$ is not necessary in practical computations,   and in fact, we just take $\lambda=1$ in our numerical tests, which is enough  to gain the numerical convex-concave splitting.  Consequently,  we  assume  that there exists a suitable $\lambda>0$ such that the convexity of $f_b^{\textnormal{convex}}(\n)$  and concavity of $f_b^{\textnormal{concave}}(\n)$ hold.

We consider a  time interval $\mathcal{I}=(0,T_{f}]$, where $T_{f}>0$.  We divide $\mathcal{I}$ into   $\mathcal{K}$ subintervals $\mathcal{I}_{k}=(t_{k},t_{k+1}]$, where $t_{0}=0$ and $t_{M}=T_{f}$. Furthermore, we denote $\delta t_{k}=t_{k+1}-t_{k}$.  For any scalar $v(t)$ or vector $\v(t)$, we denote by $v^{k}$ or $\v^{k}$ its approximation at the time $t_{k}$.
We now state the semi-implicit time marching scheme.  
First, a semi-implicit time   scheme is constructed for the molar density balance equation \eqref{eqGeneralNSEQMass}  as
\begin{eqnarray}\label{eqDiscreteMultiCompononentMassConserve01}
\frac{ n_i^{k+1}-n_i^{k}}{\delta t_{k}}+\div(n_i^{k+1}\u^{k+1})+\div\J_i^{k+1}=0,
\end{eqnarray}
where
\begin{eqnarray}\label{eqDiscreteMultiCompononentMassConserve02}
\J_i^{k+1} = -\sum_{j=1}^M\mathcal{M}^k_{ij}\nabla \mu_{j}^{k+1},
\end{eqnarray}
 \begin{eqnarray}\label{eqDiscreteMultiCompononentMassConserve03}
\mu_i^{k+1}=\mu_i^{b,\textnormal{convex}}\(\n^{k+1}\)+\mu_i^{b,\textnormal{concave}}\(\n^{k}\)-\sum_{j=1}^M\div\(c_{ij}\grad{n_j^{k+1}}\).
  \end{eqnarray}
  We note that the mobility coefficients  $\mathcal{M}_{ij}$ generally depend on the molar densities, so we use $\mathcal{M}^k_{ij}$ to denote their values computed from molar densities $\n^k$.
  
  In the previous section, we have reformulated the momentum balance equation by the form \eqref{eqGeneralNSEQb} coupling with the chemical potentials rather than pressure.
A semi-implicit time   scheme  for  the momentum conservation equation \eqref{eqGeneralNSEQb} is constructed  as
\begin{eqnarray}\label{eqDiscreteMultiComponentMomentumConserveC01}
&& \rho^{k}\frac{\u^{k+1}-\u^{k}}{\delta t_{k}}+\rho^{k+1}\(\u^{k+1}\cdot\grad\)\u^{k+1}+\sum_{i=1}^MM_{w,i}\J_i^{k+1}\cdot\grad\u^{k+1}=-\sum_{i=1}^Mn_i^{k+1}\grad \mu_i^{k+1}\nonumber\\
   &&~~+\div\(\eta\(\nabla\u^{k+1}+\(\nabla\u^{k+1}\)^T\)+\(\xi-\frac{2}{3}\eta\)\(\div\u^{k+1}\) \I\) .
\end{eqnarray}
Here, the viscosities $\eta$ and $\xi$ are assumed to be constant.

We now demonstrate that the proposed semi-implicit time scheme satisfies the dissipation property of the total (free) energy. We define  the discrete Helmholtz free energy and   discrete kinetic energy at the time $t_k$ as
\begin{eqnarray}\label{eqPDETotalEnergy}
F^k=\int_{\Omega} f\(\n^k\)d\x,~~~~~E^k=\frac{1}{2}\int_{\Omega}\rho^{k}|\u^{k}|^2d\x.
\end{eqnarray}

\begin{thm}
The sum of    the discrete Helmholtz free energy and  discrete kinetic energy   is dissipated with time steps, i.e.
\begin{eqnarray}\label{eqDiscreteTotalEnergyDecay}
E^{k+1}+F^{k+1} \leq E^{k}+F^{k} .
\end{eqnarray}
\end{thm}
\begin{proof}
Multiplying both sides of \eqref{eqDiscreteMultiCompononentMassConserve01} by $\mu_i^{k+1}$ and integrating it over $\Omega$,   we obtain
\begin{eqnarray}\label{eqDiscreteMultiCompononentMassConserveProof01}
\(\frac{ n_i^{k+1}-n_i^{k}}{\delta t_{k}},\mu_i^{k+1}\)+\(\div(n_i^{k+1}\u^{k+1}),\mu_i^{k+1}\)+\(\div\J_i^{k+1},\mu_i^{k+1}\)=0.
\end{eqnarray}
We define the bulk and gradient contributions of the discrete Helmholtz free energy  as
\begin{eqnarray}\label{eqCHEConvexSplittingPoreEnergyDecayProof01}
F_b\(\n^{k}\)=\int_{\Omega}f_b(\n^k)d\x,~~~F_\grad\(\n^{k}\)=\int_{\Omega}f_\grad(\n^k)d\x.
\end{eqnarray}
The convexity and concavity  yields 
\begin{eqnarray}\label{eqCHEConvexSplittingPoreEnergyDecayProof03}
F_b\(\n^{k+1}\)-F_b\(\n^{k}\)\leq\sum_{i=1}^M\(n_i^{k+1}-n_i^k,\mu_i^{b,\textnormal{convex}}(\n^{k+1}) +\mu_i^{b,\textnormal{concave}}(\n^k)\).
\end{eqnarray}
Using the symmetry of $c_{ij}$, we can estimate the gradient contribution of Helmholtz  free energy  as
\begin{eqnarray}\label{eqDiscreteMultiCompononentMassConserveProof02}
-\sum_{i,j=1}^M\(n_i^{k+1}-n_i^{k},\div\(c_{ij}\grad{n_j^{k+1}}\)\)&=&\sum_{i,j=1}^M\(c_{ij}\grad\(n_i^{k+1}-n_i^{k}\),\grad{n_j^{k+1}}\)\nonumber\\
&\geq&F_\grad\(\n^{k+1}\)-F_\grad\(\n^{k}\).
\end{eqnarray}
Since $F^{k}=F_b\(\n^{k}\)+F_\grad\(\n^{k}\)$, we obtain from \eqref{eqDiscreteMultiCompononentMassConserveProof01}, \eqref{eqCHEConvexSplittingPoreEnergyDecayProof03} and \eqref{eqDiscreteMultiCompononentMassConserveProof02}  that
\begin{eqnarray}\label{eqDiscreteMultiCompononentMassConserveProof03a}
F^{k+1}-F^{k}\leq-\delta t_{k}\sum_{i=1}^M\(\(\div(n_i^{k+1}\u^{k+1}),\mu_i^{k+1}\)-\(\J_i^{k+1},\grad\mu_i^{k+1}\)\).
\end{eqnarray}
We denote $\bm\mu^{k+1}=[\mu_1^{k+1},\cdots,\mu_M^{k+1}]^T$, and  define the  norm of $\bm\mu^{k+1}$ as $$\big\|\grad\bm\mu^{k+1}\big\|_{M^k}^2=\sum_{i,j=1}^{M}\big(\mathcal{M}^{k}_{ij}\nabla \mu_{i}^{k+1},\grad\mu_j^{k+1}\big).$$
Substituting \eqref{eqDiscreteMultiCompononentMassConserve02} into \eqref{eqDiscreteMultiCompononentMassConserveProof03a} yields 
\begin{eqnarray}\label{eqDiscreteMultiCompononentMassConserveProof03}
F^{k+1}-F^{k}\leq-\delta t_{k}\sum_{i=1}^M\(\div(n_i^{k+1}\u^{k+1}),\mu_i^{k+1}\)-\delta t_{k}\big\|\grad\bm\mu^{k+1}\big\|_{M^k}^2.
\end{eqnarray}
Multiplying both sides of \eqref{eqDiscreteMultiComponentMomentumConserveC01} by $\u^{k+1}$ and  and integrating it over $\Omega$,   we obtain
\begin{eqnarray}\label{eqDiscreteMultiComponentMomentumConserveProof01}
&& \(\rho^{k}\frac{\u^{k+1}-\u^{k}}{\delta t^k},\u^{k+1}\)+\(\rho^{k+1}\u^{k+1}\cdot\grad\u^{k+1},\u^{k+1}\)\nonumber\\
&&~~+\sum_{i=1}^MM_{w,i}\(\J_i^{k+1}\cdot\grad\u^{k+1},\u^{k+1}\)=-\sum_{i=1}^M\(n_i^{k+1}\grad \mu_i^{k+1},\u^{k+1}\)\nonumber\\
   &&~~- \left\|\(\xi-\frac{2}{3}\eta\)^{1/2}\div\u^{k+1}\right\|^2 -\frac{1}{2} \left\|\eta^{1/2}\(\nabla\u^{k+1}+\(\nabla\u^{k+1}\)^T\)\right\|^2.
\end{eqnarray}
Multiplying \eqref{eqDiscreteMultiCompononentMassConserve01} by $M_{w,i}$ and summing it from $i=1$ to $M$,  we obtain the mass balance equation
\begin{eqnarray}\label{eqDiscreteMultiCompononentMassConserveProof04}
\frac{ \rho^{k+1}-\rho^{k}}{\delta t_{k}}+\div(\rho^{k+1}\u^{k+1})+\sum_{i=1}^MM_{w,i}\div\J_i^{k+1}=0.
\end{eqnarray}
Using \eqref{eqDiscreteMultiCompononentMassConserveProof04}, we estimate
 \begin{eqnarray}\label{eqDiscreteMultiComponentMomentumConserveProof02}
\(\rho^{k}\(\u^{k+1}-\u^{k}\),\u^{k+1}\)&=&\frac{1}{2}\(\rho^{k}\(|\u^{k+1}|^2-|\u^{k}|^2+|\u^{k+1}-\u^{k}|^2\),1\)\nonumber\\
   &\geq&E^{k+1}-E^{k}-\frac{1}{2}\(\rho^{k+1}-\rho^{k},|\u^{k+1}|^2\)\nonumber\\
   &=&E^{k+1}-E^{k}+\frac{\delta t_{k}}{2}\(\div(\rho^{k+1}\u^{k+1}),|\u^{k+1}|^2\)\nonumber\\
   &&+\frac{\delta t_{k}}{2}\sum_{i=1}^MM_{w,i}\(\div\J_i^{k+1},|\u^{k+1}|^2\)\nonumber\\
   &=&E^{k+1}-E^{k}-\delta t_{k}\(\u^{k+1}\cdot\grad\u^{k+1},\rho^{k+1}\u^{k+1}\)\nonumber\\
   &&- \delta t_{k} \sum_{i=1}^MM_{w,i}\( \J_i^{k+1}\cdot\grad\u^{k+1},\u^{k+1}\).
\end{eqnarray}
Substituting \eqref{eqDiscreteMultiComponentMomentumConserveProof02} into \eqref{eqDiscreteMultiComponentMomentumConserveProof01} yields
\begin{eqnarray}\label{eqDiscreteMultiComponentMomentumConserveProof03}
E^{k+1}-E^{k}&\leq&-\delta t_{k}\sum_{i=1}^M\(n_i^{k+1}\grad \mu_i^{k+1},\u^{k+1}\) -\delta t_{k}\left\|\(\xi-\frac{2}{3}\eta\)^{1/2}\div\u^{k+1}\right\|^2\nonumber\\
   &&-\frac{1}{2}\delta t_{k}\left\|\eta^{1/2}\(\nabla\u^{k+1}+\(\nabla\u^{k+1}\)^T\)\right\|^2.
\end{eqnarray}
 Noticing that $$\(n_i^{k+1}\grad \mu_i^{k+1},\u^{k+1}\)+\(\div(n_i^{k+1}\u^{k+1}),\mu_i^{k+1}\)=0,$$
 we combine   \eqref{eqDiscreteMultiCompononentMassConserveProof03} and \eqref{eqDiscreteMultiComponentMomentumConserveProof03} and then obtain
 \begin{eqnarray}\label{eqDiscreteMultiCompononentEnergyDecay}
E^{k+1}-E^{k}+F^{k+1}-F^{k}&\leq&-\delta t_{k}\big\|\grad\bm\mu^{k+1}\big\|_{M^k}^2-\delta t_{k}\left\|\(\xi-\frac{2}{3}\eta\)^{1/2}\div\u^{k+1}\right\|^2\nonumber\\
   &&-\frac{1}{2}\delta t_{k}\left\|\eta^{1/2}\(\nabla\u^{k+1}+\(\nabla\u^{k+1}\)^T\)\right\|^2,
   \end{eqnarray}
which yields  the energy dissipation \eqref{eqDiscreteTotalEnergyDecay}. 
\end{proof}

The above semi-implicit time   schemes lead to a nonlinear  systems of equations, and consequently, a linearized  iterative method is required. The Newton's method is  commonly used to solve such systems. Here, we present a mixed  iterative method, in which we apply Newton's linearization for $\mu_i^{b,\textnormal{convex}}$ and treat the  rest terms through a simple linearization approach. We use the superscript $l$ to represent the iteration step; that is,  $\n^{k+1,l} $ and $\u^{k+1,l}$ denote the approximations of $\n^{k+1} $ and $\u^{k+1}$ at the $l$th iteration.  The initial guess solutions are taken from the previous time step, i.e. $\n^{k+1,0} =\n^{k}$ and $\u^{k+1,0}=\u^{k}$.  Once the approximations $\n^{k+1,l} $ and $\u^{k+1,l}$ are obtained, the new approximations $\n^{k+1,l+1} $ and $\u^{k+1,l+1}$ are calculated as
\begin{eqnarray}\label{eqIterativeDiscreteMultiCompononentMassConserve01}
\frac{ n_i^{k+1,l+1}-n_i^{k}}{\delta t_{k}}+\div(n_i^{k,l+1}\u^{k+1,l})+\div\J_i^{k+1,l+1}=0,
\end{eqnarray}
\begin{eqnarray}\label{eqIterativeDiscreteMultiCompononentMassConserve02}
\J_i^{k+1,l+1} = -\sum_{j=1}^M\mathcal{M}^k_{ij}\nabla \mu_{j}^{k+1,l+1},
\end{eqnarray}
 \begin{eqnarray}\label{eqIterativeDiscreteMultiCompononentMassConserve03}
\mu_i^{k+1,l+1}&=&\mu_i^{b,\textnormal{convex}}\(\n^{k+1,l}\) + \mathbf{ H}_b^{\textnormal{convex}}\(\n^{k+1,l}\)\(\n^{k+1,l+1}-\n^{k+1,l}\)\nonumber\\
   &&~~+\mu_i^{b,\textnormal{concave}}\(\n^{k}\)-\sum_{j=1}^M\div\(c_{ij}\grad{n_j^{k+1,l+1}}\),
  \end{eqnarray}
  \begin{eqnarray}\label{eqIterativeDiscreteMultiComponentMomentumConserveC01}
&& \rho^{k}\frac{\u^{k+1,l+1}-\u^{k}}{\delta t_{k}}+\rho^{k+1,l}\(\u^{k+1,l}\cdot\grad\)\u^{k+1,l+1}\nonumber\\
   &&~~+\sum_{i=1}^MM_{w,i}\J_i^{k+1,l+1}\cdot\grad\u^{k+1,l}=-\sum_{i=1}^Mn_i^{k+1,l}\grad \mu_i^{k+1,l+1}\nonumber\\
   &&~~+\div\(\eta\(\nabla\u^{k+1,l+1}+\(\nabla\u^{k+1,l+1}\)^T\)+\(\xi-\frac{2}{3}\eta\)\(\div\u^{k+1,l+1}\) \I\) ,
\end{eqnarray}
where $\mathbf{ H}_b^{\textnormal{convex}}$ stands for the Hessian matrix of   $f_b^{\textnormal{convex}}$. The convergence of this iterative method will be studied  in the future work.

\section{Numerical tests}

In this section, we carry out  numerical tests regarding multi-component two-phase  fluid flow problems and verify the effectiveness of the proposed method. We consider the fluid flow of a binary hydrocarbon  mixture, which is composed of methane (CH$_4$) and decane (nC$_{10}$). We take   the volumetric  viscosity and    the shear viscosity as $\xi=\eta=0.01$Pa$\cdot$s, and we take  the diffusion coefficients in \eqref{eqMultiCompononentMassConserveDiffusionB} as $D_1=D_2=10^{-6}$m$^2$/s.  In all numerical examples, the spatial  domain is a square  domain with the length $20$nm, and a  uniform  rectangular mesh with $40\times40$ elements  is applied on the   domain. We employ the cell-centered finite difference method and the upwind scheme  to discretize  the mass balance equation and  the finite volume method  on the staggered mesh \cite{Tryggvason2011book} for the momentum balance equation, and these spatial discretization schemes  can be equivalent to special mixed finite element methods with quadrature rules \cite{arbogast1997mixed,Girault1996Mac}. The stop criterion of the nonlinear solver is that the 2-norm of  the relative variation of molar density and velocity between the current and previous  iterations  is less than $10^{-3}$, and    in many cases, about 2-3 nonlinear iterations are required to reach this criterion. The maximum nonlinear iterations are also set to be not larger than 5 for preventing too many  loops.

\subsection{Example 1:  a square shape droplet}
In this example, we simulate the dynamical evolution of a square shape fluid droplet, which is initially located in   the center of the domain. The  temperature of the fluid keeps 320K constant. The initial densities of liquid and gas fluids are computed by  PR-EOS under the pressure 160bar and temperature 320K.  The initial molar densities of methane and decane in gas phase  are $7.1339\times10^3$mol/m$^3$ and $0.0265\times10^3$mol/m$^3$ respectively, while the initial molar densities of methane and decane in  liquid phase are $3.5132\times10^3$mol/m$^3$ and $ 3.8146\times10^3$mol/m$^3$  respectively. The time step size is taken as $10^{-6}$s, and 45  time steps are simulated.

Figure \ref{SquareCH4andnC10TotalEnergyTemperature320}(a) clearly depicts the strict dissipation of the total (free) energy with time steps.   Figure \ref{SquareCH4andnC10TotalEnergyTemperature320}(b) is a zoom-in plot of Figure \ref{SquareCH4andnC10TotalEnergyTemperature320}(a)  in the later time steps, and it shows  that the total (free) energy remains to decrease all the time.   These results  verify the effectiveness of the proposed method.

The initial molar density distributions for  methane and decane are illustrated in  Figure \ref{SquareCH4andnC10MolarDensityOfCH4Temperature320K}(a) and Figure \ref{SquareCH4andnC10MolarDensityOfnC10Temperature320K}(a) respectively.  In Figure \ref{SquareCH4andnC10MolarDensityOfCH4Temperature320K}(b)(c) and Figure \ref{SquareCH4andnC10MolarDensityOfnC10Temperature320K}(b)(c), we show the changes of each component molar density  in the dynamical evolution. It is obviously observed that the droplet turns to a circle from its initial square shape under the effect of driving force  (i.e. the gradient of chemical potential of this species).  In Figures \ref{SquareCH4andnC10VelocityTemperature320K}, we illustrate    the velocity field and  magnitudes of both velocity components at the last time step.

\begin{figure}
            \centering \subfigure[]{
            \begin{minipage}[b]{0.45\textwidth}
               \centering
             \includegraphics[width=0.95\textwidth,height=2in]{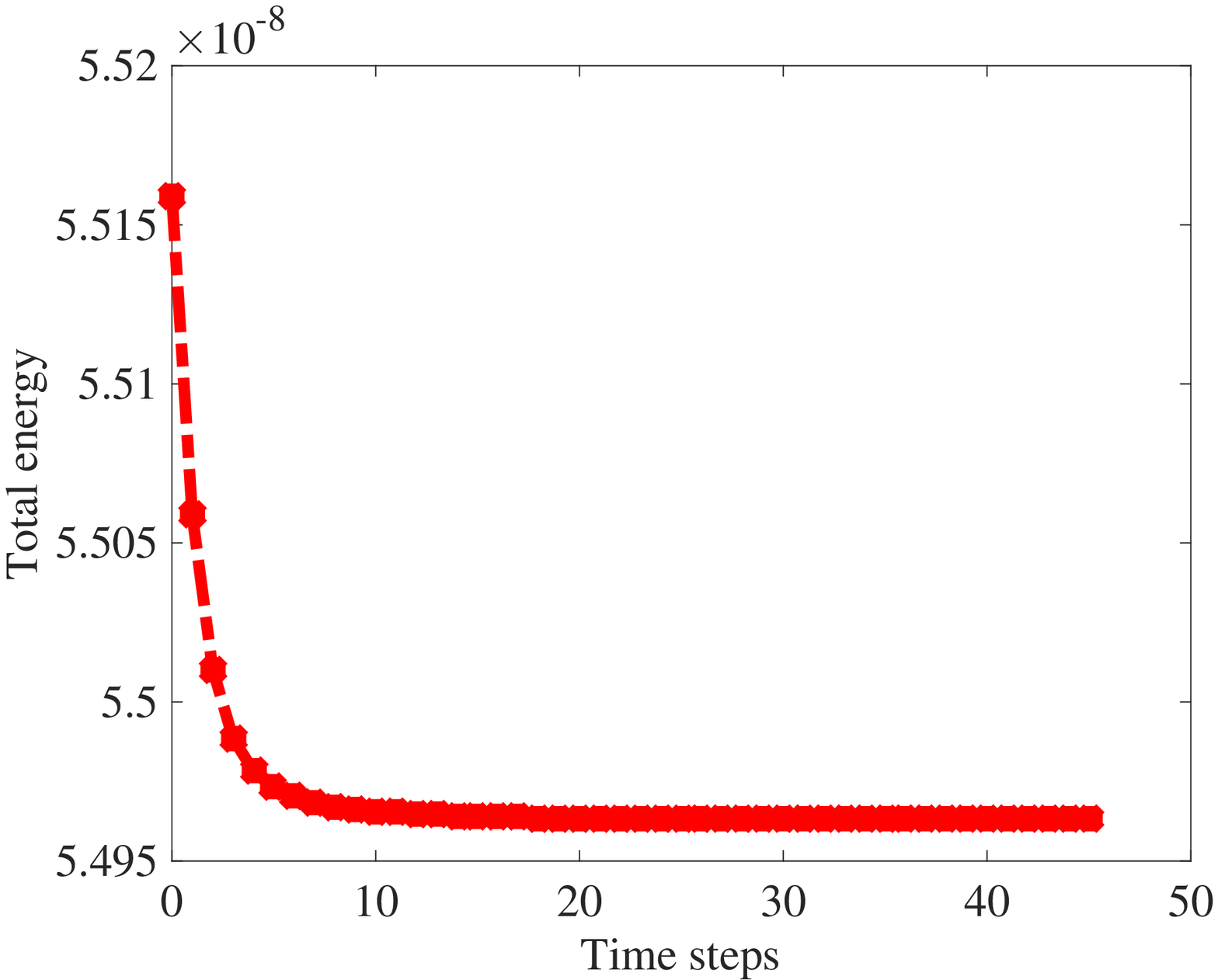}
            \end{minipage}
            }
            \centering \subfigure[]{
            \begin{minipage}[b]{0.45\textwidth}
            \centering
             \includegraphics[width=0.95\textwidth,height=2in]{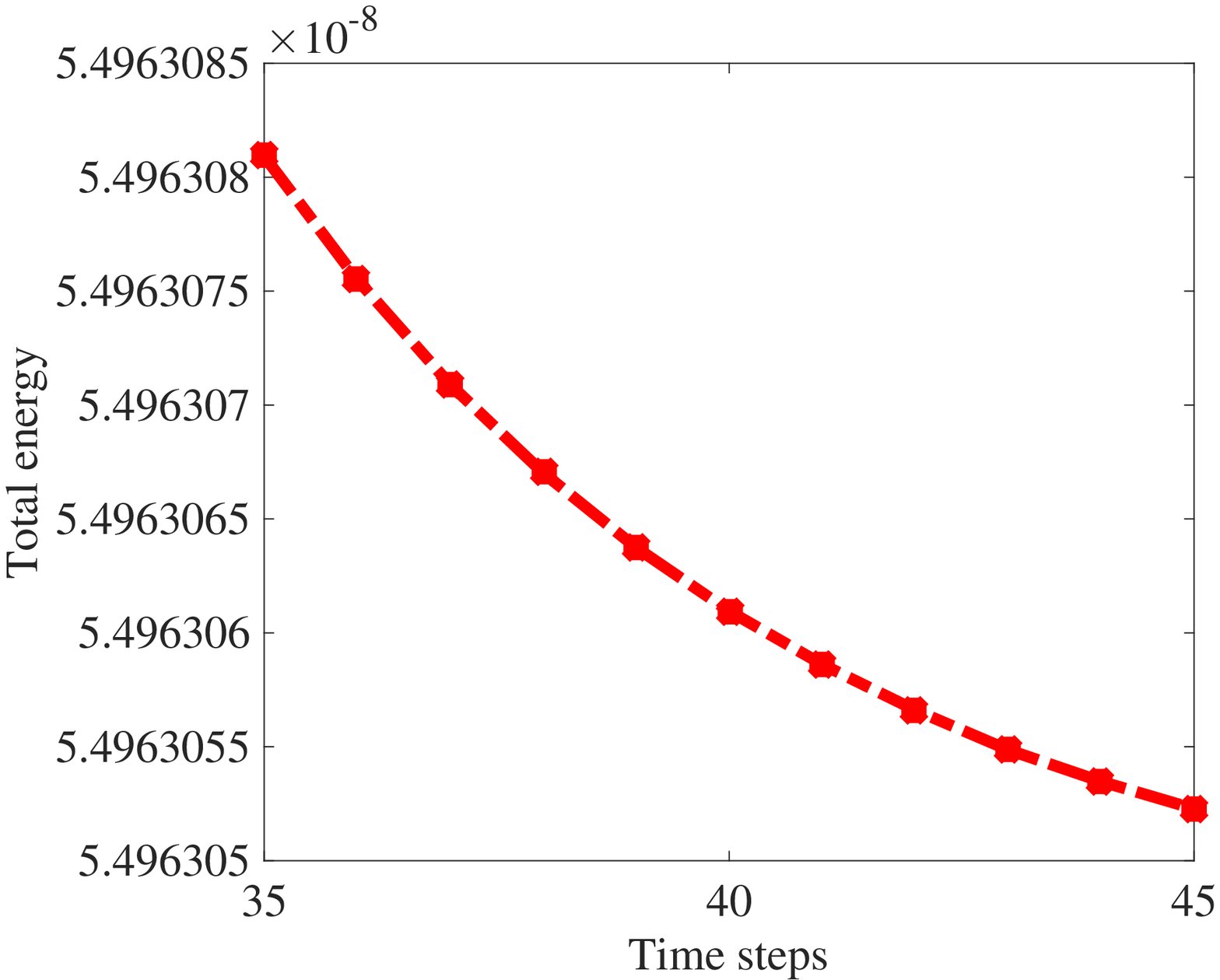}
            \end{minipage}
            }
           \caption{Example 1:  energy dissipation with time steps.}
            \label{SquareCH4andnC10TotalEnergyTemperature320}
 \end{figure}

\begin{figure}
            \centering \subfigure[]{
            \begin{minipage}[b]{0.31\textwidth}
               \centering
             \includegraphics[width=0.95\textwidth,height=0.9\textwidth]{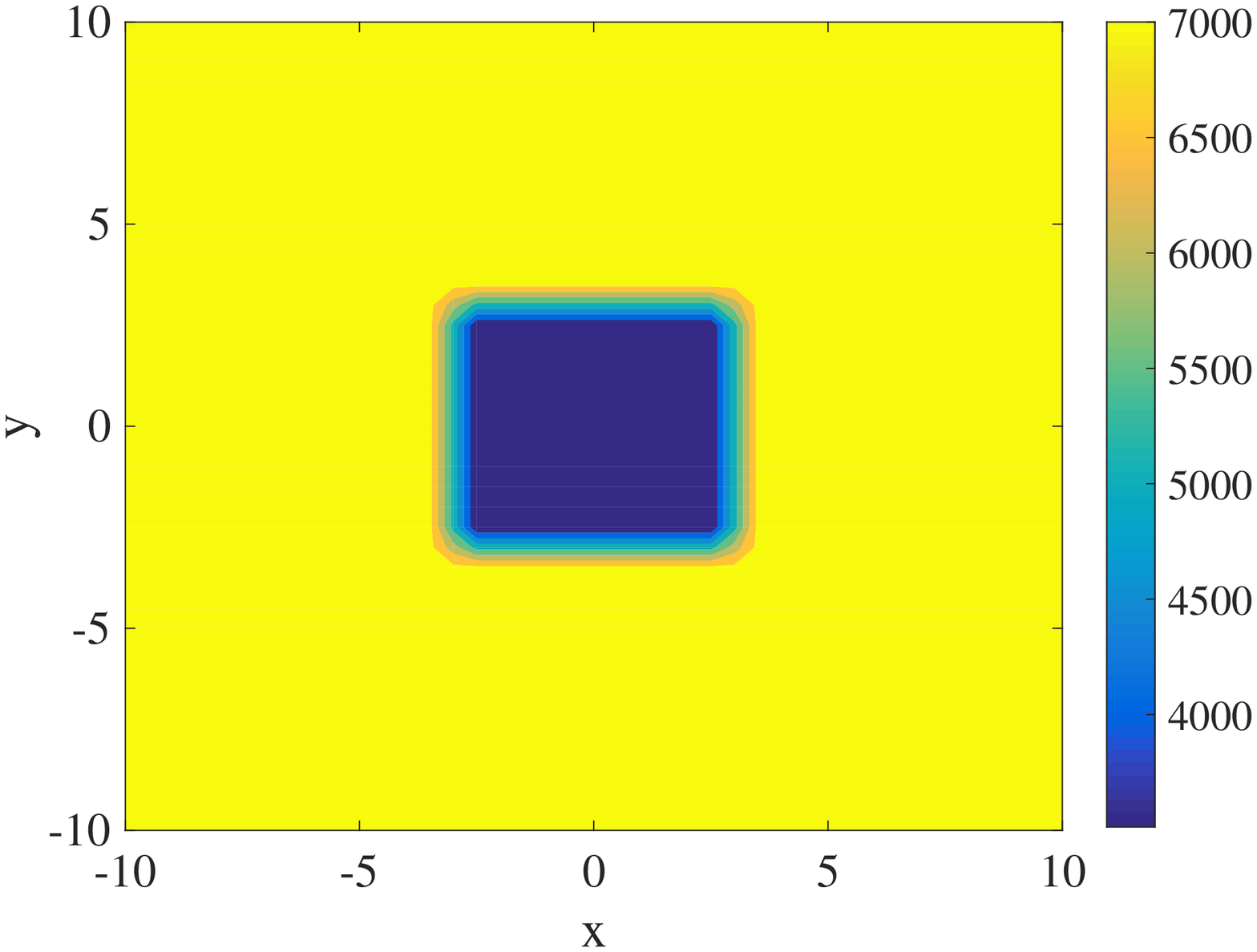}
            \end{minipage}
            }
            \centering \subfigure[]{
            \begin{minipage}[b]{0.31\textwidth}
            \centering
             \includegraphics[width=0.95\textwidth,height=0.9\textwidth]{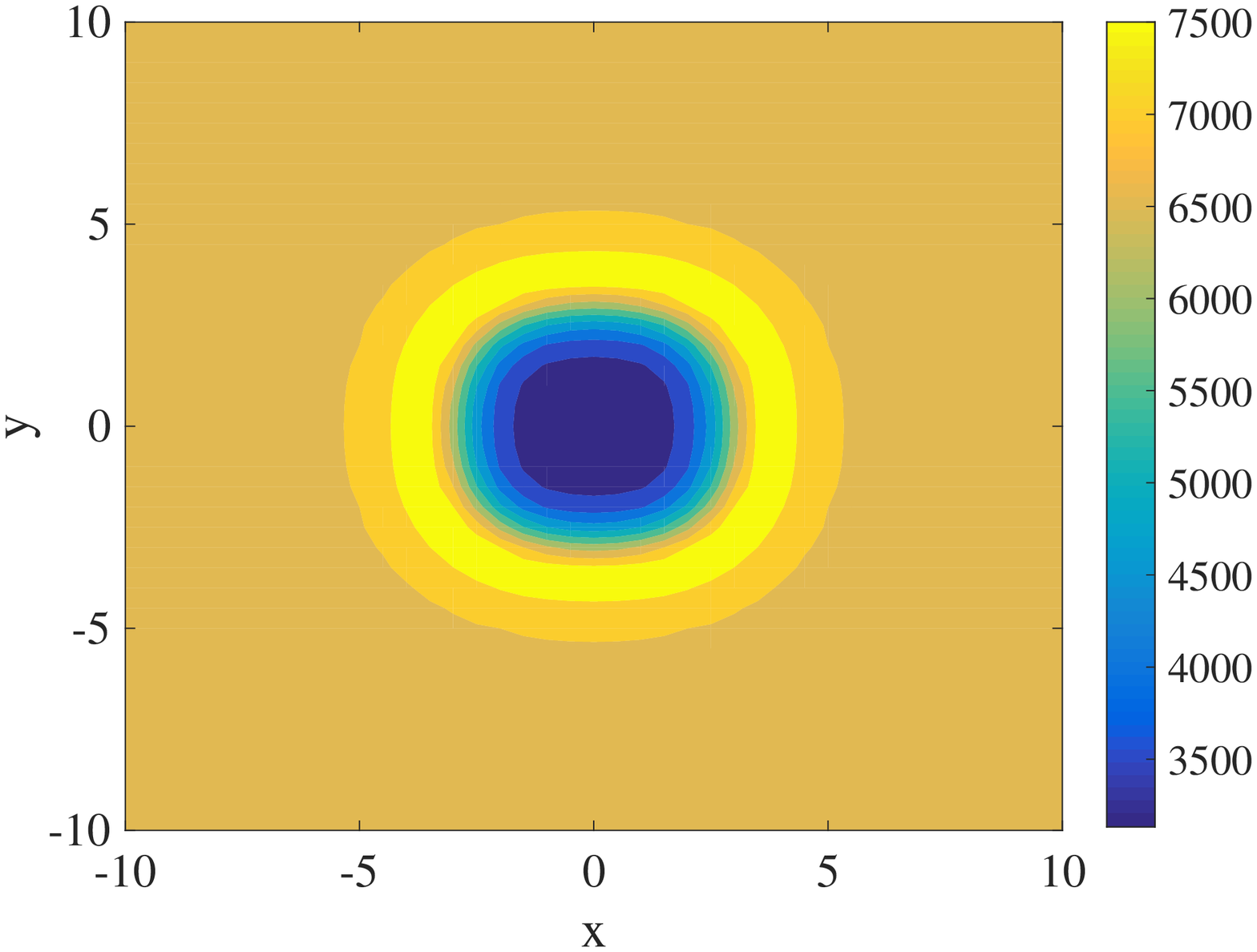}
            \end{minipage}
            }
           \centering \subfigure[]{
            \begin{minipage}[b]{0.31\textwidth}
               \centering
             \includegraphics[width=0.95\textwidth,height=0.9\textwidth]{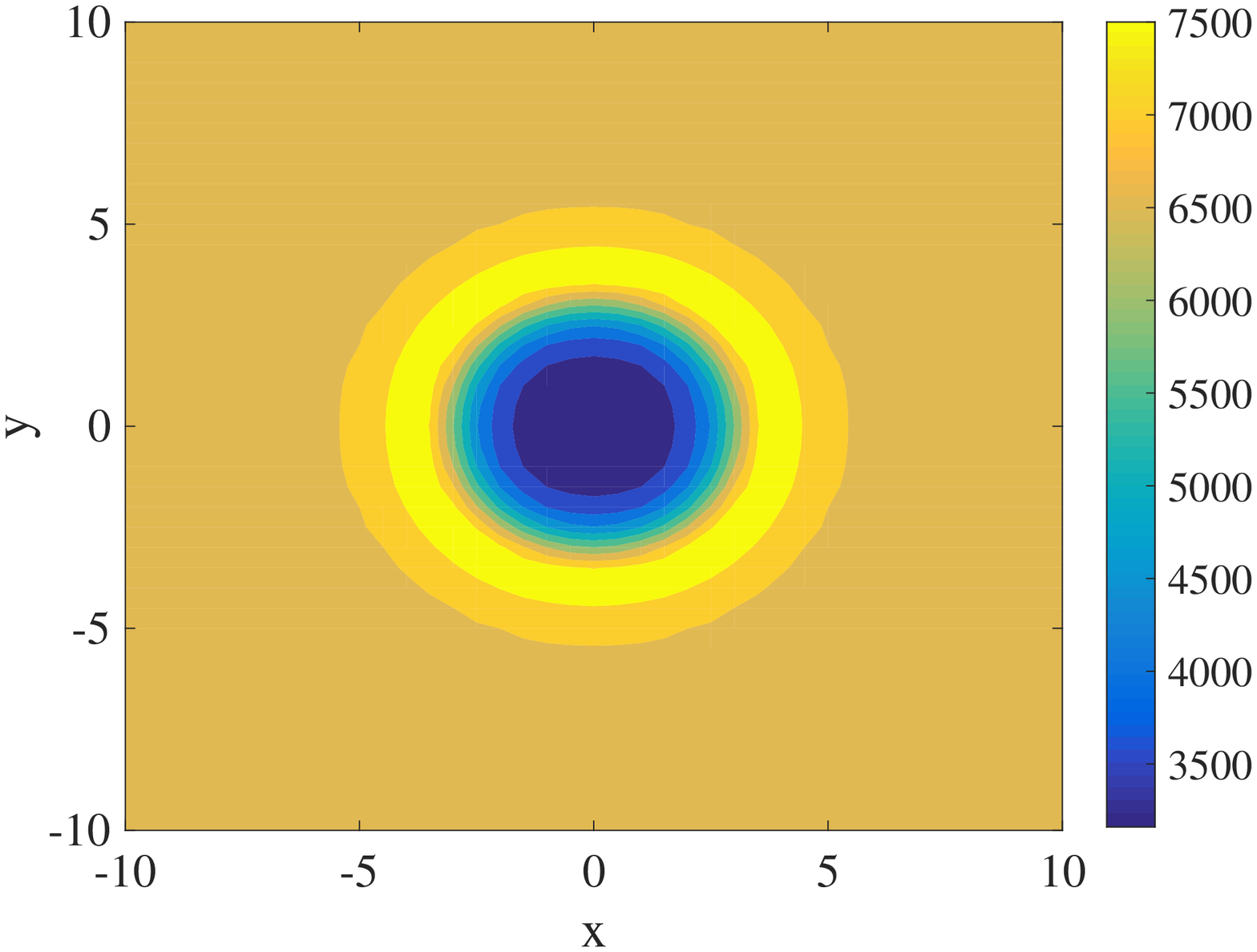}
            \end{minipage}
            }
            \caption{Example 1:   CH$_4$ molar densities at  the the initial(a),  20th(b),   and 45th(c)    time step  respectively.}
            \label{SquareCH4andnC10MolarDensityOfCH4Temperature320K}
 \end{figure}

\begin{figure}
            \centering \subfigure[]{
            \begin{minipage}[b]{0.31\textwidth}
               \centering
             \includegraphics[width=0.95\textwidth,height=0.9\textwidth]{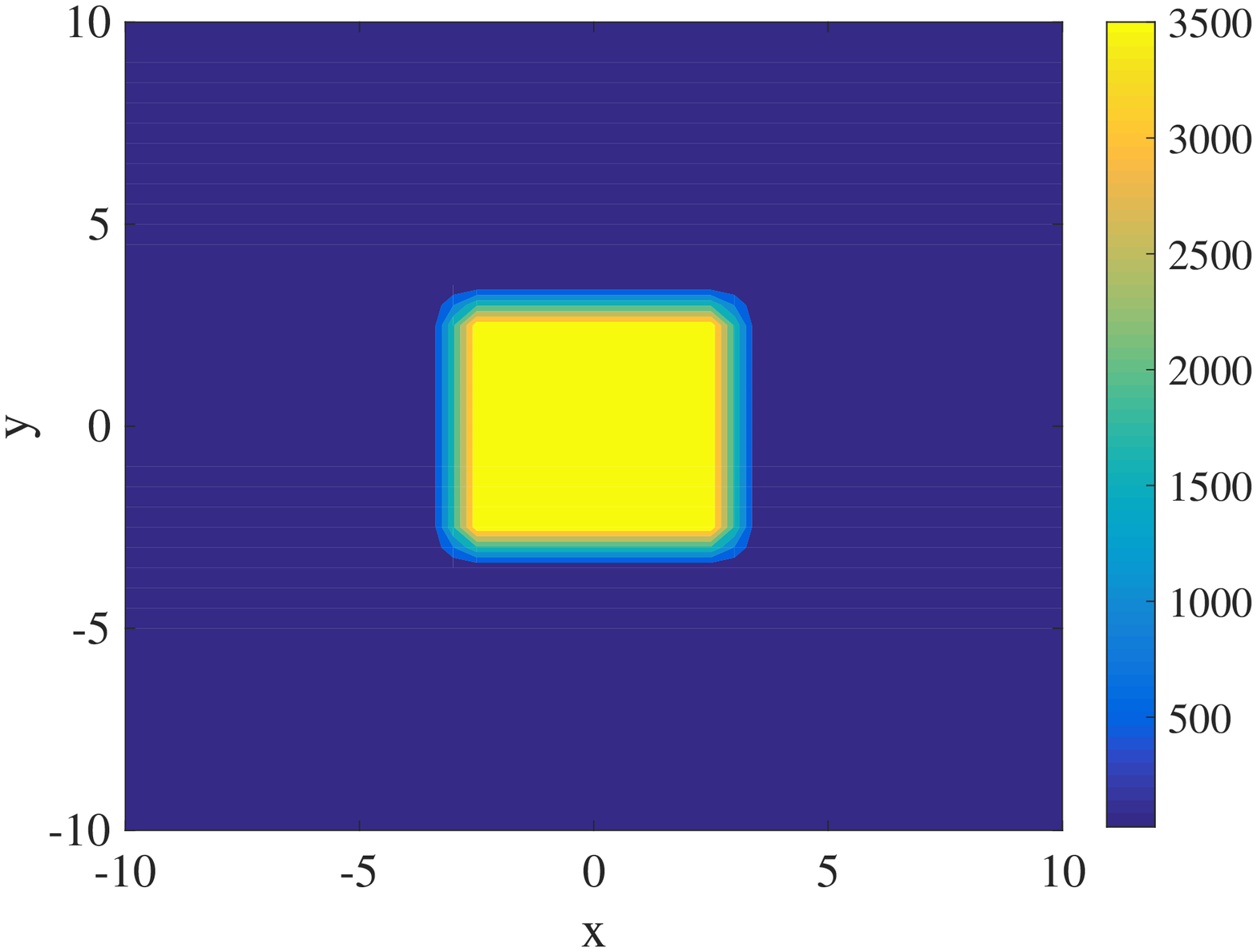}
            \end{minipage}
            }
            \centering \subfigure[]{
            \begin{minipage}[b]{0.31\textwidth}
            \centering
             \includegraphics[width=0.95\textwidth,height=0.9\textwidth]{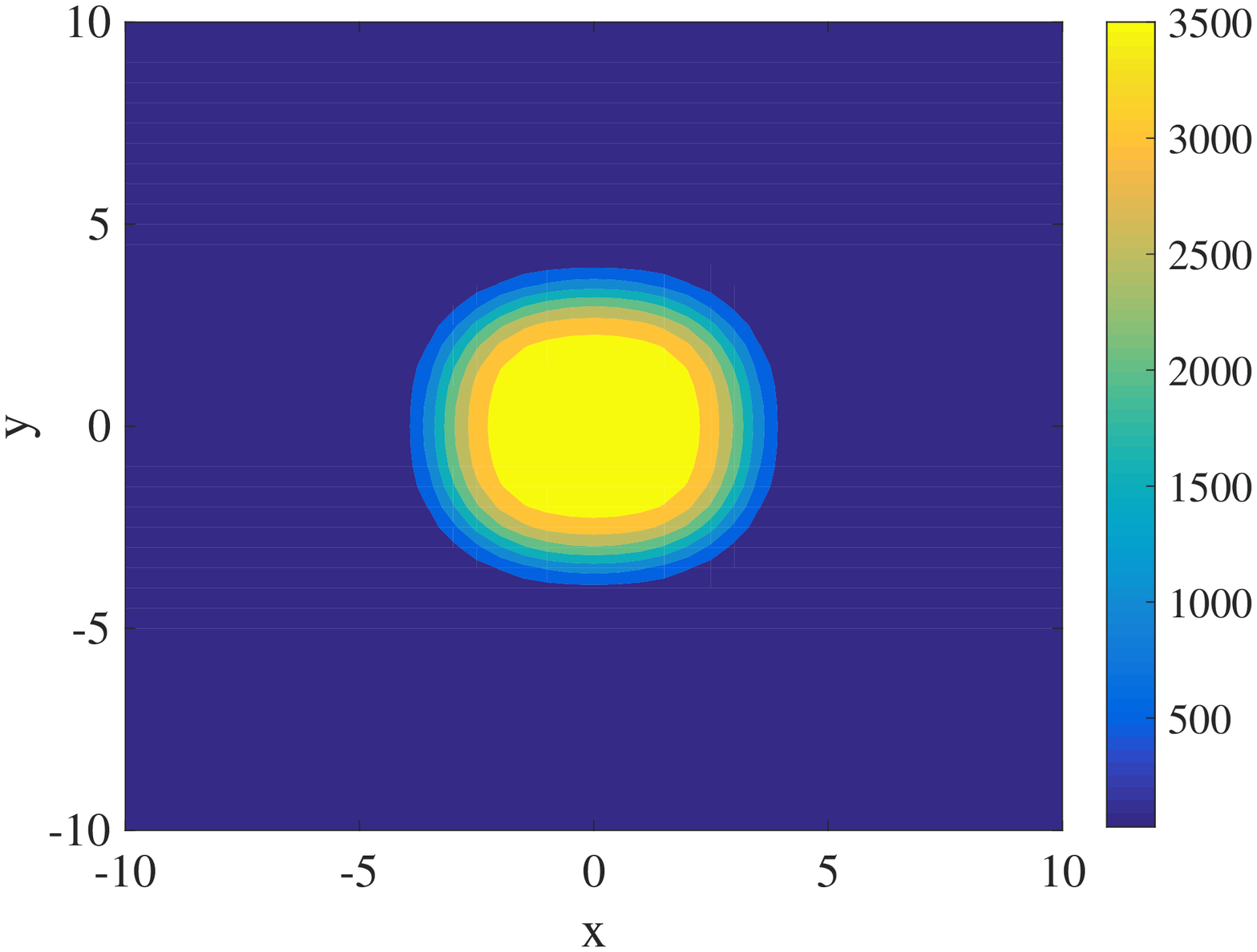}
            \end{minipage}
            }
           \centering \subfigure[]{
            \begin{minipage}[b]{0.31\textwidth}
               \centering
             \includegraphics[width=0.95\textwidth,height=0.9\textwidth]{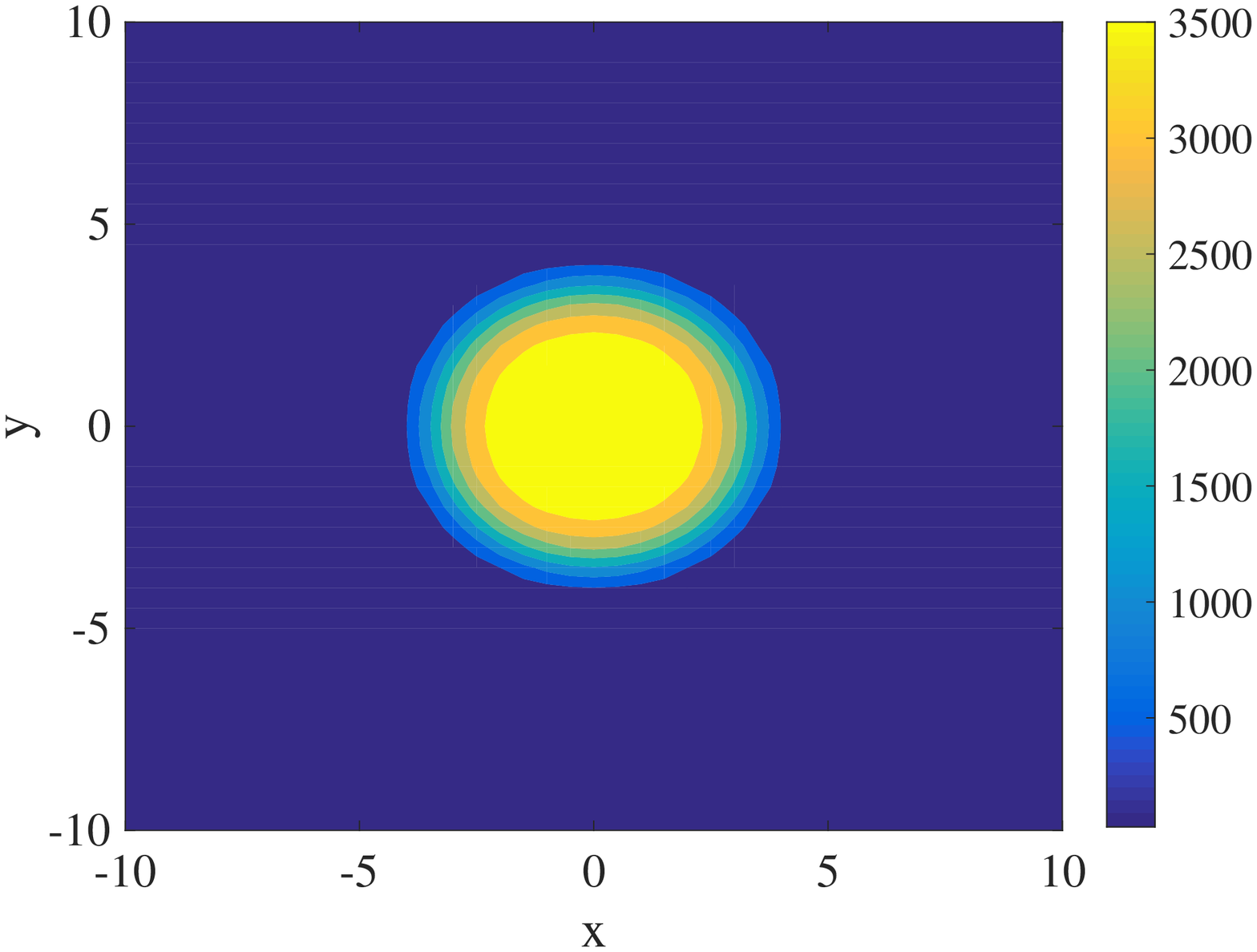}
            \end{minipage}
            }
            \caption{Example 1: nC$_{10}$ molar densities at  the initial(a),   20th(b),   and 45th(c)   time step  respectively.}
            \label{SquareCH4andnC10MolarDensityOfnC10Temperature320K}
 \end{figure}

\begin{figure}
           \centering \subfigure[]{
            \begin{minipage}[b]{0.3\textwidth}
            \centering
             \includegraphics[width=0.95\textwidth,height=0.9\textwidth]{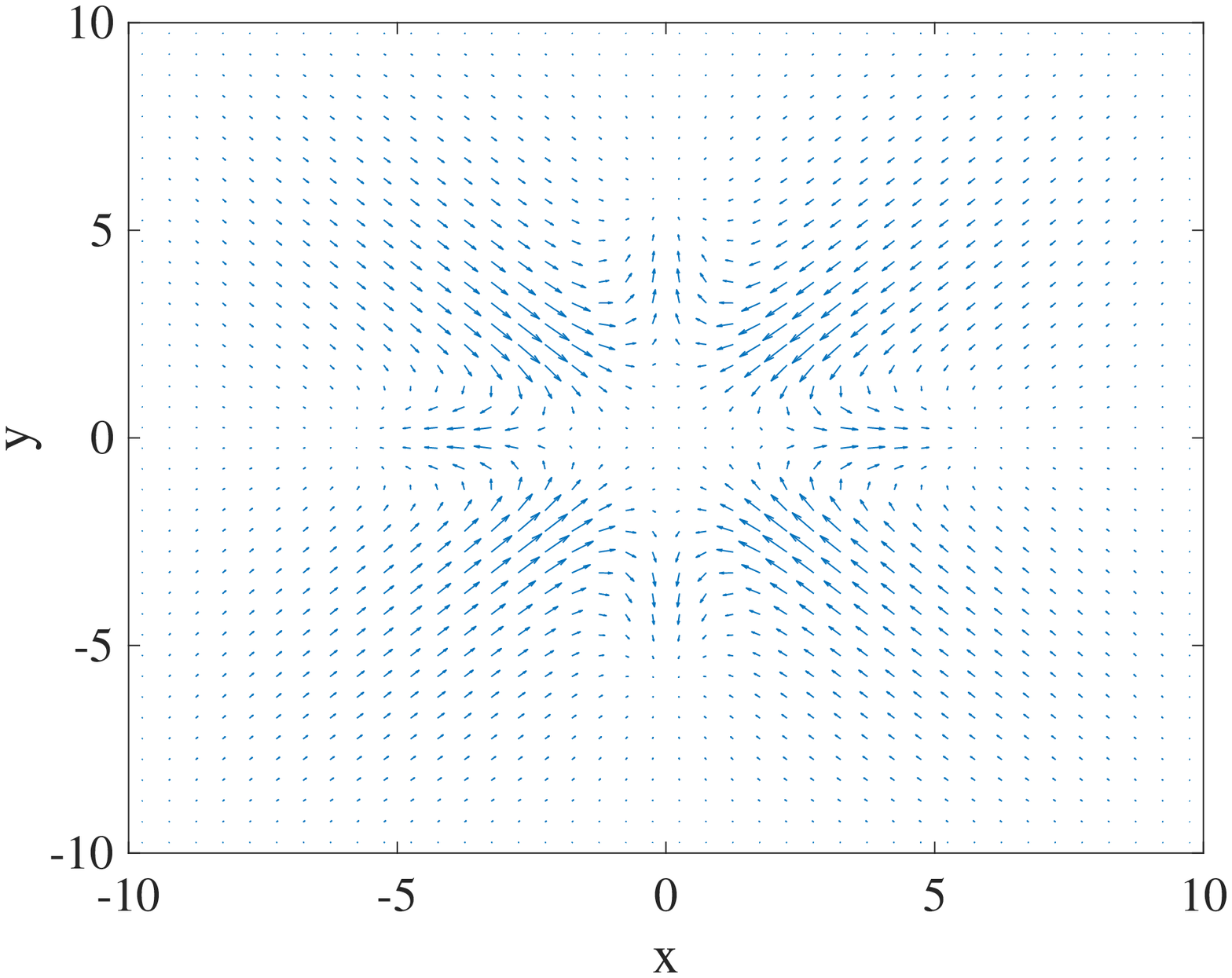}
            \end{minipage}
            }
           \centering \subfigure[]{
            \begin{minipage}[b]{0.3\textwidth}
            \centering
             \includegraphics[width=0.95\textwidth,height=0.9\textwidth]{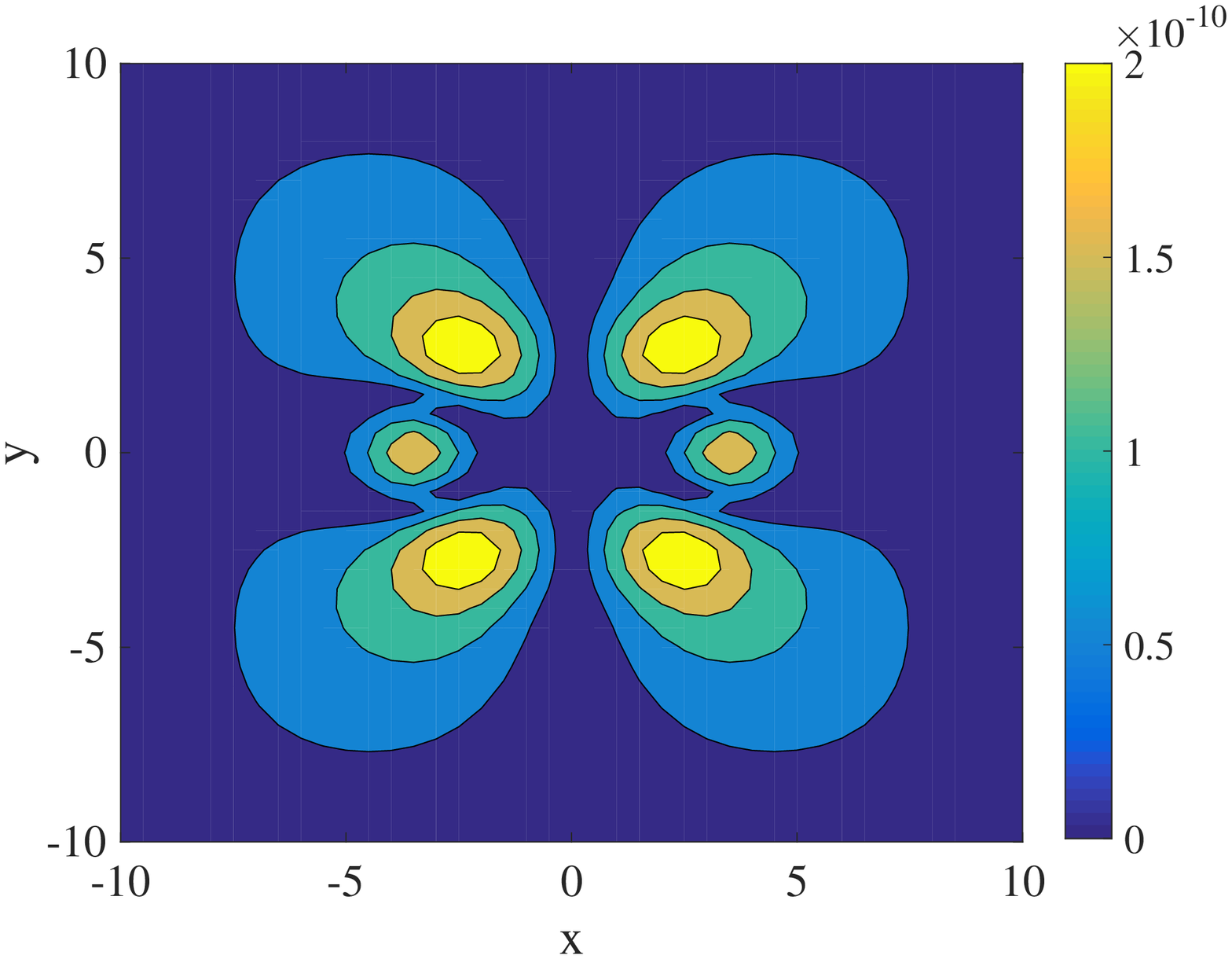}
            \end{minipage}
            }
           \centering \subfigure[]{
            \begin{minipage}[b]{0.3\textwidth}
            \centering
             \includegraphics[width=0.95\textwidth,height=0.9\textwidth]{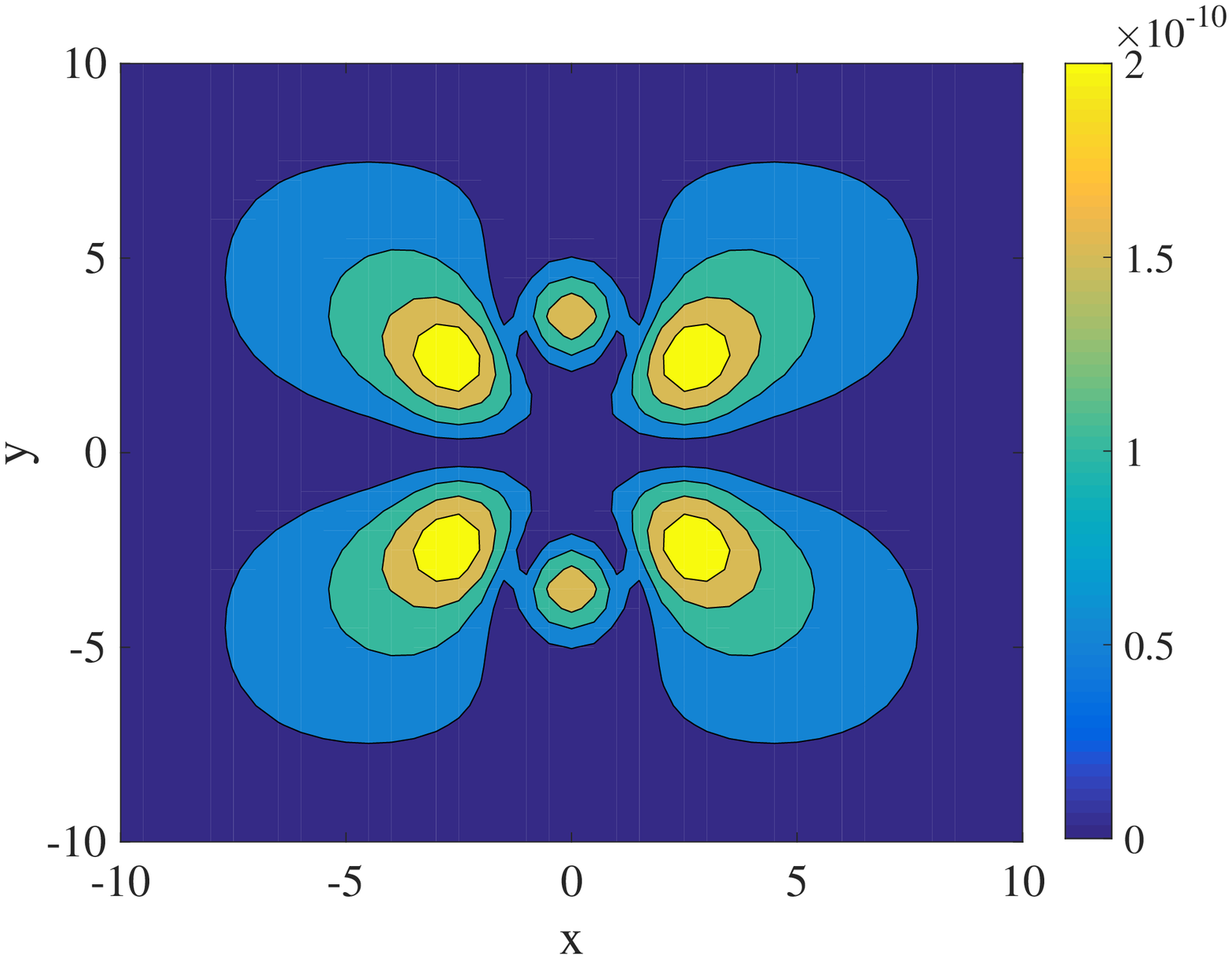}
            \end{minipage}
            }
           \caption{Example 1:   flow quiver(a), magnitude contour of $x$-direction velocity component(b), and magnitude contour of $y$-direction velocity component(c)   at the  45th  time step.}
            \label{SquareCH4andnC10VelocityTemperature320K}
 \end{figure}

\subsection{Example 2:  an ellipse shape bubble}
We simulate  the dynamical evolution of a bubble, which is   ellipse shaped  in   the center of the domain at the initial moment.  The  temperature of the fluid is constant at 330K.   The initial molar densities of methane and decane in gas phase  are $7.6181\times10^3$mol/m$^3$ and $0.0445\times10^3$mol/m$^3$ respectively, while the initial molar densities of methane and decane in  liquid phase are $3.8336\times10^3$mol/m$^3$ and $ 3.6843\times10^3$mol/m$^3$  respectively.    We set the time step size   $10^{-5}$s, and carry out the simulation for 100  time steps.

We illustrate the total (free) energy profiles with time steps in Figure \ref{EllipseCH4andnC10TotalEnergyTemperature330}(a),  and further present a zoom-in plot  of the later time steps in Figure \ref{EllipseCH4andnC10TotalEnergyTemperature330}(b).  It is still observed that the  total (free) energy is dissipated with time steps. 

In  Figure \ref{EllipseCH4andnC10MolarDensityOfCH4Temperature330K} and Figure \ref{EllipseCH4andnC10MolarDensityOfnC10Temperature330K}, we illustrate the initial molar density configurations for  methane and decane and  the molar density changes with time steps for   each component.  Moreover, in Figures \ref{EllipseCH4andnC10VelocityTemperature330K}, we illustrate   the velocity fields and  magnitudes of  velocity components at different time steps. 
We can observe from these results that at the $x$-direction the fluids flow towards to the center from the left and right sides, while at  the $y$-direction, the fluids are flowing from the center to the bottom and top sides. As a result,  the   bubble tends towards  a circle from its original ellipse shape.

\begin{figure}
            \centering \subfigure[]{
            \begin{minipage}[b]{0.45\textwidth}
               \centering
             \includegraphics[width=0.95\textwidth,height=2in]{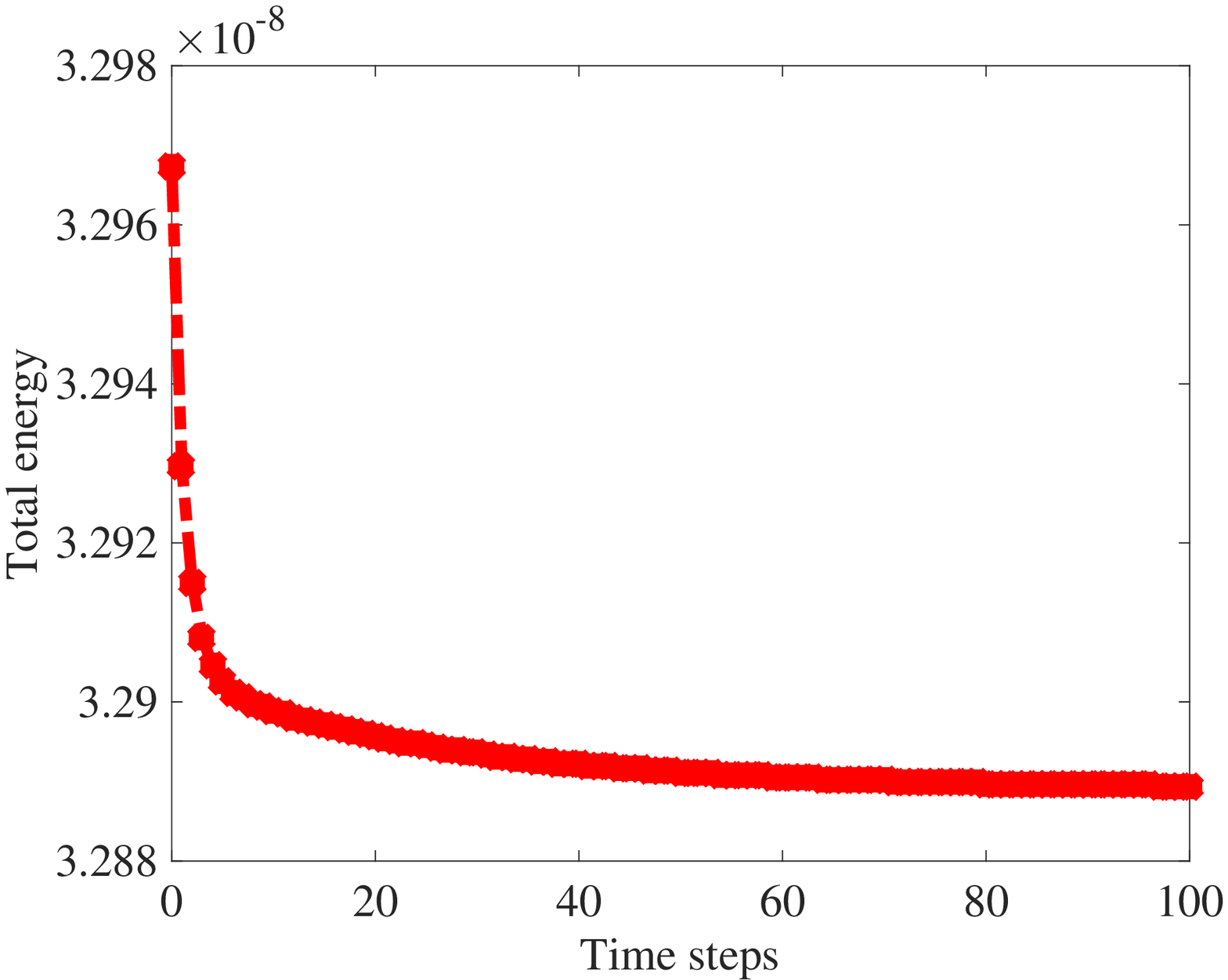}
            \end{minipage}
            }
            \centering \subfigure[]{
            \begin{minipage}[b]{0.45\textwidth}
            \centering
             \includegraphics[width=0.95\textwidth,height=2in]{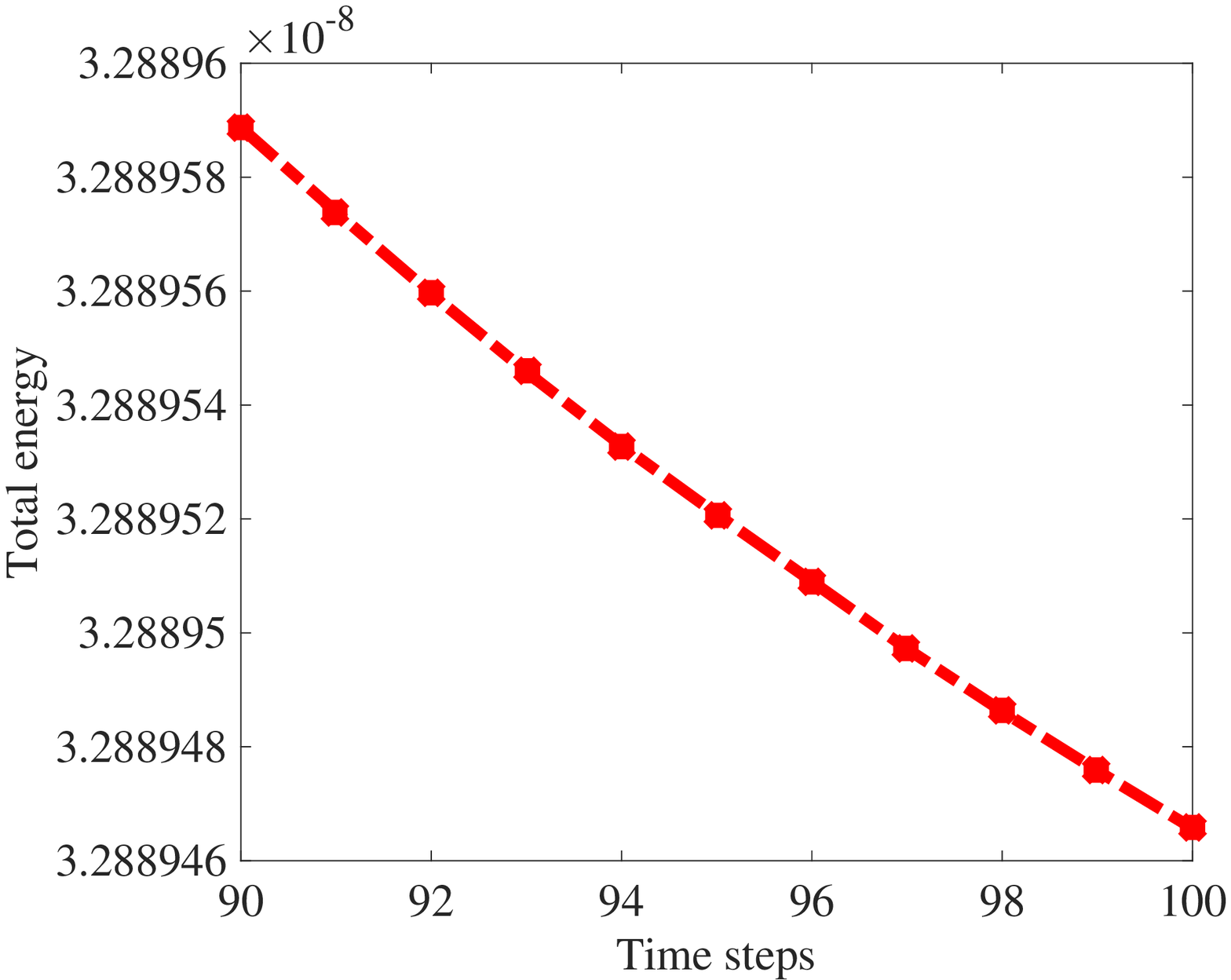}
            \end{minipage}
            }
           \caption{Example 2:  energy dissipation with time steps.}
            \label{EllipseCH4andnC10TotalEnergyTemperature330}
 \end{figure}

\begin{figure}
            \centering \subfigure[]{
            \begin{minipage}[b]{0.3\textwidth}
               \centering
             \includegraphics[width=0.95\textwidth,height=0.9\textwidth]{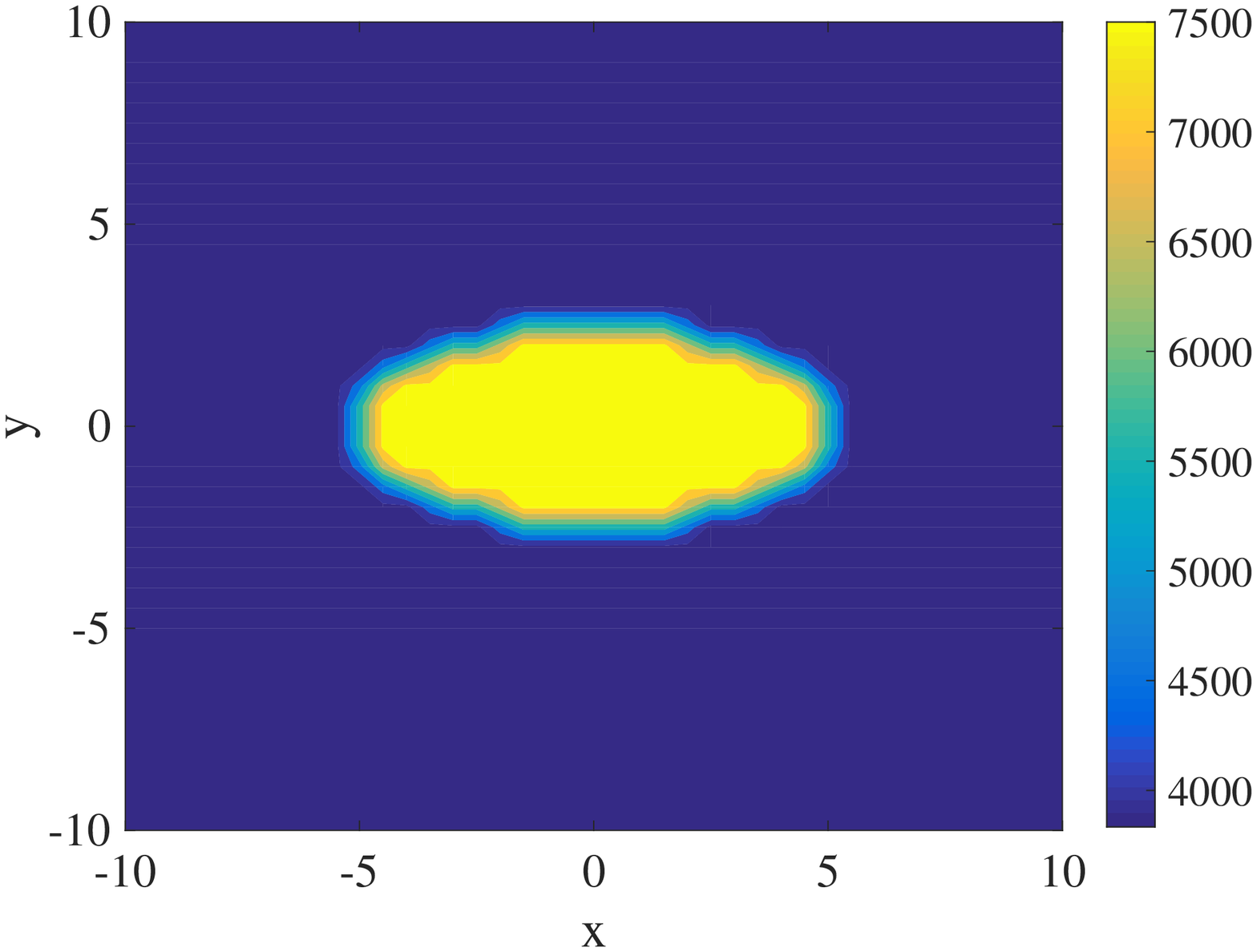}
            \end{minipage}
            }
            \centering \subfigure[]{
            \begin{minipage}[b]{0.3\textwidth}
            \centering
             \includegraphics[width=0.95\textwidth,height=0.9\textwidth]{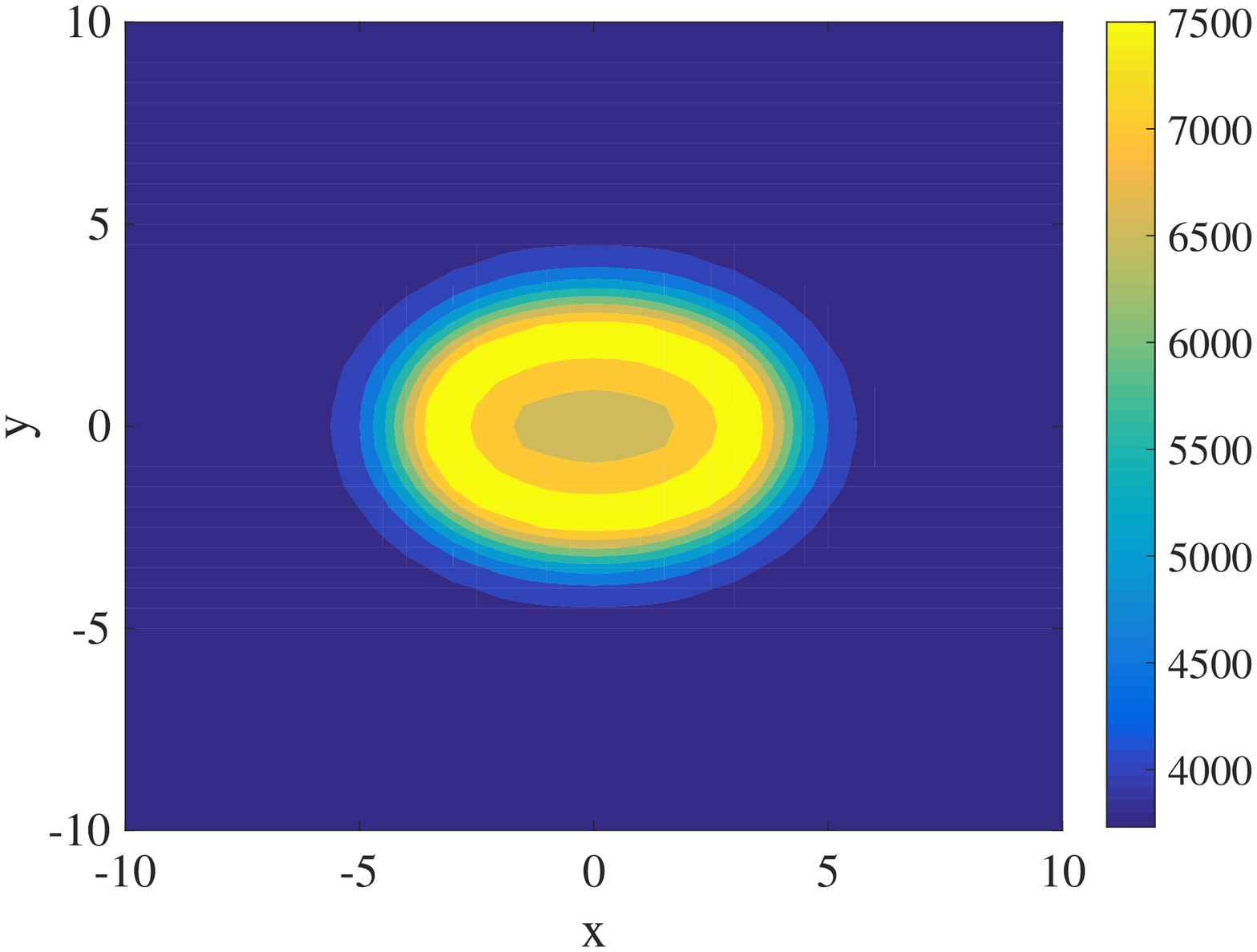}
            \end{minipage}
            }
           \centering \subfigure[]{
            \begin{minipage}[b]{0.3\textwidth}
            \centering
             \includegraphics[width=0.95\textwidth,height=0.9\textwidth]{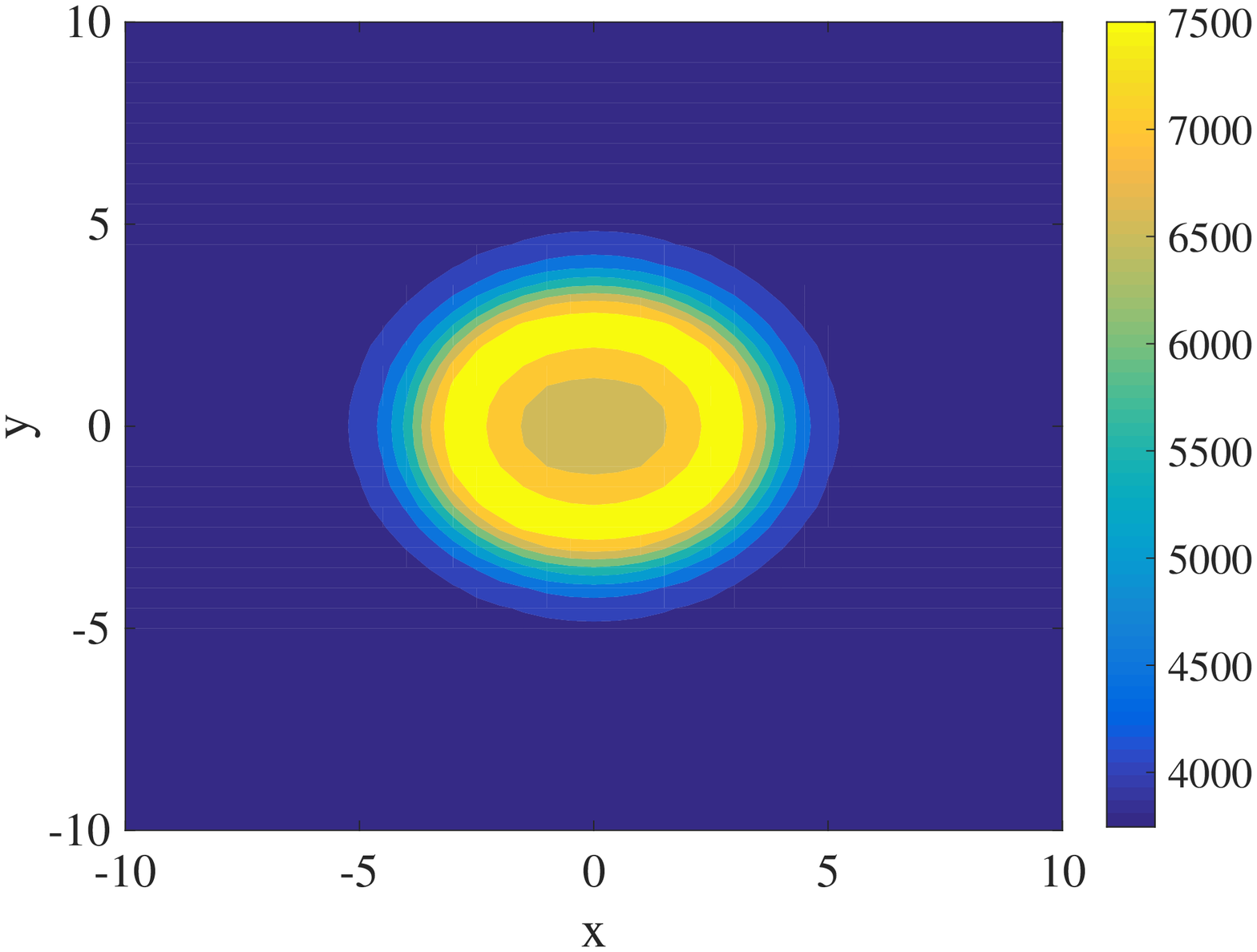}
            \end{minipage}
            }
           \caption{Example 2:  CH$_4$ molar densities at  the the initial(a), 50th(b), and 100th(c)    time step  respectively.}
            \label{EllipseCH4andnC10MolarDensityOfCH4Temperature330K}
 \end{figure}

\begin{figure}
            \centering \subfigure[]{
            \begin{minipage}[b]{0.3\textwidth}
               \centering
             \includegraphics[width=0.95\textwidth,height=0.9\textwidth]{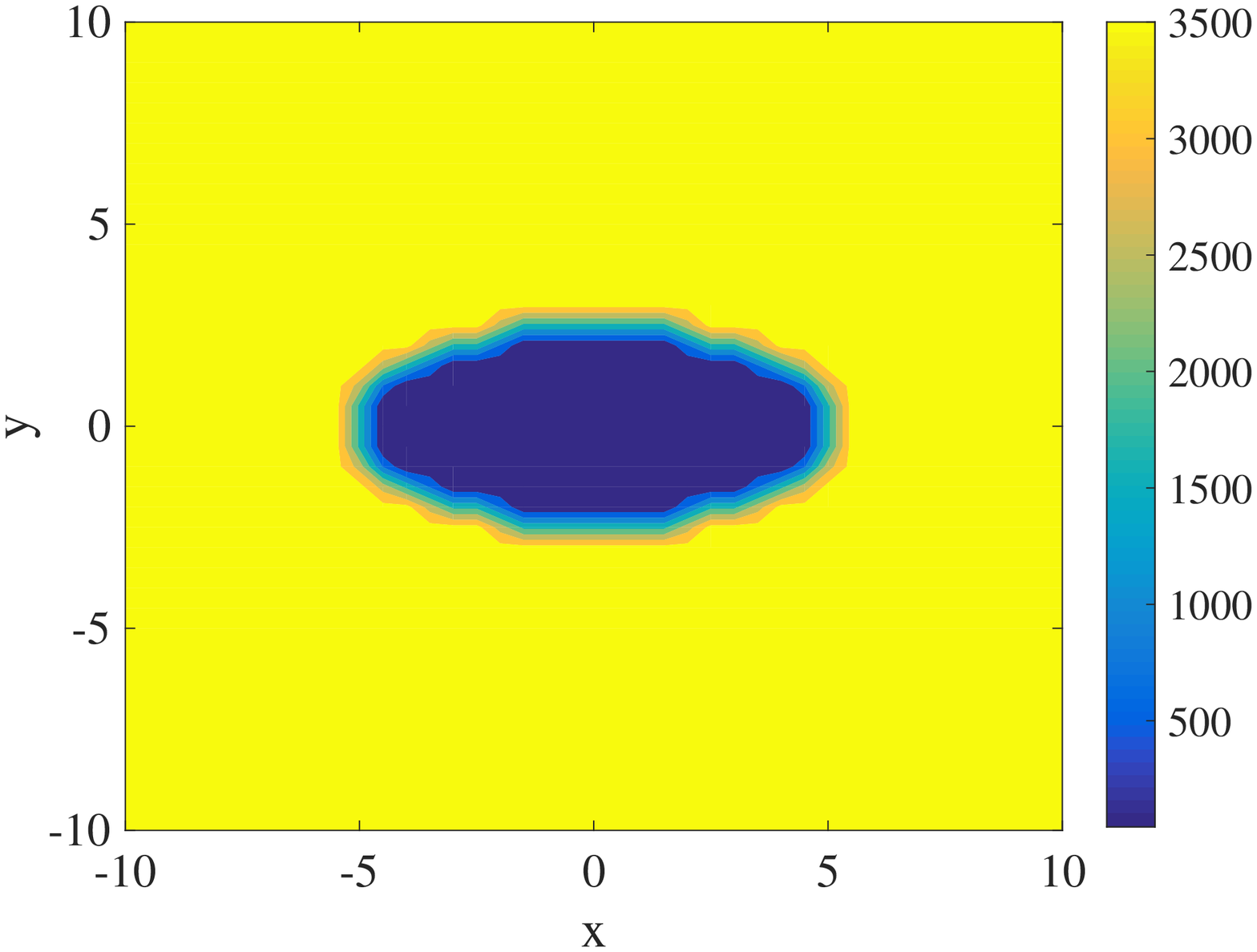}
            \end{minipage}
            }
            \centering \subfigure[]{
            \begin{minipage}[b]{0.3\textwidth}
            \centering
             \includegraphics[width=0.95\textwidth,height=0.9\textwidth]{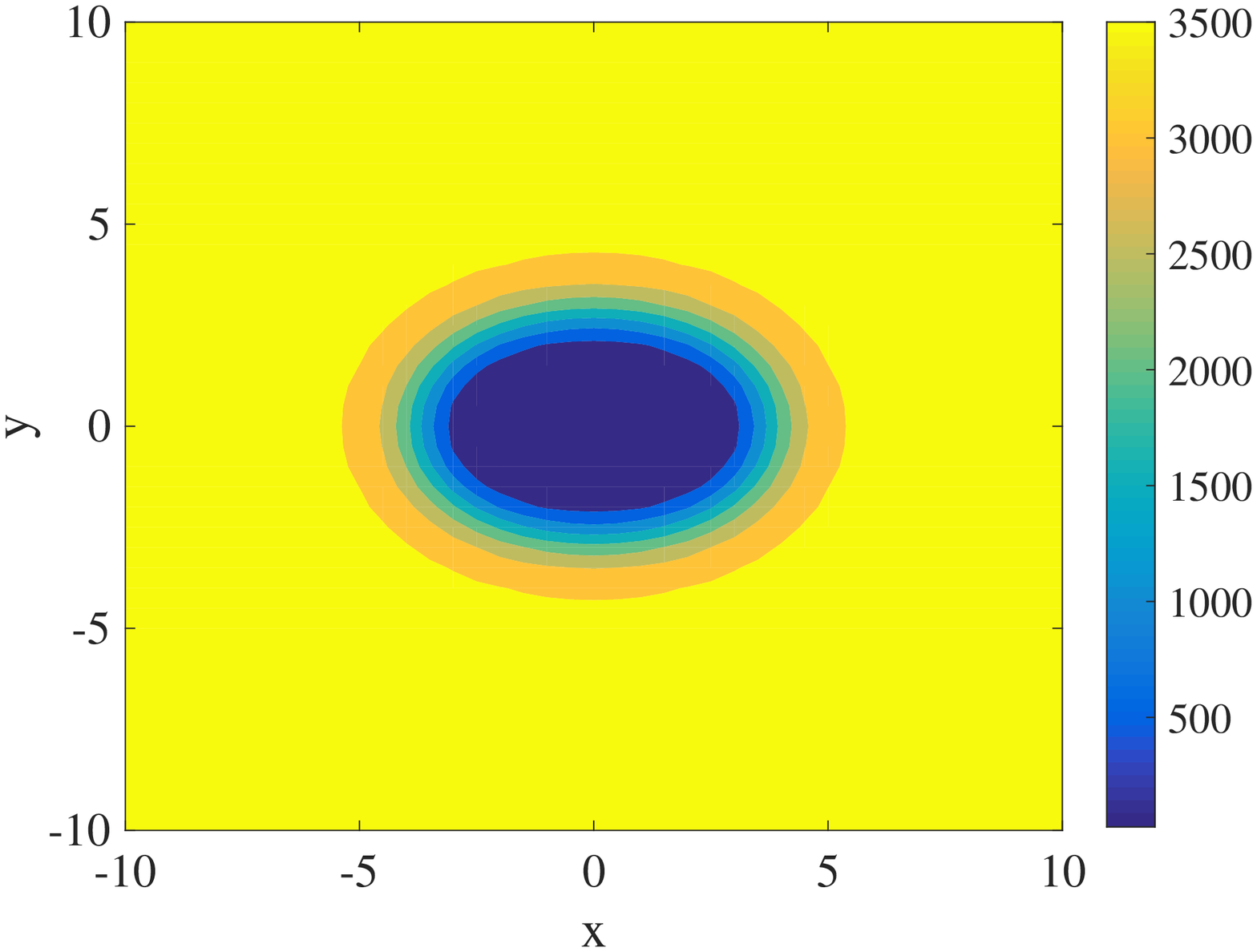}
            \end{minipage}
            }
           \centering \subfigure[]{
            \begin{minipage}[b]{0.3\textwidth}
            \centering
             \includegraphics[width=0.95\textwidth,height=0.9\textwidth]{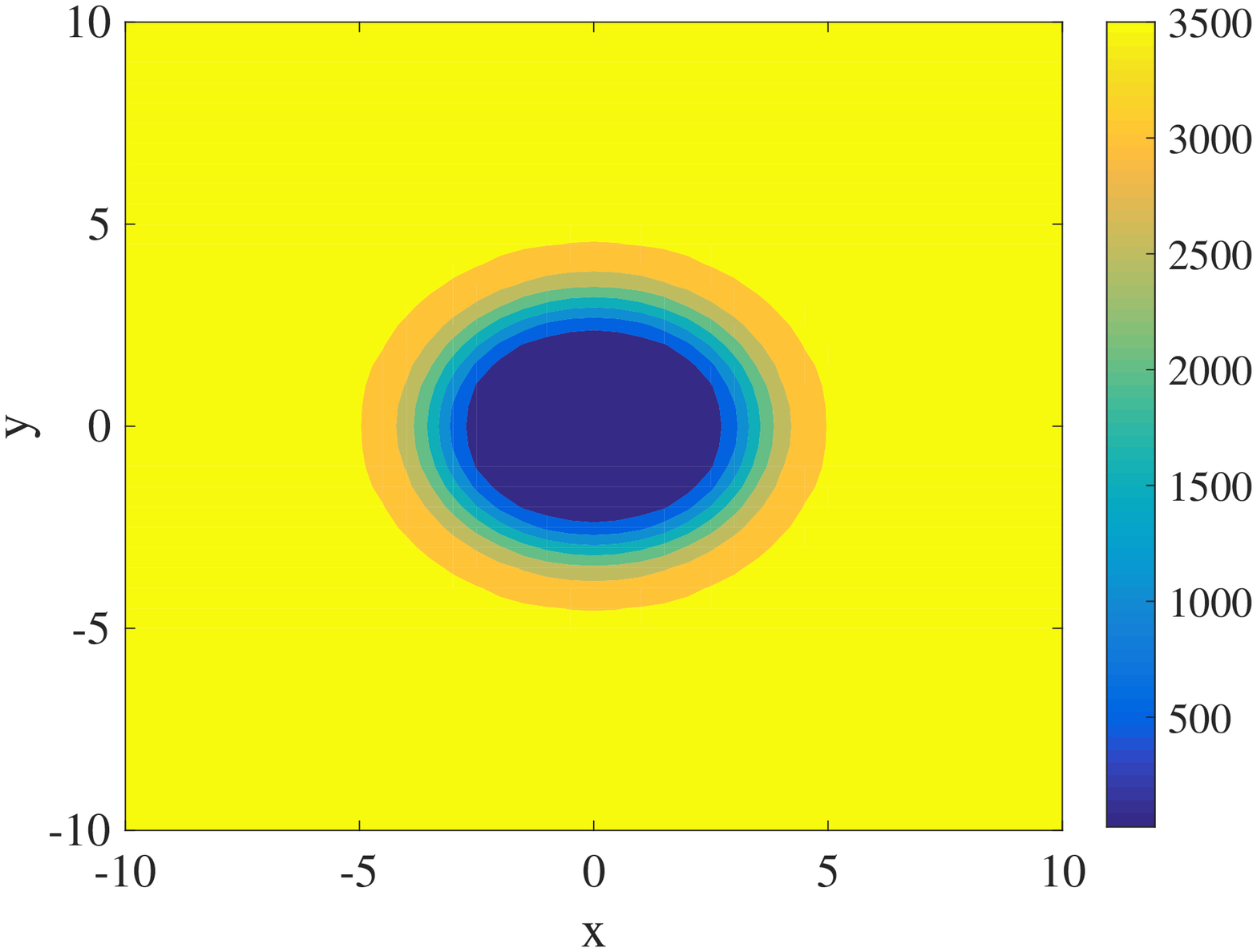}
            \end{minipage}
            }
           \caption{Example 2:  nC$_{10}$ molar densities at  the initial(a), 50th(b), and 100th(c)   time step  respectively.}
            \label{EllipseCH4andnC10MolarDensityOfnC10Temperature330K}
 \end{figure}

\begin{figure}
           \centering \subfigure[]{
            \begin{minipage}[b]{0.3\textwidth}
               \centering
             \includegraphics[width=0.95\textwidth,height=0.9\textwidth]{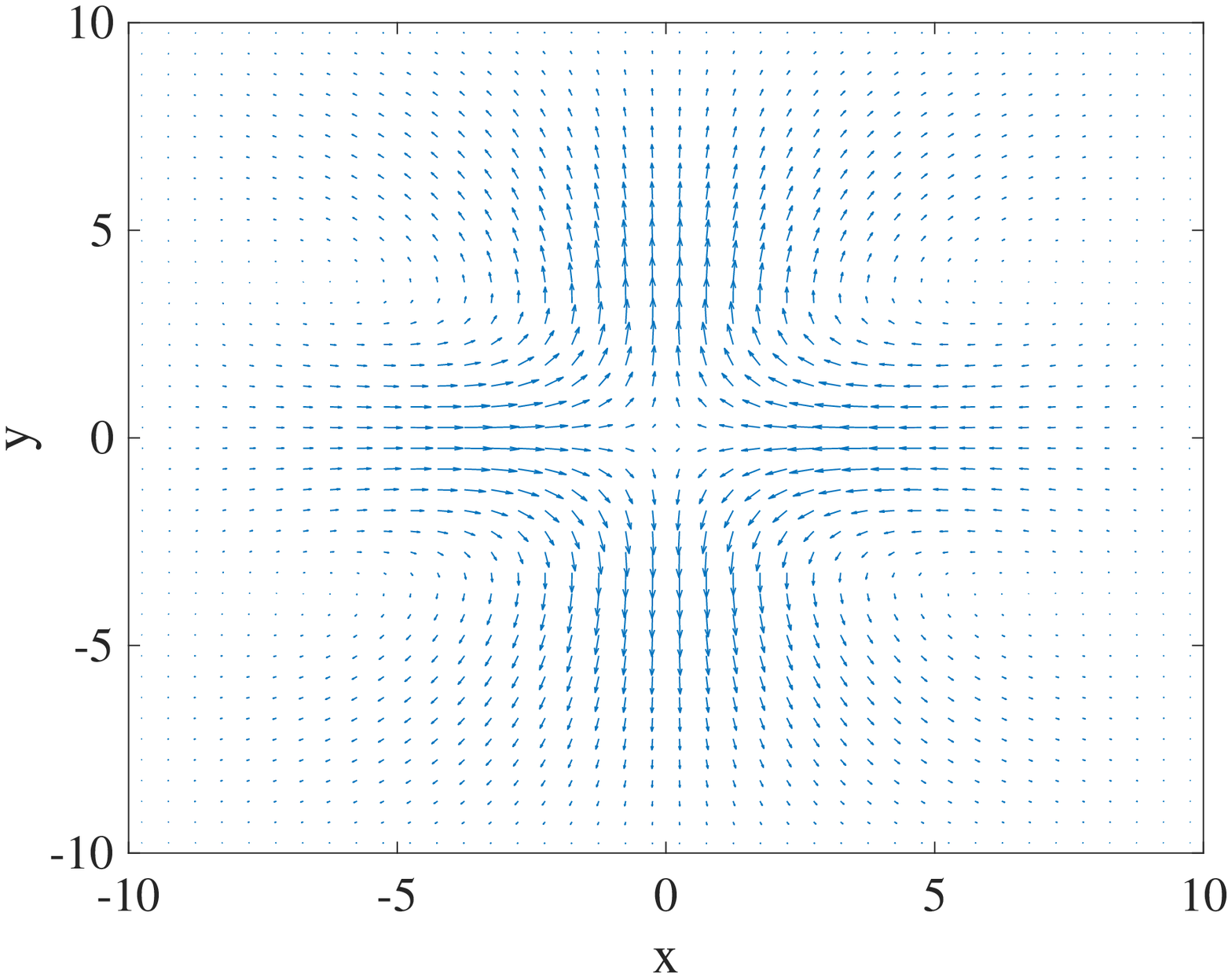}
            \end{minipage}
            }
            \centering \subfigure[]{
            \begin{minipage}[b]{0.3\textwidth}
               \centering
             \includegraphics[width=0.95\textwidth,height=0.9\textwidth]{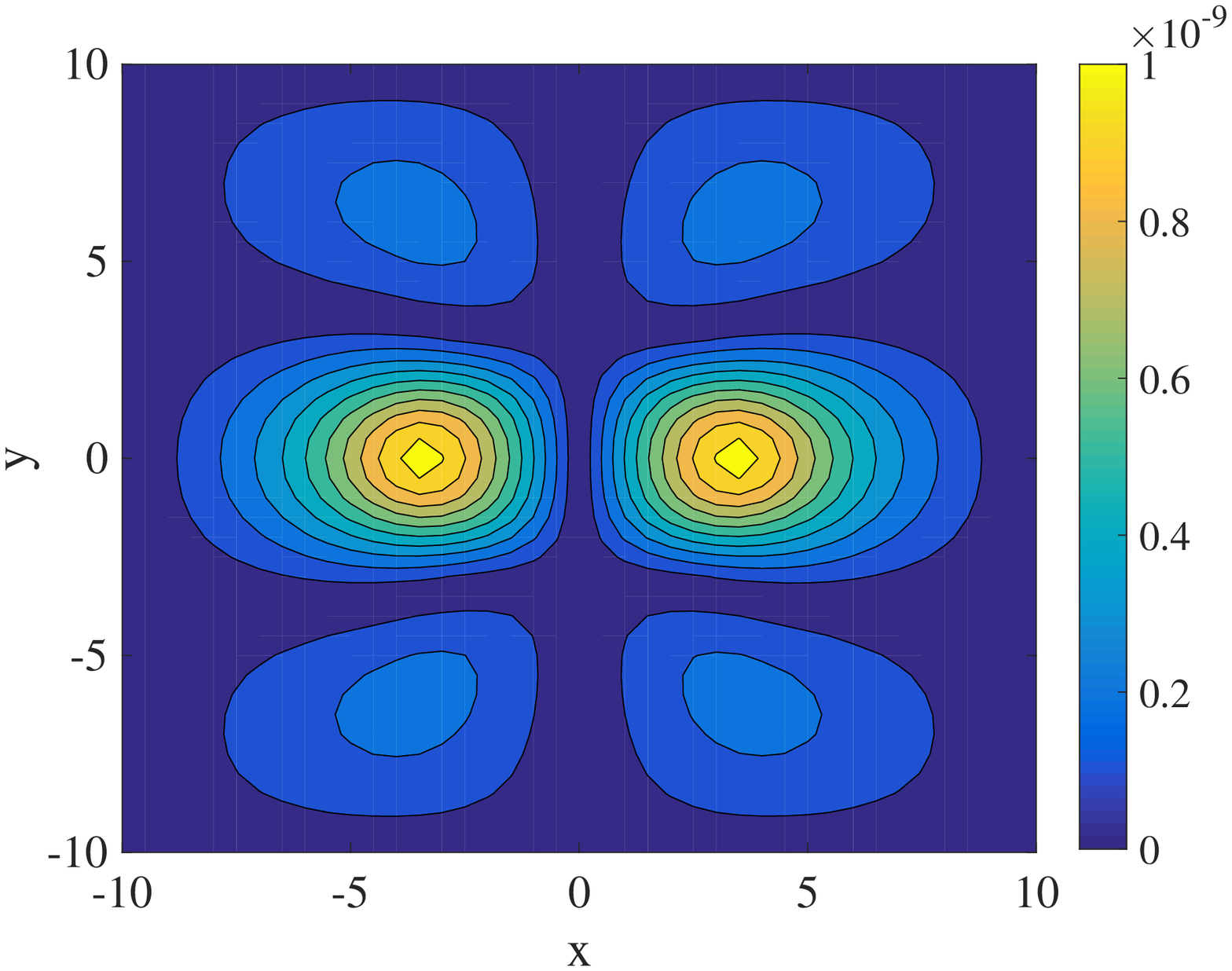}
            \end{minipage}
            }
            \centering \subfigure[]{
            \begin{minipage}[b]{0.3\textwidth}
            \centering
             \includegraphics[width=0.95\textwidth,height=0.9\textwidth]{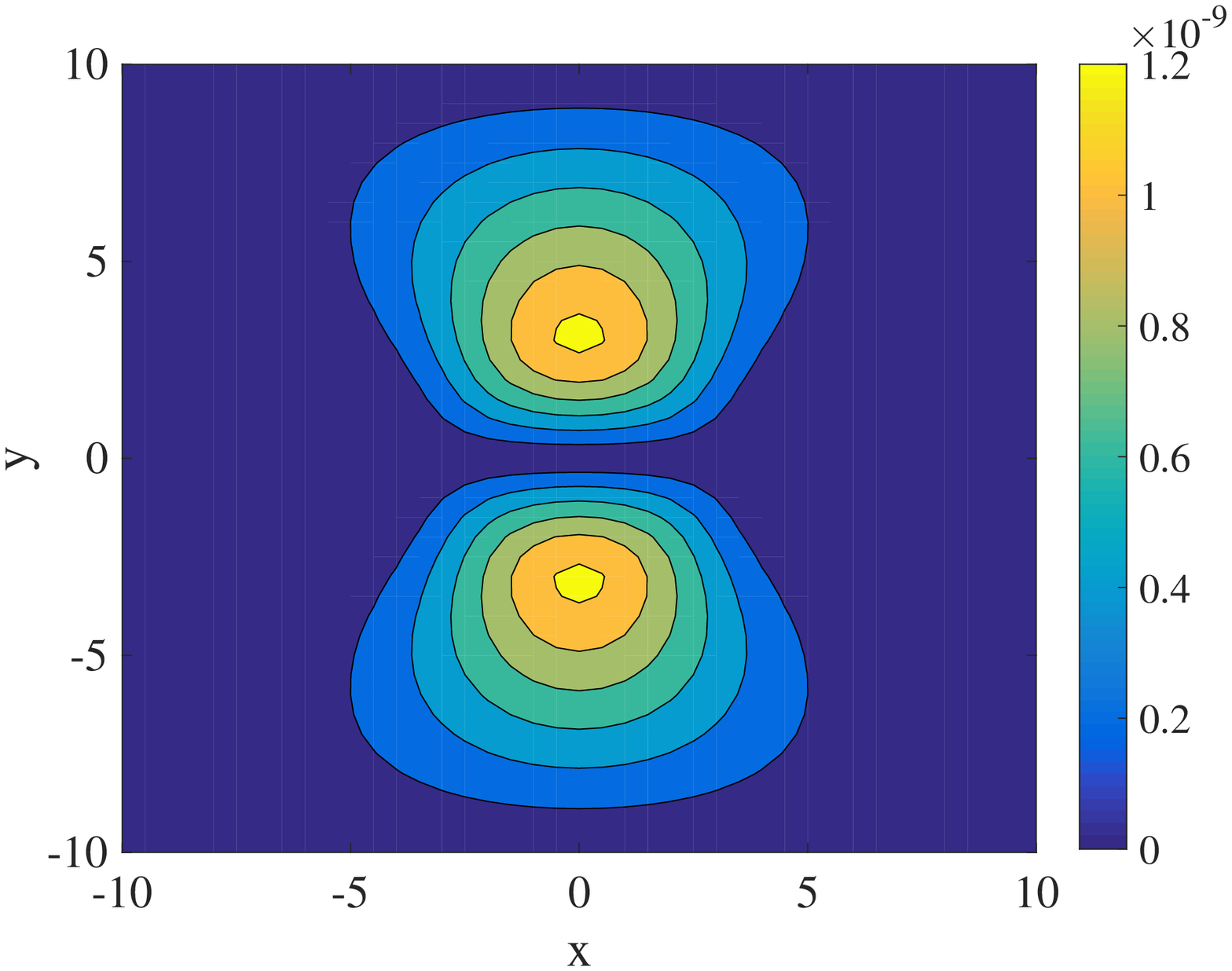}
            \end{minipage}
            }
           \centering \subfigure[]{
            \begin{minipage}[b]{0.3\textwidth}
               \centering
             \includegraphics[width=0.95\textwidth,height=0.9\textwidth]{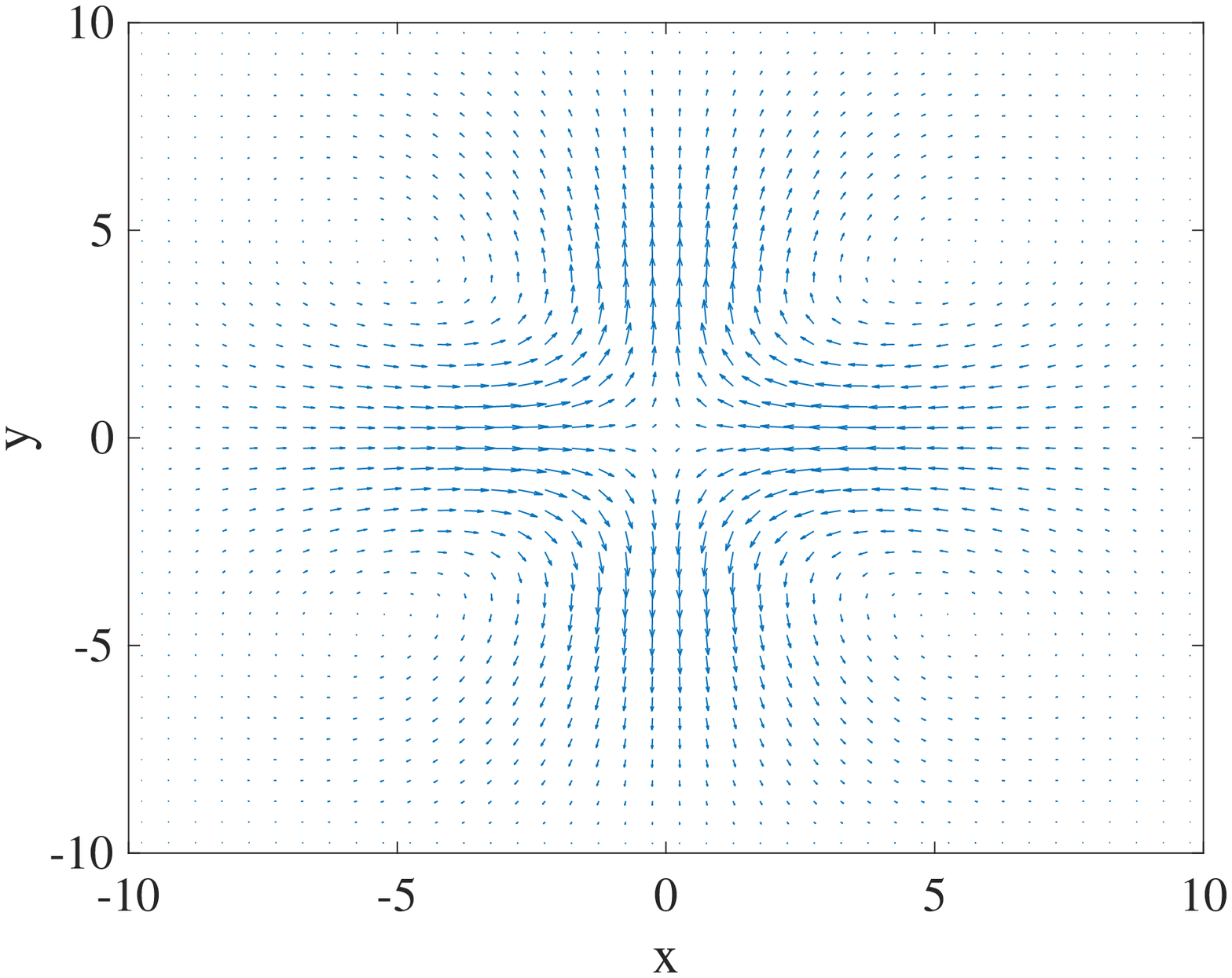}
            \end{minipage}
            }
            \centering \subfigure[]{
            \begin{minipage}[b]{0.3\textwidth}
               \centering
             \includegraphics[width=0.95\textwidth,height=0.9\textwidth]{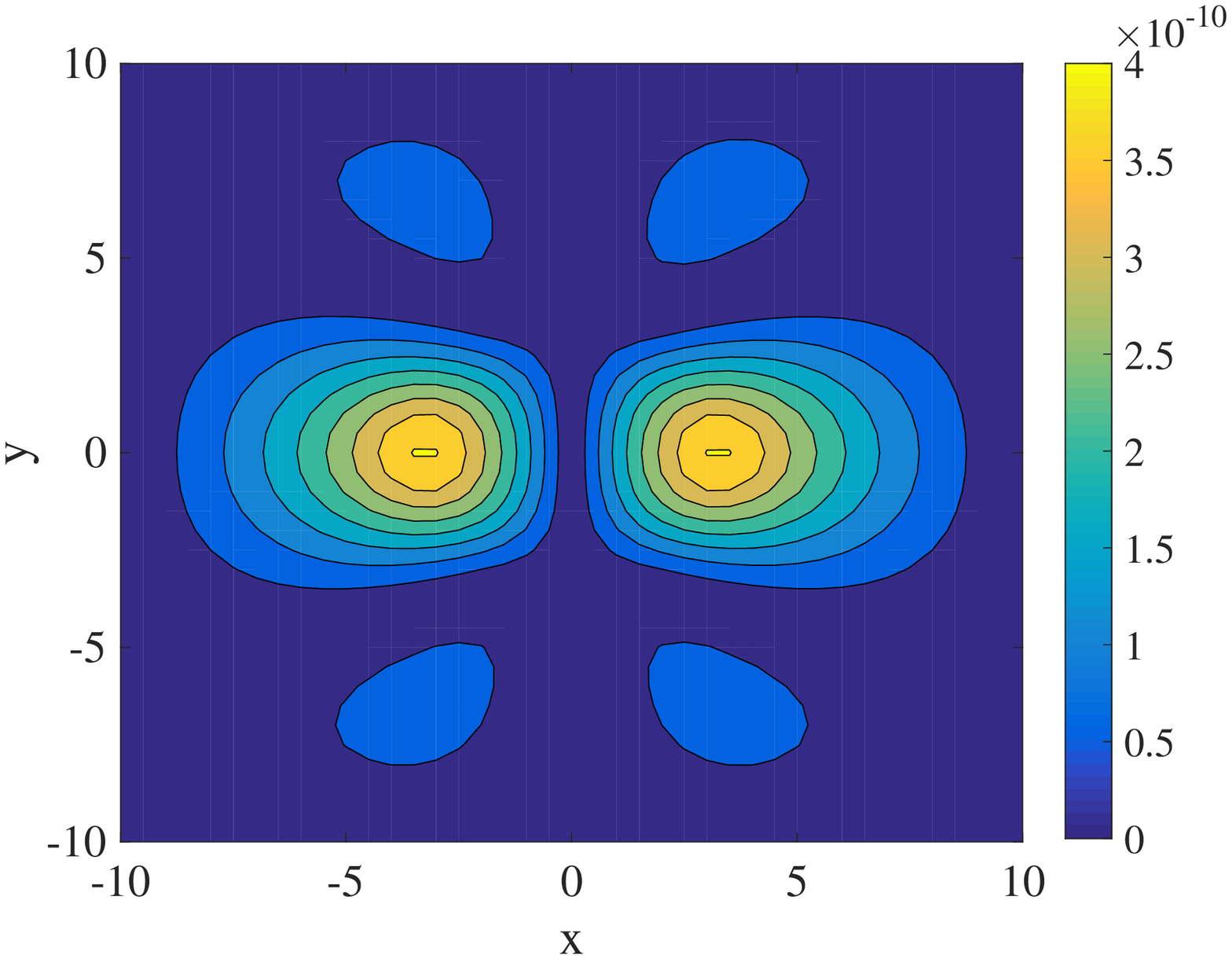}
            \end{minipage}
            }
            \centering \subfigure[]{
            \begin{minipage}[b]{0.3\textwidth}
            \centering
             \includegraphics[width=0.95\textwidth,height=0.9\textwidth]{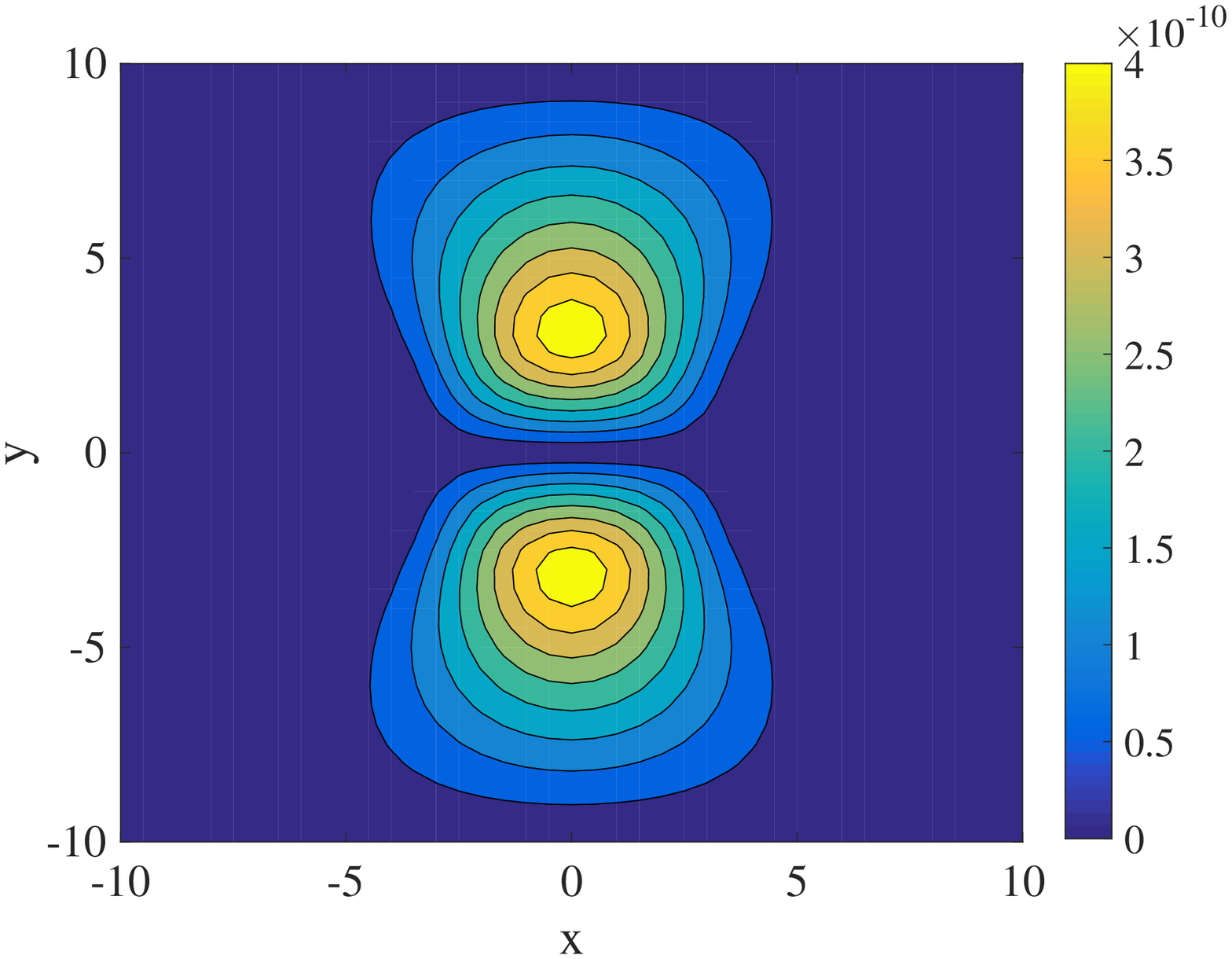}
            \end{minipage}
            }
            \caption{Example 2:   flow quivers (left column), magnitude contours of $x$-direction velocity component (center column), and magnitude contours of $y$-direction velocity component (right column)   at the  50th(top row) and 100th(bottom row) time step  respectively.}
            \label{EllipseCH4andnC10VelocityTemperature330K}
 \end{figure}

\subsection{Example 3:   bubble merging}
In this example, there are two bubbles at the initial moment, and the initial molar densities of methane and decane in gas and liquid phases are the same to those in Example 2.  The  temperature of the fluid keeps constant at 330K.  We simulate its dynamical evolution  for 120  time steps with  the time step size   $10^{-5}$s.

The total (free) energy dissipation profile with time steps is plotted   in Figure \ref{TwoEllipseCH4andnC10TotalEnergyTemperature330}(a),  and Figure \ref{TwoEllipseCH4andnC10TotalEnergyTemperature330}(b)  is a zoom-in plot in the   later time steps.  The results in Figure \ref{TwoEllipseCH4andnC10TotalEnergyTemperature330} confirm  that the total  energy  always decays  with time steps and consequently the proposed method is effective.

Figure \ref{TwoEllipseCH4andnC10MolarDensityOfCH4Temperature330K} and Figure \ref{TwoEllipseCH4andnC10MolarDensityOfnC10Temperature330K} depict the initial molar density configurations  and   molar density changes with time steps    for  methane and decane.  In Figures \ref{TwoEllipseCH4andnC10VelocityTemperature330K}, we illustrate  the velocity fields and  magnitudes of  velocity components at different time steps. 
 It can be observed from these simulated results that in the presence of gradients of chemical potentials,   the bubbles   is first emerging  with each other in the sequent dynamical evolution, and at the finial  time,  the merged  bubbles gradually  tend  to a circle.

\begin{figure}
            \centering \subfigure[]{
            \begin{minipage}[b]{0.45\textwidth}
               \centering
             \includegraphics[width=0.95\textwidth,height=2in]{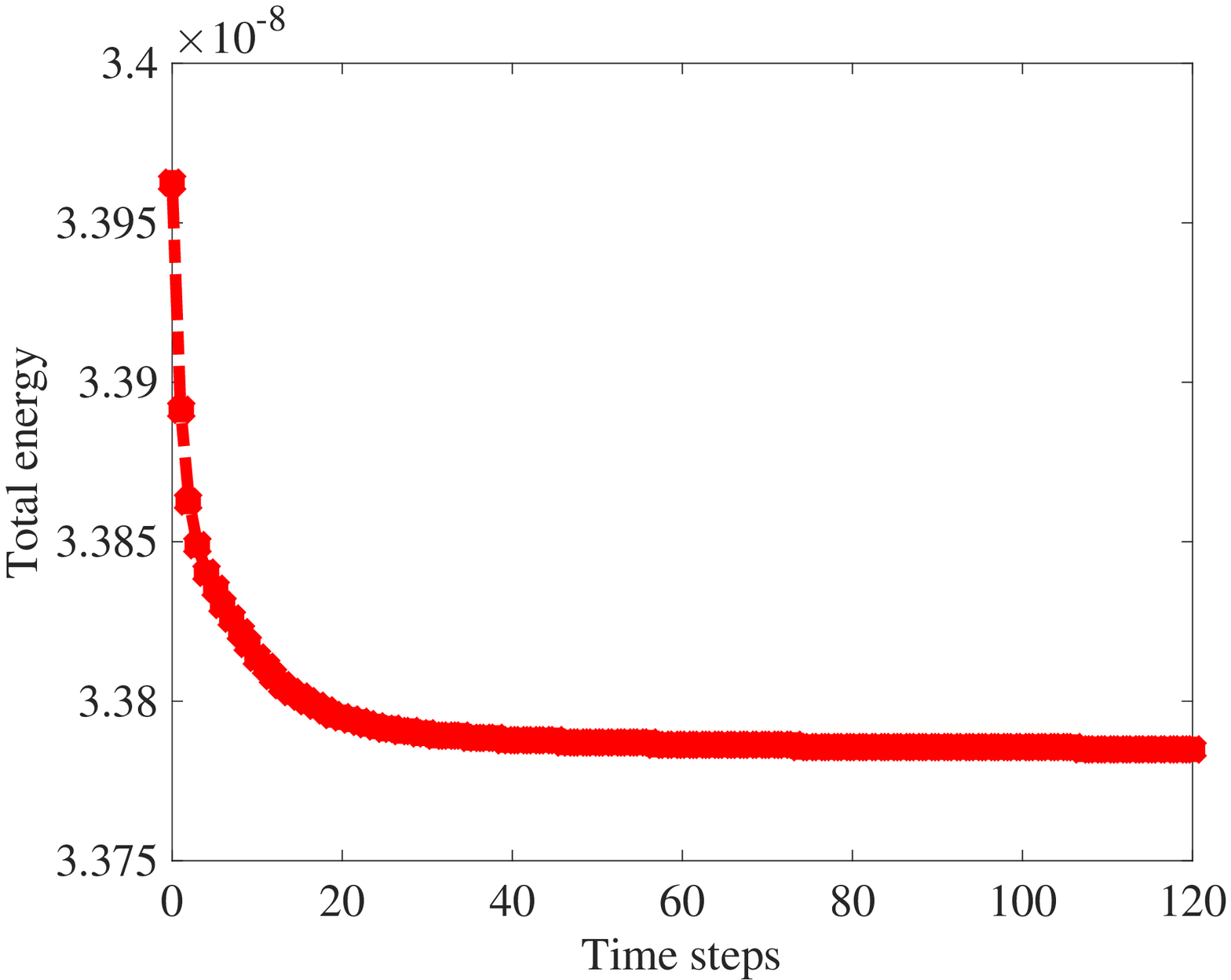}
            \end{minipage}
            }
            \centering \subfigure[]{
            \begin{minipage}[b]{0.45\textwidth}
            \centering
             \includegraphics[width=0.95\textwidth,height=2in]{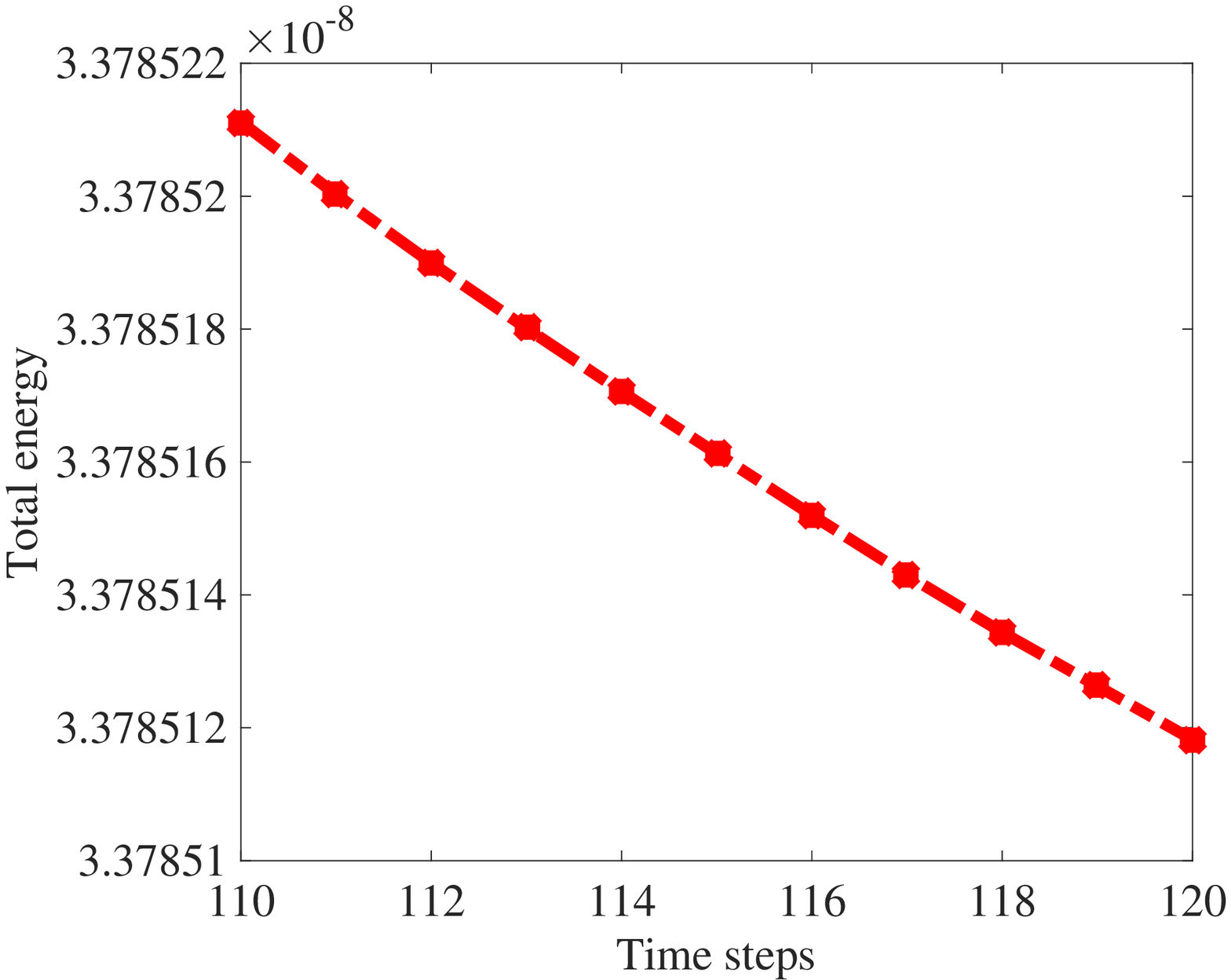}
            \end{minipage}
            }
           \caption{Example 3: energy dissipation with time steps.}
            \label{TwoEllipseCH4andnC10TotalEnergyTemperature330}
 \end{figure}

\begin{figure}
            \centering \subfigure[]{
            \begin{minipage}[b]{0.3\textwidth}
               \centering
             \includegraphics[width=0.95\textwidth,height=0.9\textwidth]{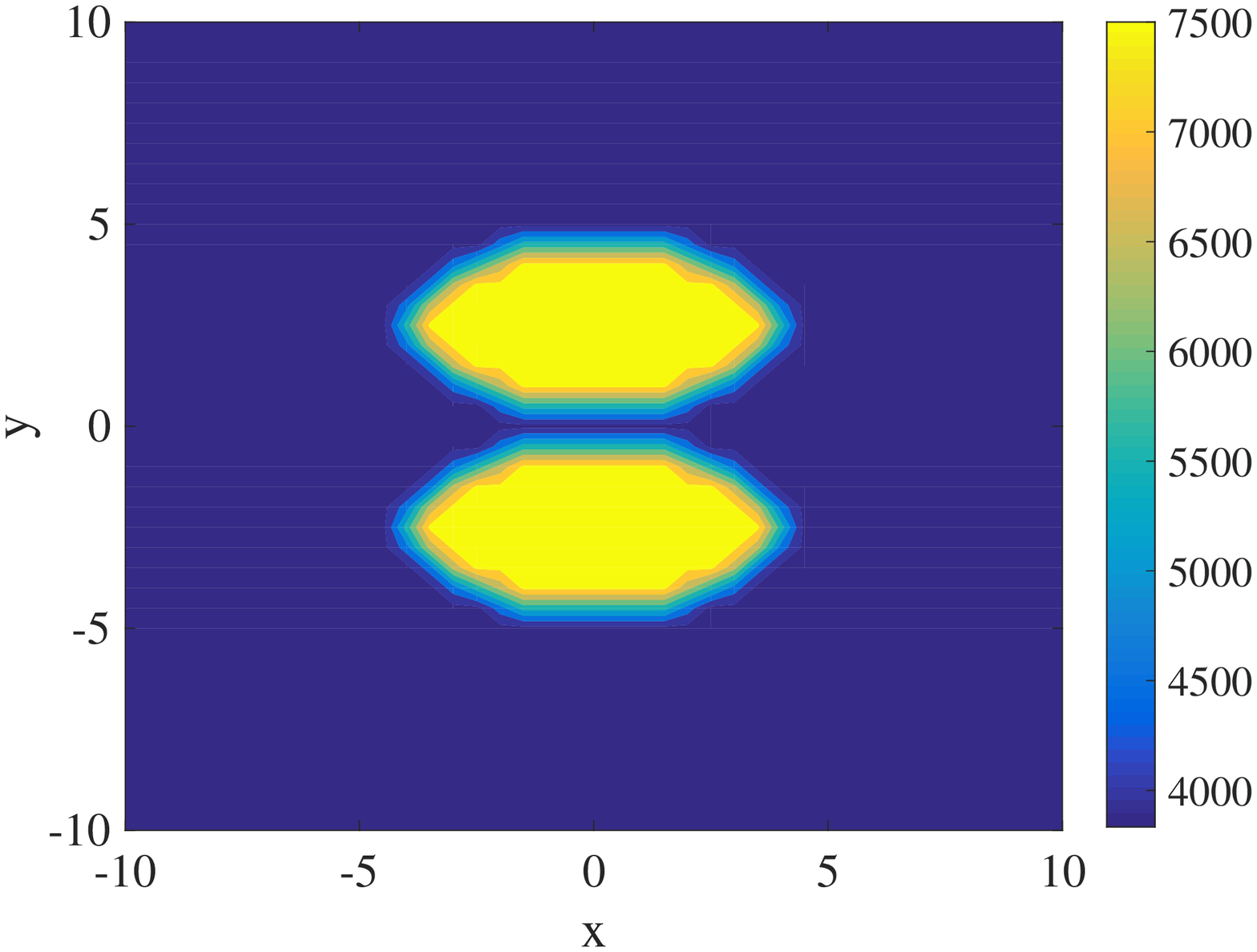}
            \end{minipage}
            }
            \centering \subfigure[]{
            \begin{minipage}[b]{0.3\textwidth}
            \centering
             \includegraphics[width=0.95\textwidth,height=0.9\textwidth]{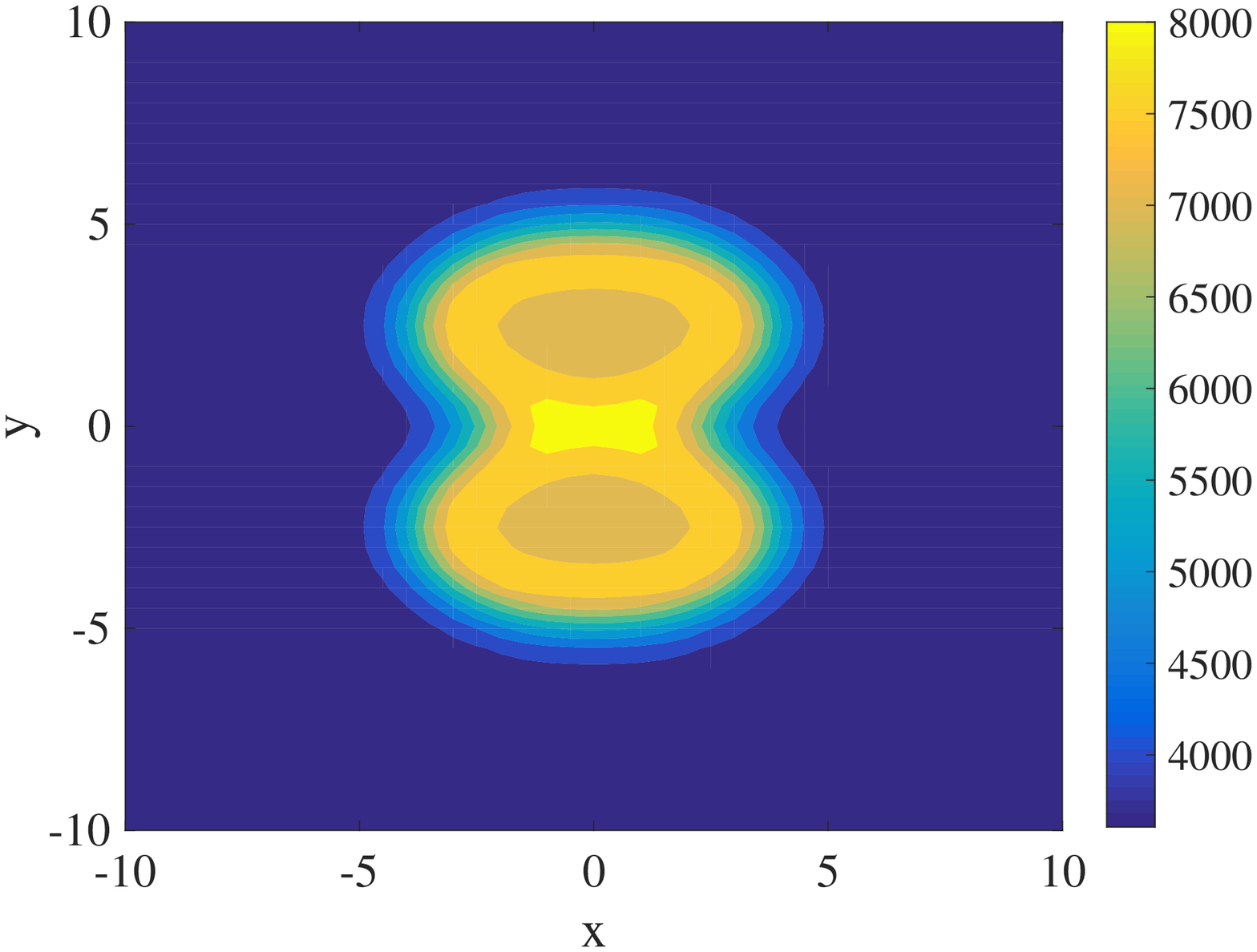}
            \end{minipage}
            }
           \centering \subfigure[]{
            \begin{minipage}[b]{0.3\textwidth}
            \centering
             \includegraphics[width=0.95\textwidth,height=0.9\textwidth]{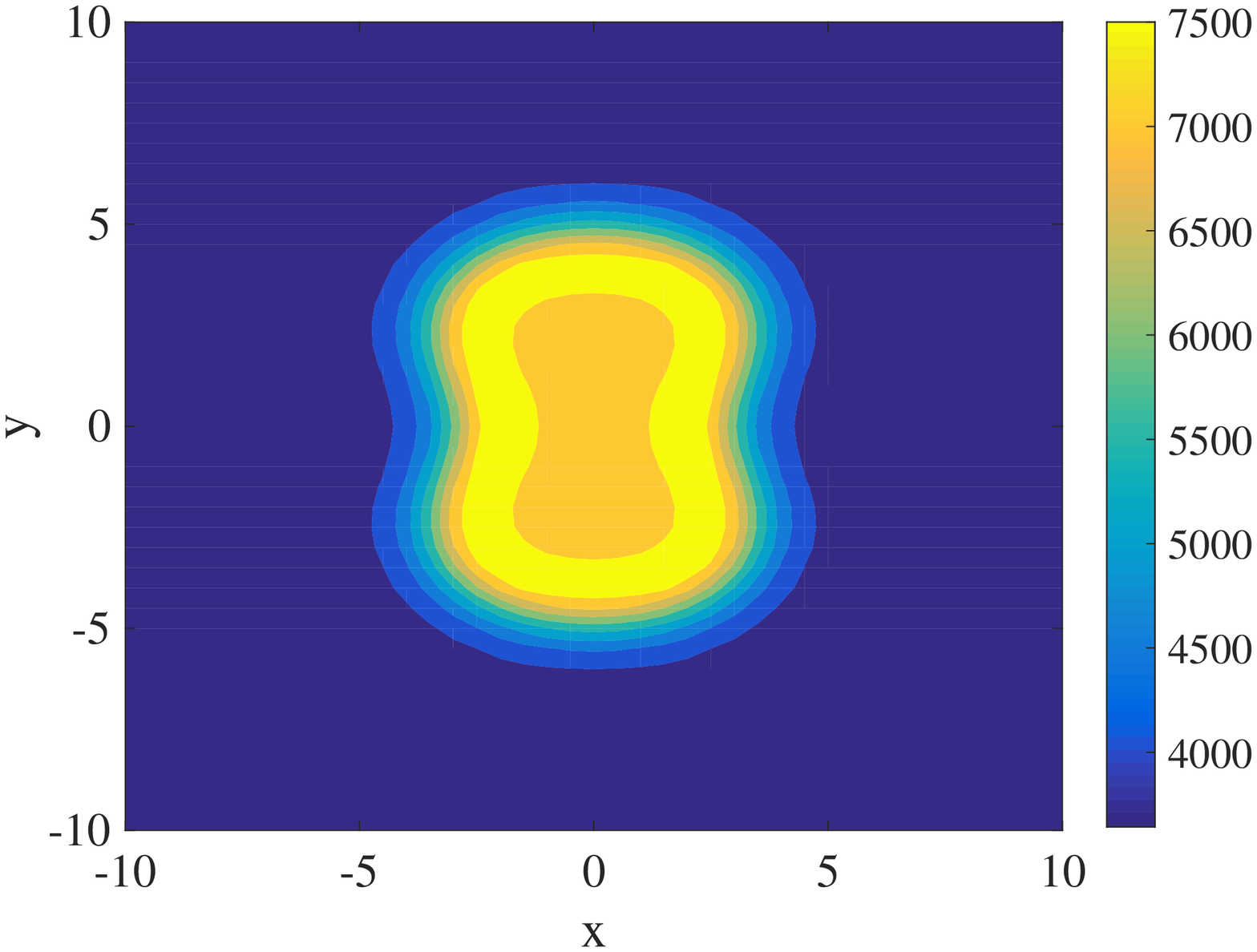}
            \end{minipage}
            }
           \centering \subfigure[]{
            \begin{minipage}[b]{0.3\textwidth}
               \centering
             \includegraphics[width=0.95\textwidth,height=0.9\textwidth]{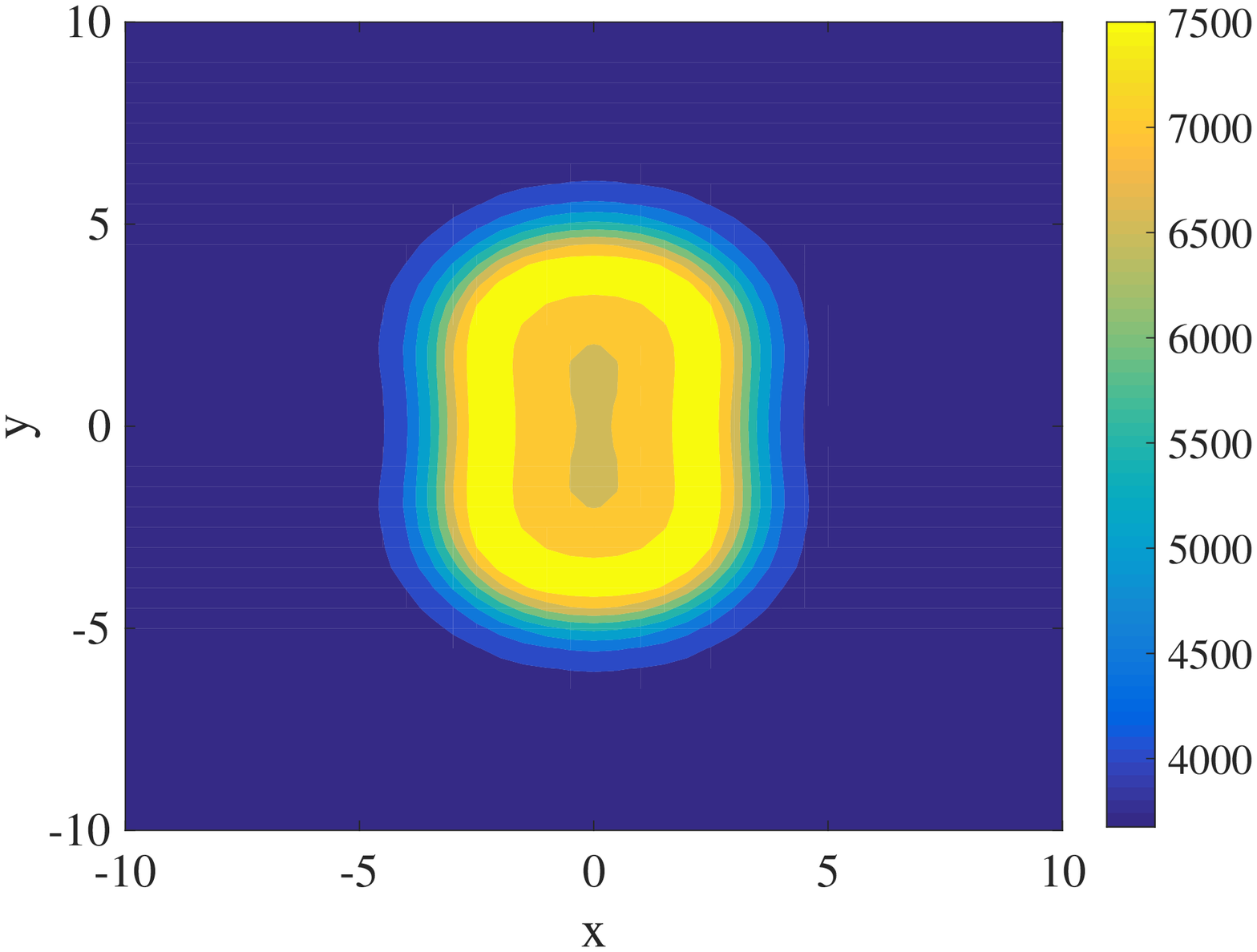}
            \end{minipage}
            }
            \centering \subfigure[]{
            \begin{minipage}[b]{0.3\textwidth}
            \centering
             \includegraphics[width=0.95\textwidth,height=0.9\textwidth]{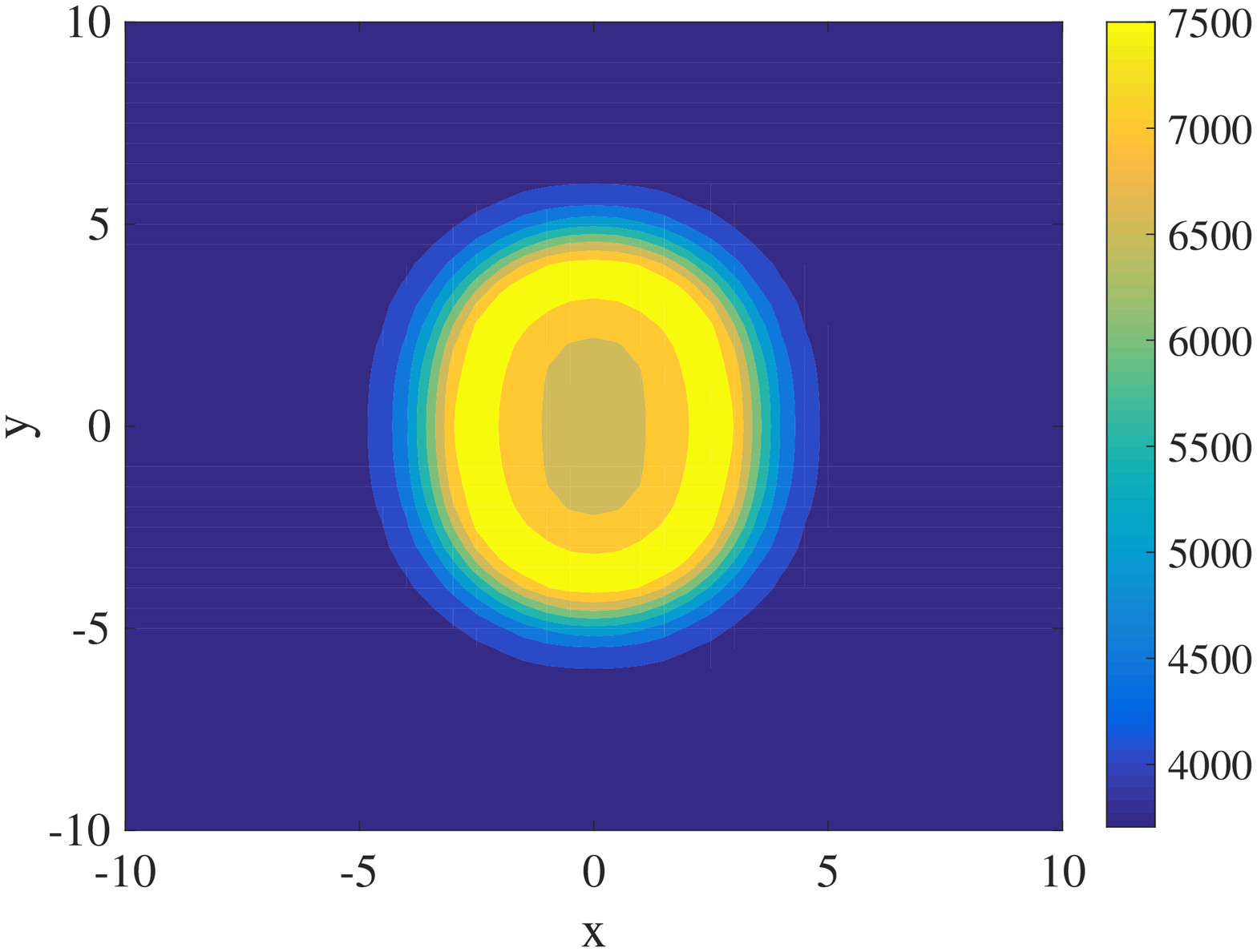}
            \end{minipage}
            }
           \centering \subfigure[]{
            \begin{minipage}[b]{0.3\textwidth}
            \centering
             \includegraphics[width=0.95\textwidth,height=0.9\textwidth]{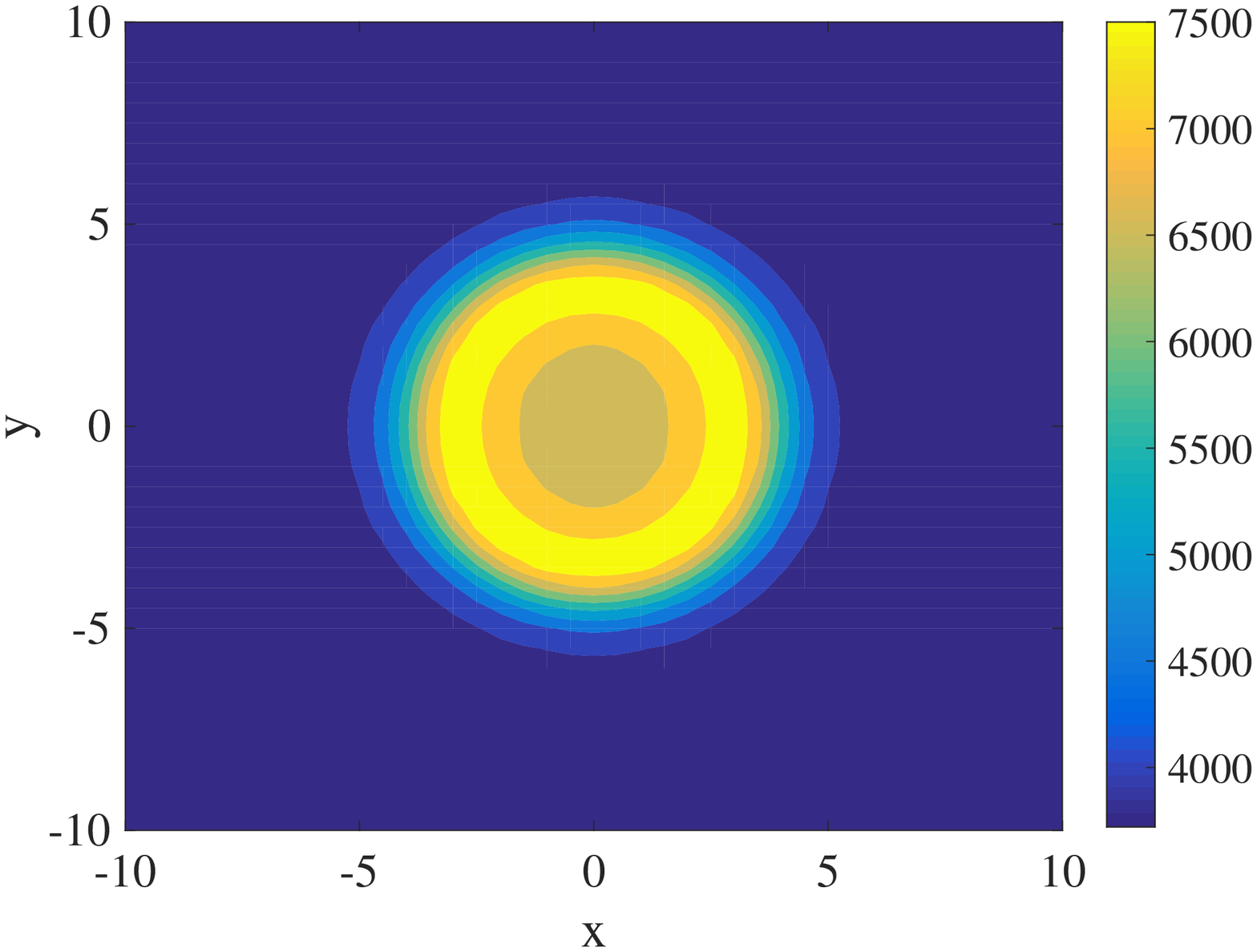}
            \end{minipage}
            }
           \caption{Example 3: CH$_4$ molar densities at  the the initial(a), 10th(b), 20th(c), 30th(d), 50th(e) and 120th(f)   time step  respectively.}
            \label{TwoEllipseCH4andnC10MolarDensityOfCH4Temperature330K}
 \end{figure}

\begin{figure}
            \centering \subfigure[]{
            \begin{minipage}[b]{0.3\textwidth}
               \centering
             \includegraphics[width=0.95\textwidth,height=0.9\textwidth]{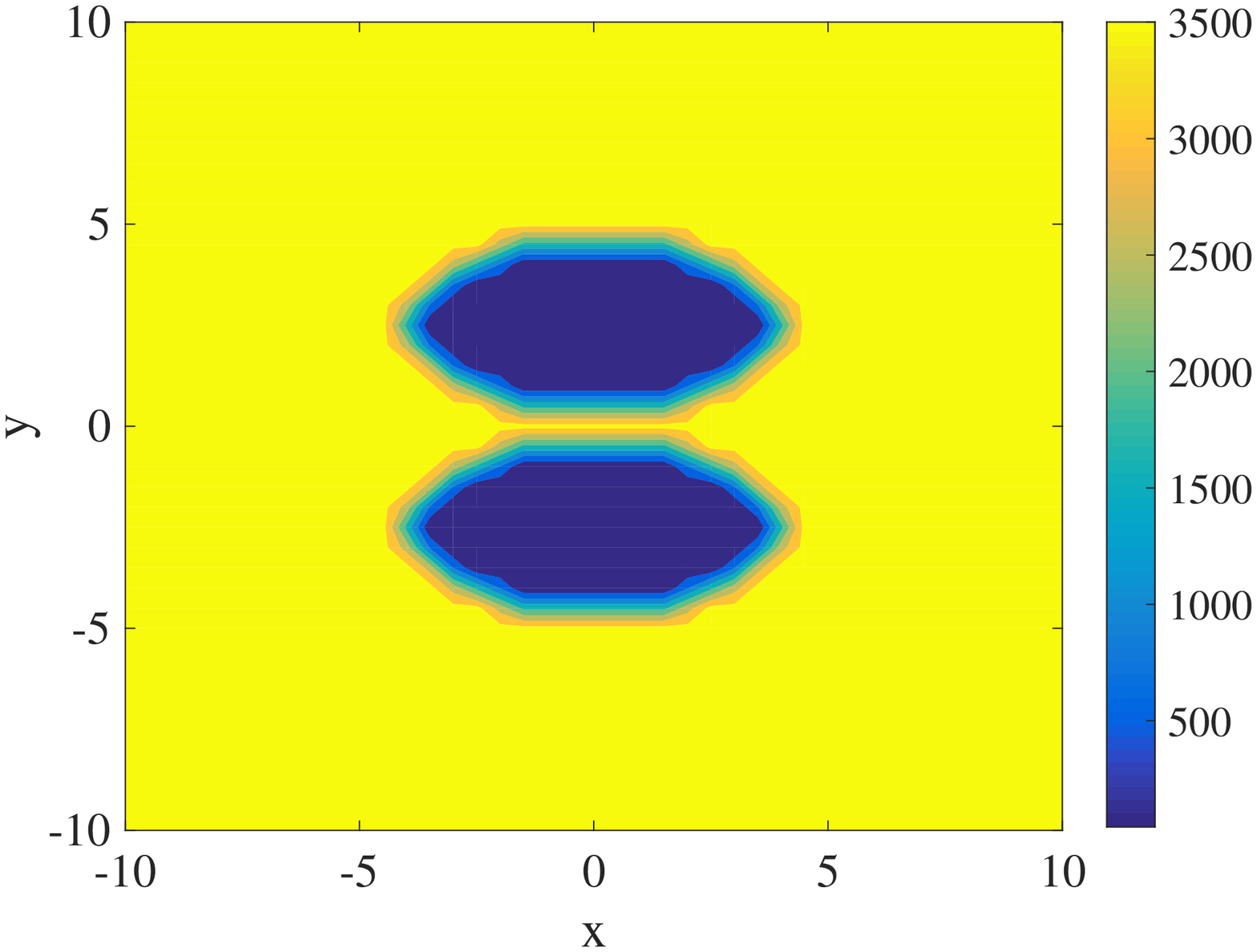}
            \end{minipage}
            }
            \centering \subfigure[]{
            \begin{minipage}[b]{0.3\textwidth}
            \centering
             \includegraphics[width=0.95\textwidth,height=0.9\textwidth]{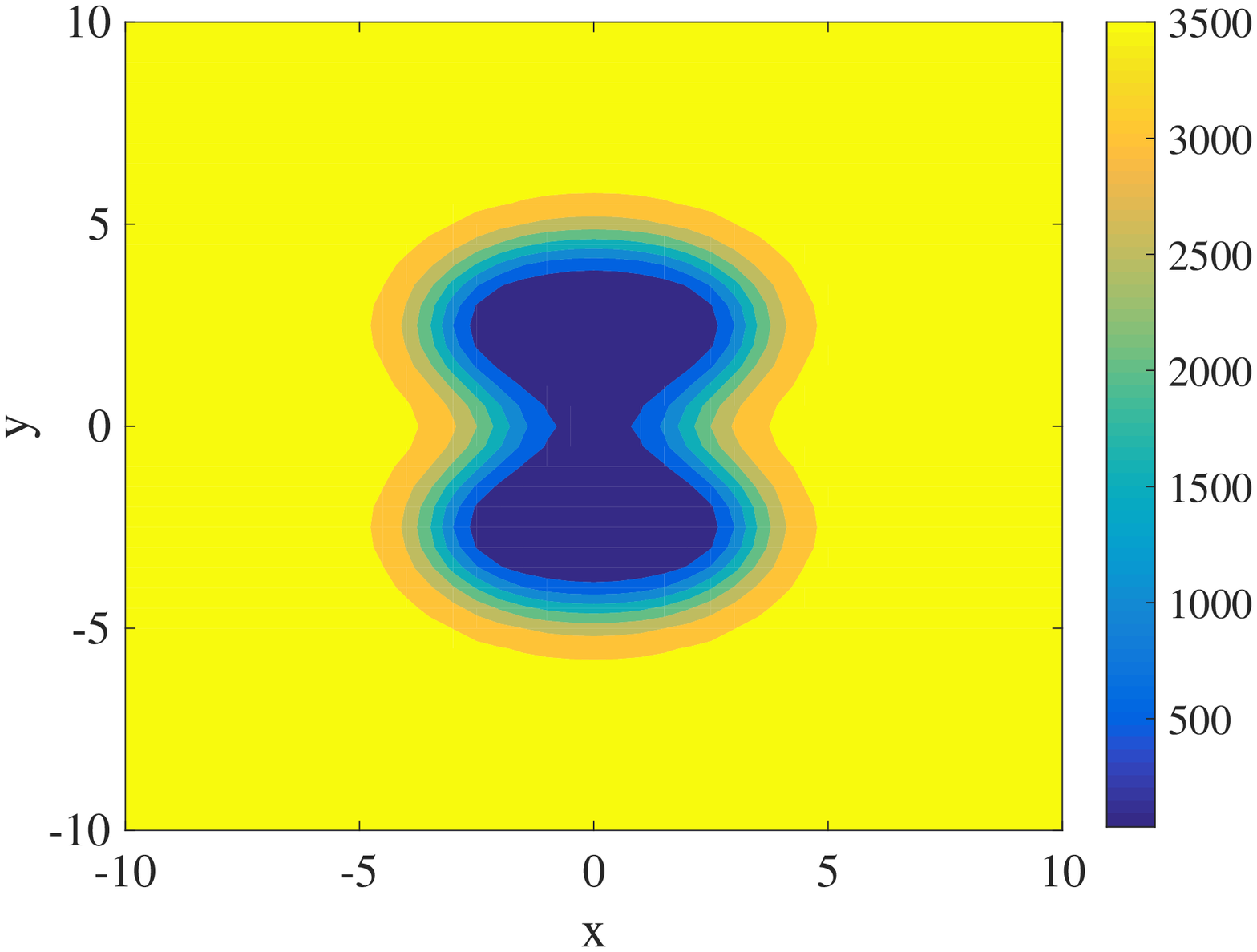}
            \end{minipage}
            }
           \centering \subfigure[]{
            \begin{minipage}[b]{0.3\textwidth}
            \centering
             \includegraphics[width=0.95\textwidth,height=0.9\textwidth]{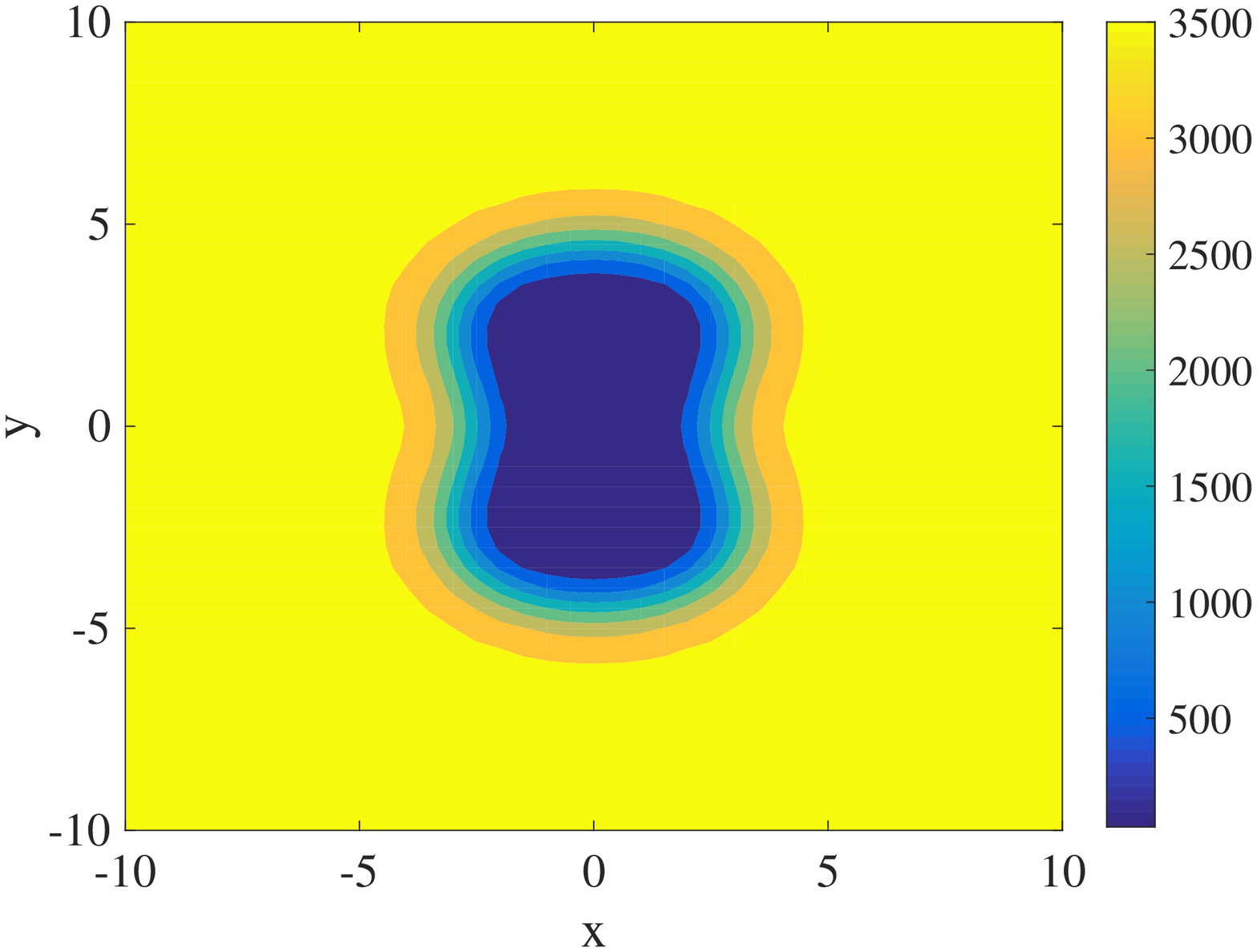}
            \end{minipage}
            }
           \centering \subfigure[]{
            \begin{minipage}[b]{0.3\textwidth}
               \centering
             \includegraphics[width=0.95\textwidth,height=0.9\textwidth]{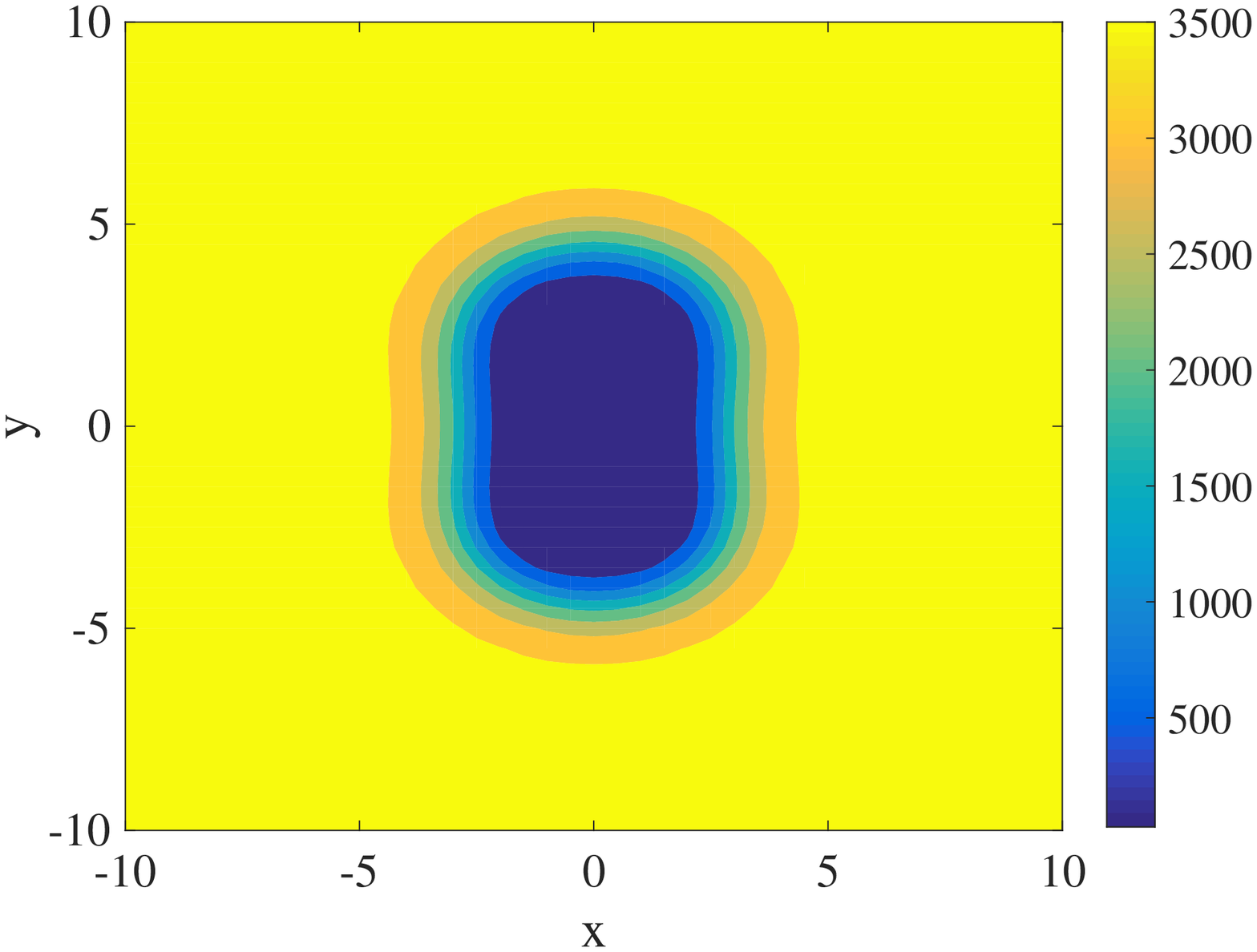}
            \end{minipage}
            }
            \centering \subfigure[]{
            \begin{minipage}[b]{0.3\textwidth}
            \centering
             \includegraphics[width=0.95\textwidth,height=0.9\textwidth]{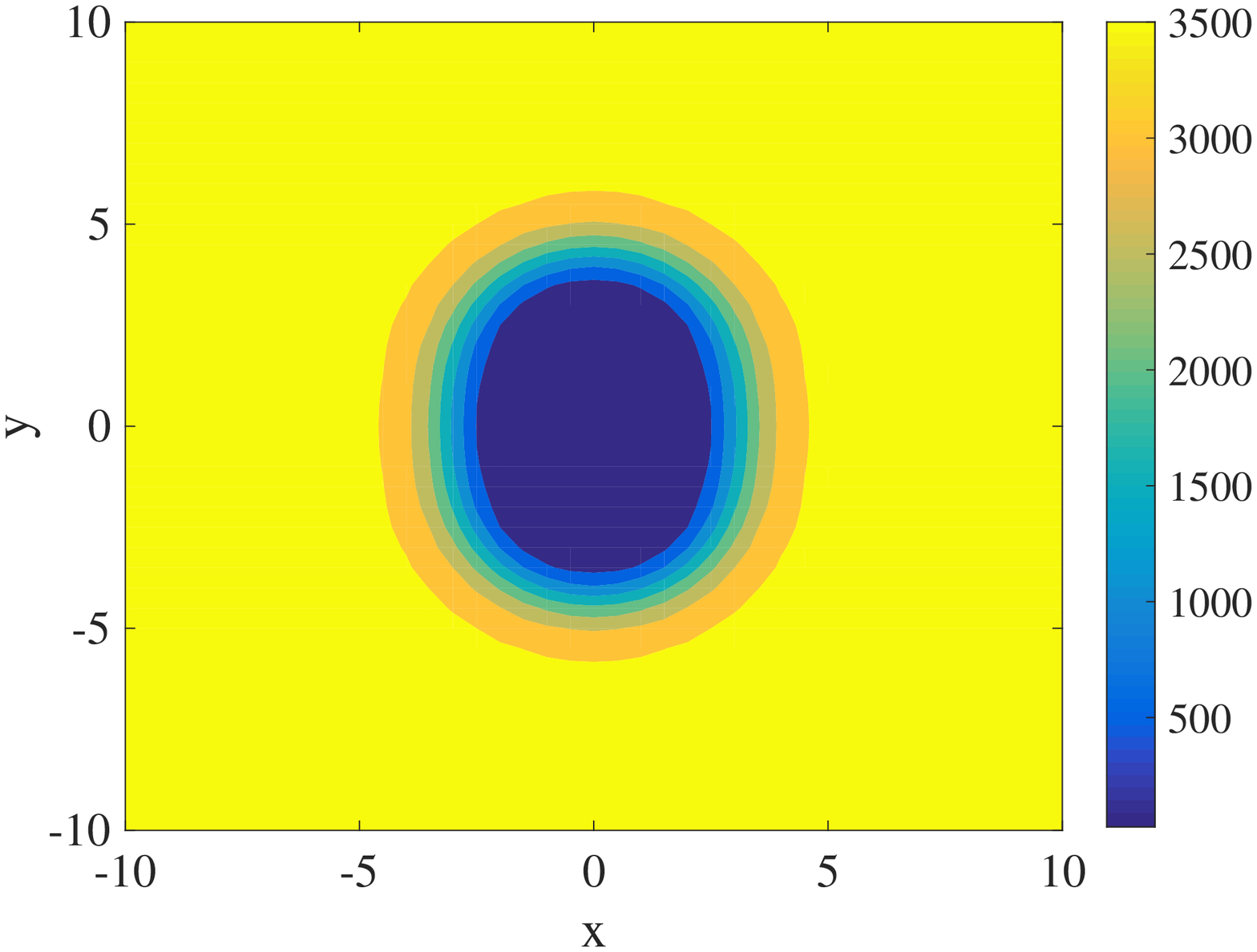}
            \end{minipage}
            }
           \centering \subfigure[]{
            \begin{minipage}[b]{0.3\textwidth}
            \centering
             \includegraphics[width=0.95\textwidth,height=0.9\textwidth]{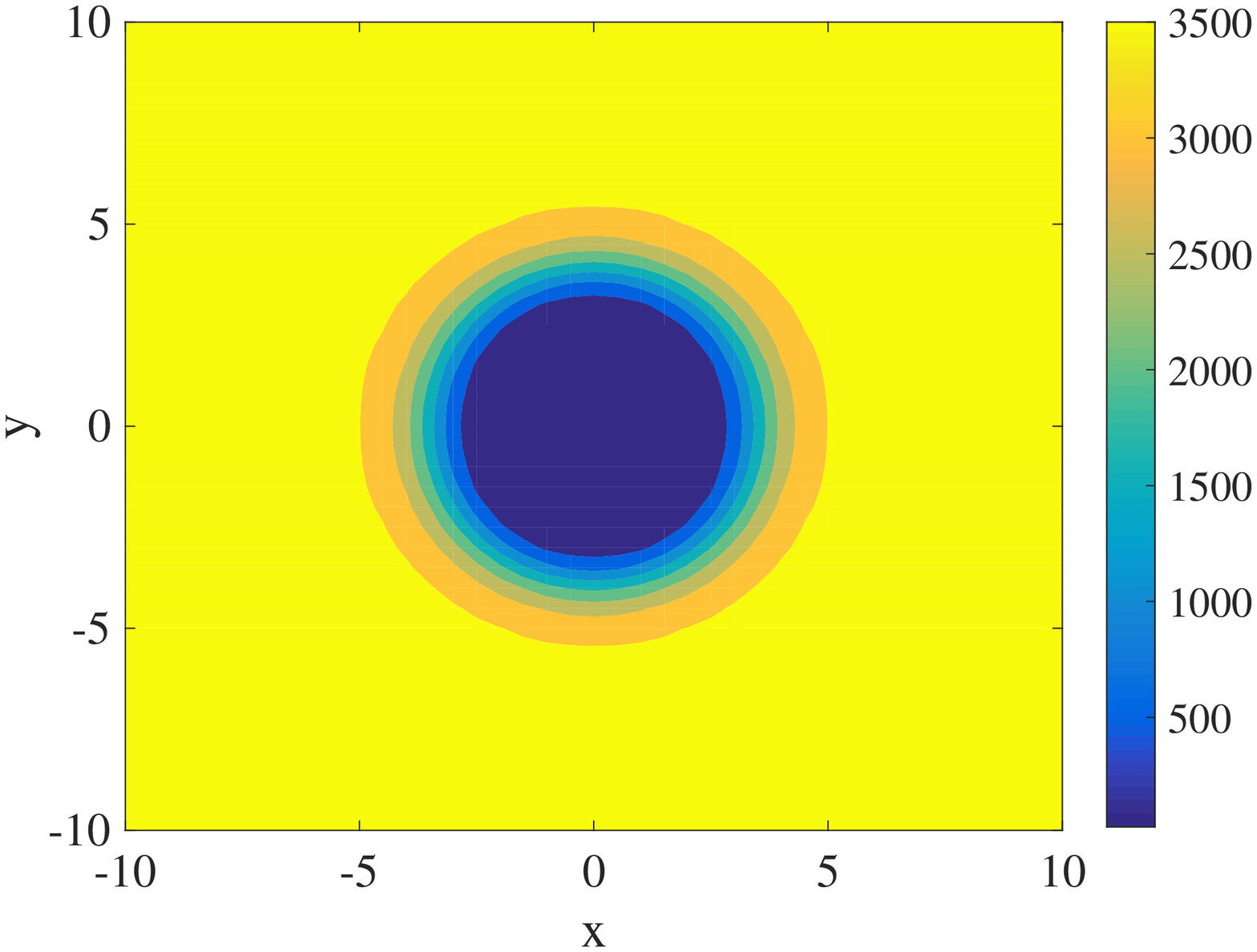}
            \end{minipage}
            }
           \caption{Example 3: nC$_{10}$ molar densities at  the initial(a),  10th(b), 20th(c), 30th(d), 50th(e) and 120th(f)   time step  respectively.}
            \label{TwoEllipseCH4andnC10MolarDensityOfnC10Temperature330K}
 \end{figure}

\begin{figure}
            \centering \subfigure[]{
            \begin{minipage}[b]{0.3\textwidth}
               \centering
             \includegraphics[width=0.95\textwidth,height=0.9\textwidth]{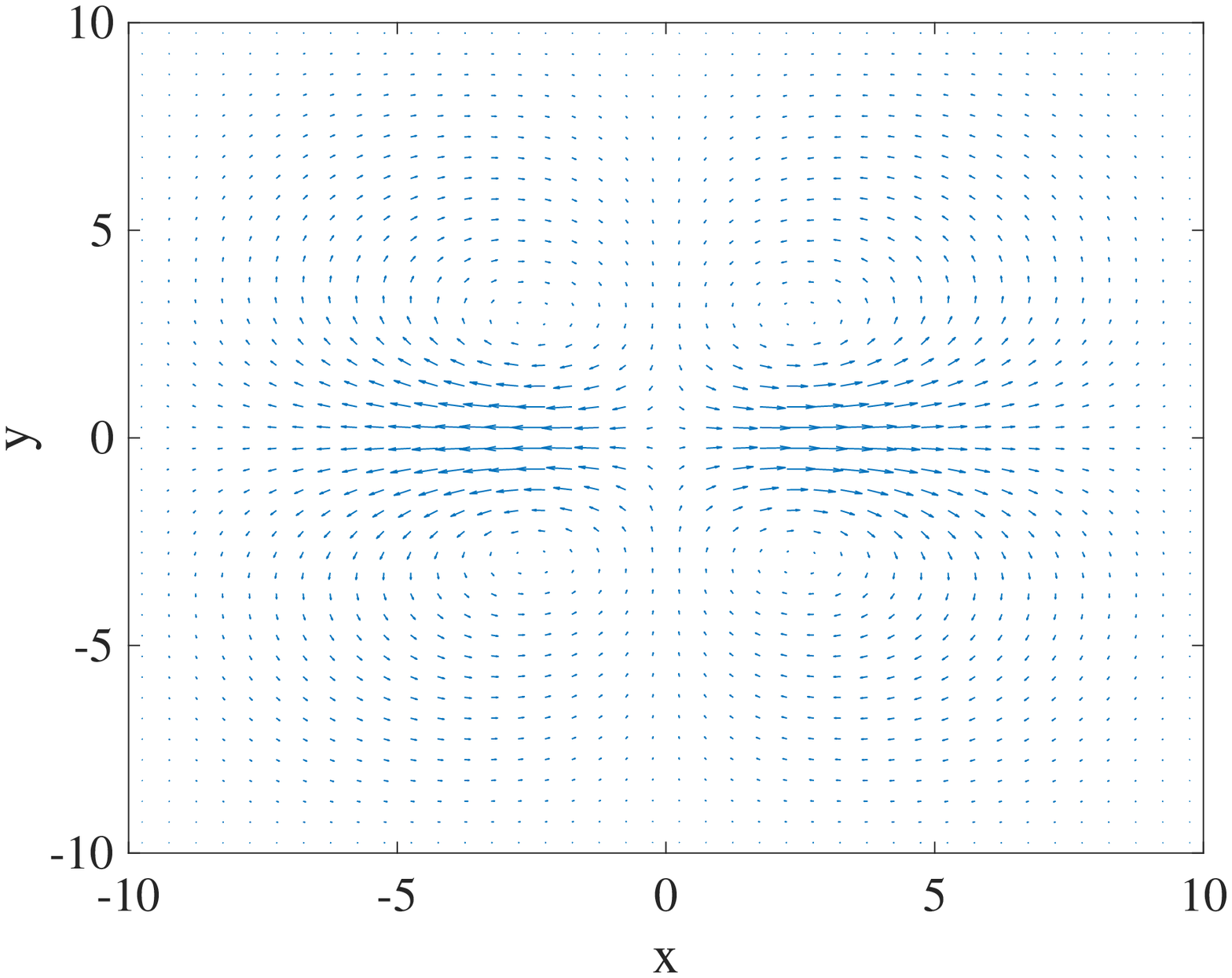}
            \end{minipage}
            }
            \centering \subfigure[]{
            \begin{minipage}[b]{0.3\textwidth}
               \centering
             \includegraphics[width=0.95\textwidth,height=0.9\textwidth]{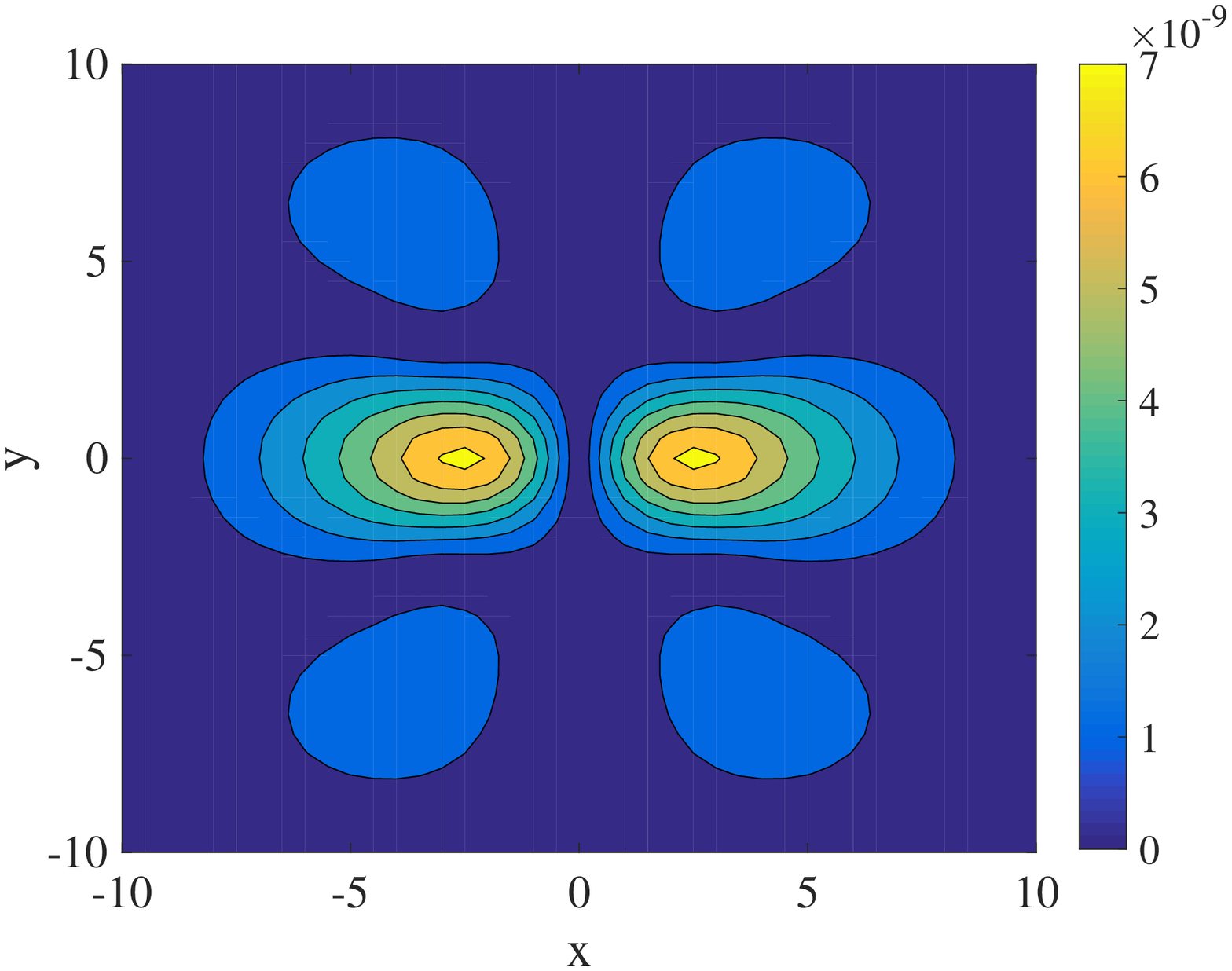}
            \end{minipage}
            }
            \centering \subfigure[]{
            \begin{minipage}[b]{0.3\textwidth}
            \centering
             \includegraphics[width=0.95\textwidth,height=0.9\textwidth]{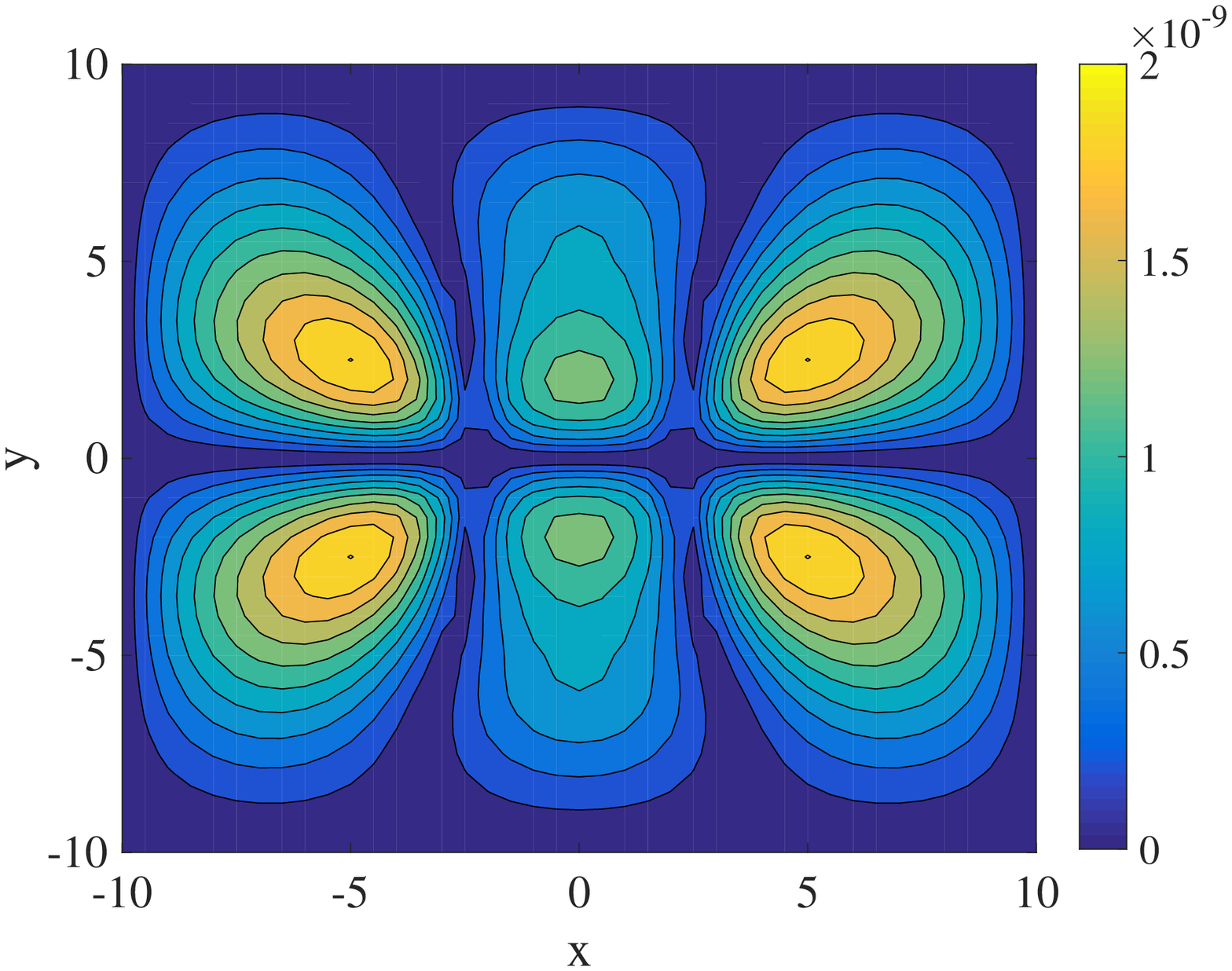}
            \end{minipage}
            }
           \centering \subfigure[]{
            \begin{minipage}[b]{0.3\textwidth}
               \centering
             \includegraphics[width=0.95\textwidth,height=0.9\textwidth]{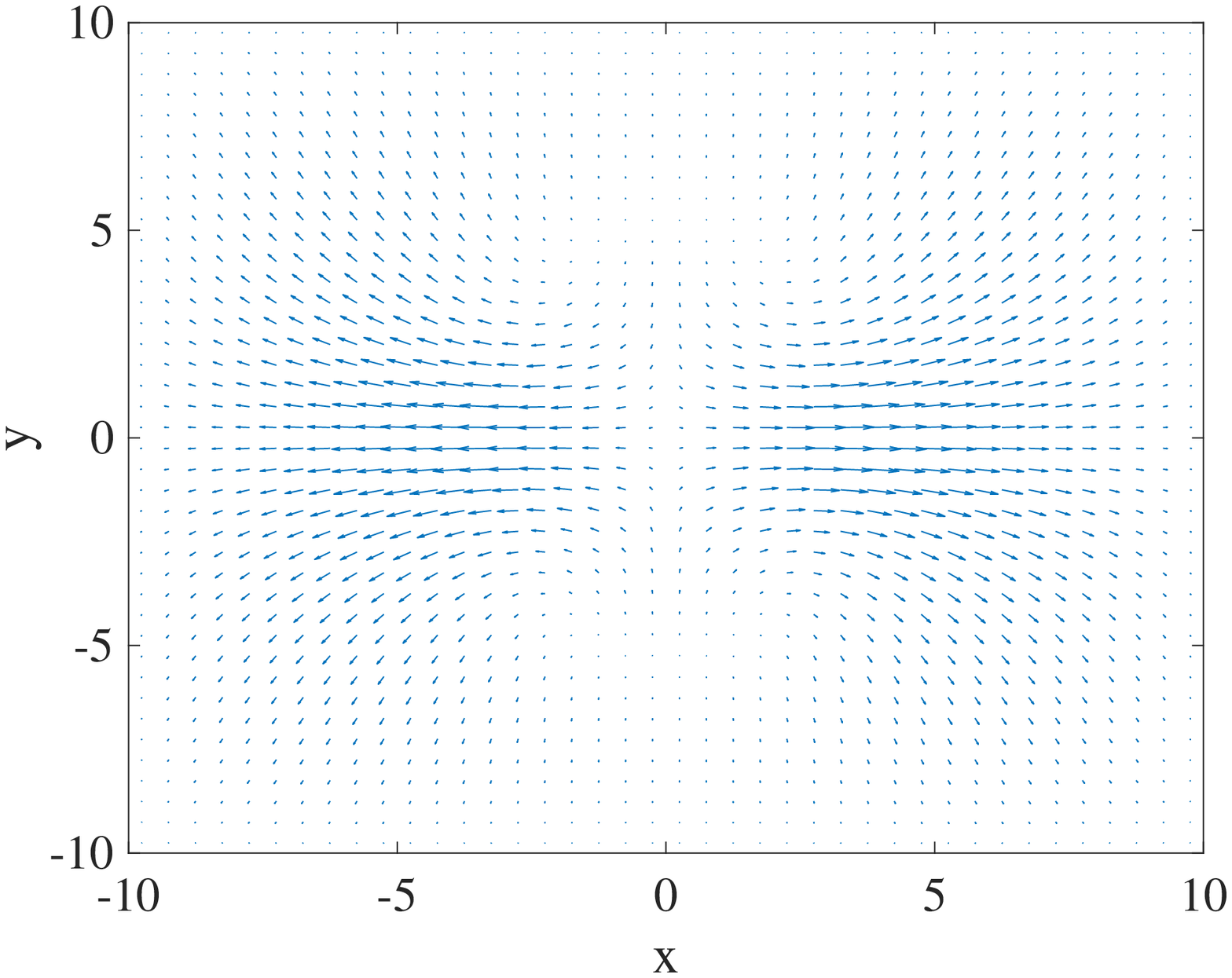}
            \end{minipage}
            }
            \centering \subfigure[]{
            \begin{minipage}[b]{0.3\textwidth}
               \centering
             \includegraphics[width=0.95\textwidth,height=0.9\textwidth]{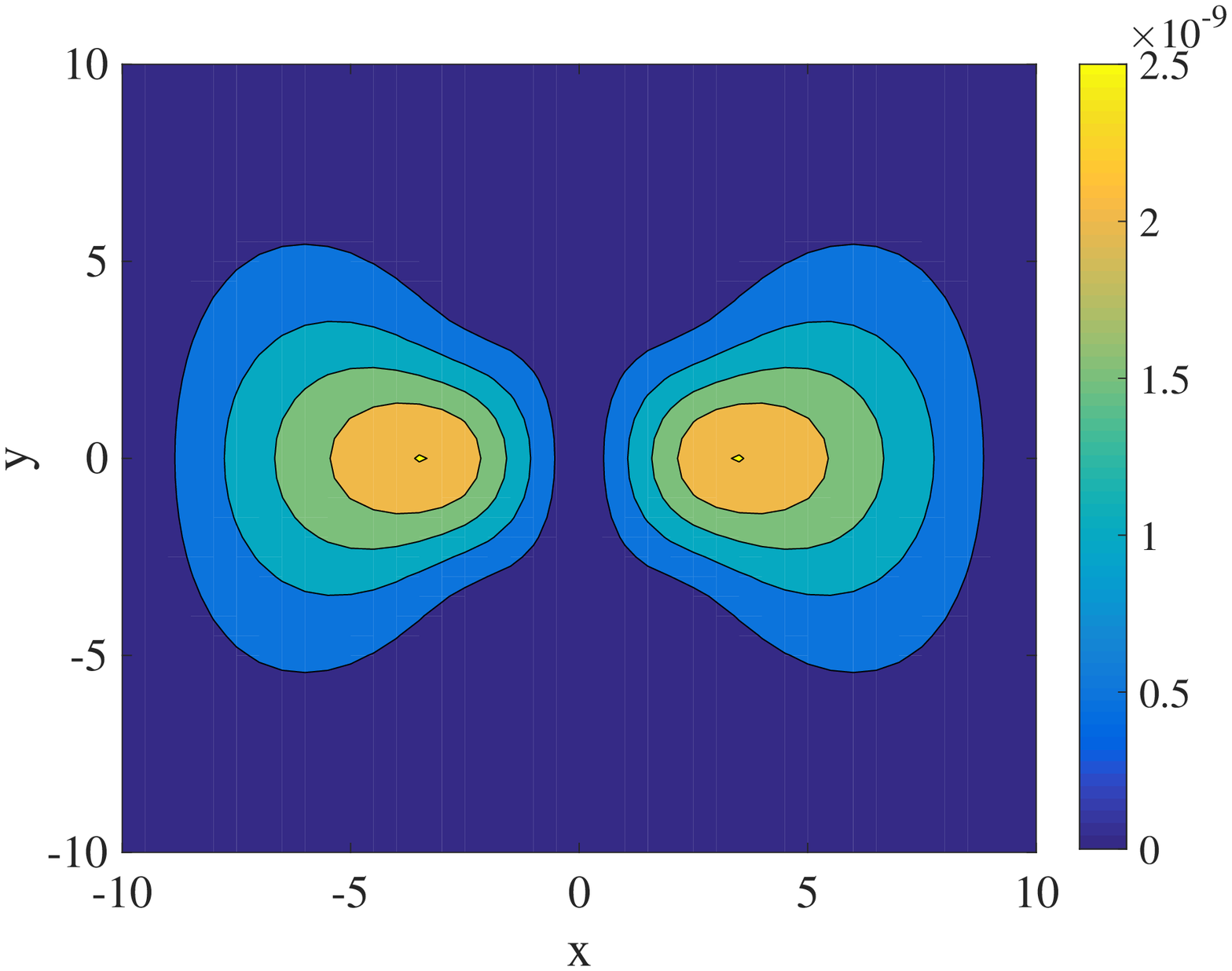}
            \end{minipage}
            }
            \centering \subfigure[]{
            \begin{minipage}[b]{0.3\textwidth}
            \centering
             \includegraphics[width=0.95\textwidth,height=0.9\textwidth]{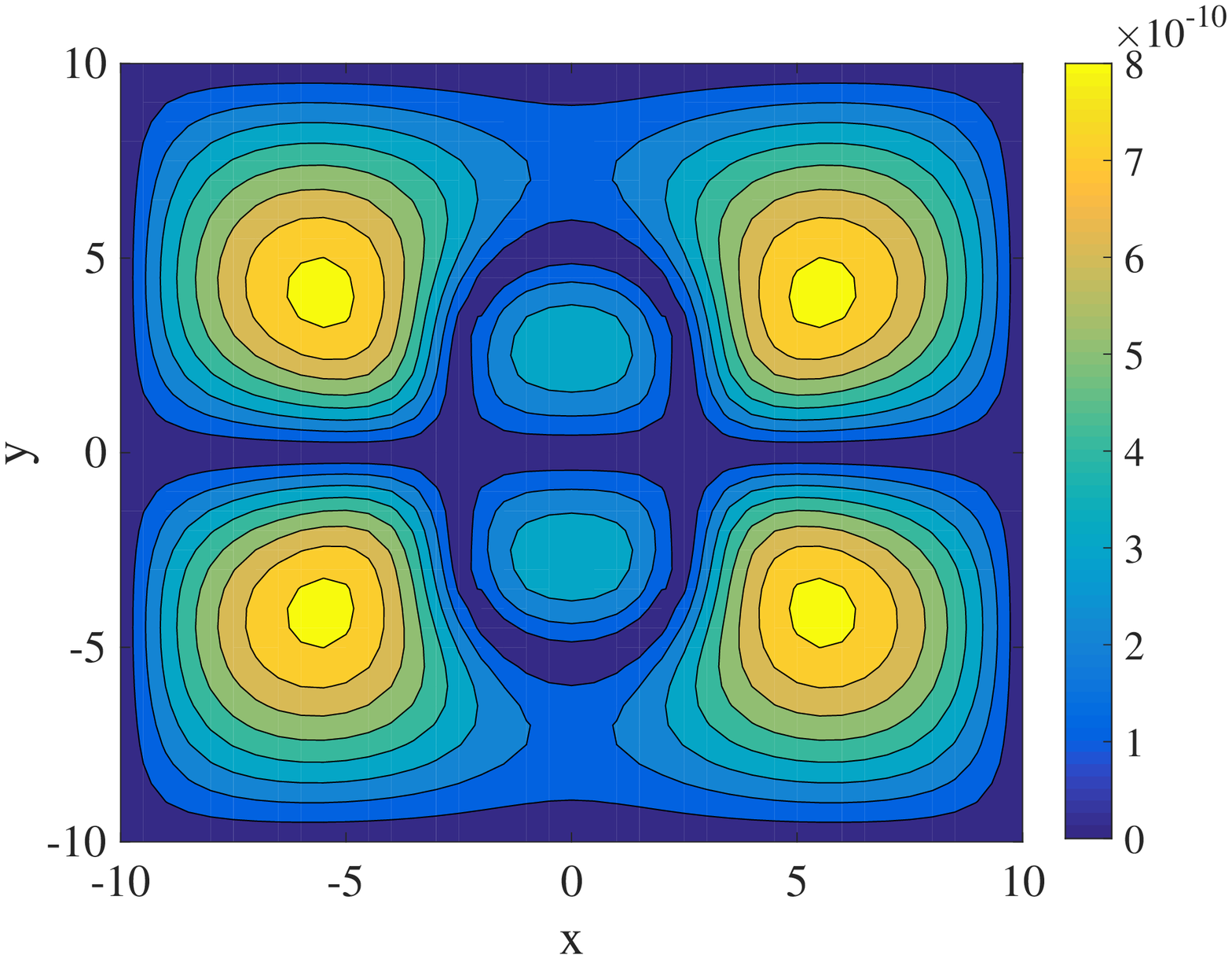}
            \end{minipage}
            }
           \centering \subfigure[]{
            \begin{minipage}[b]{0.3\textwidth}
               \centering
             \includegraphics[width=0.95\textwidth,height=0.9\textwidth]{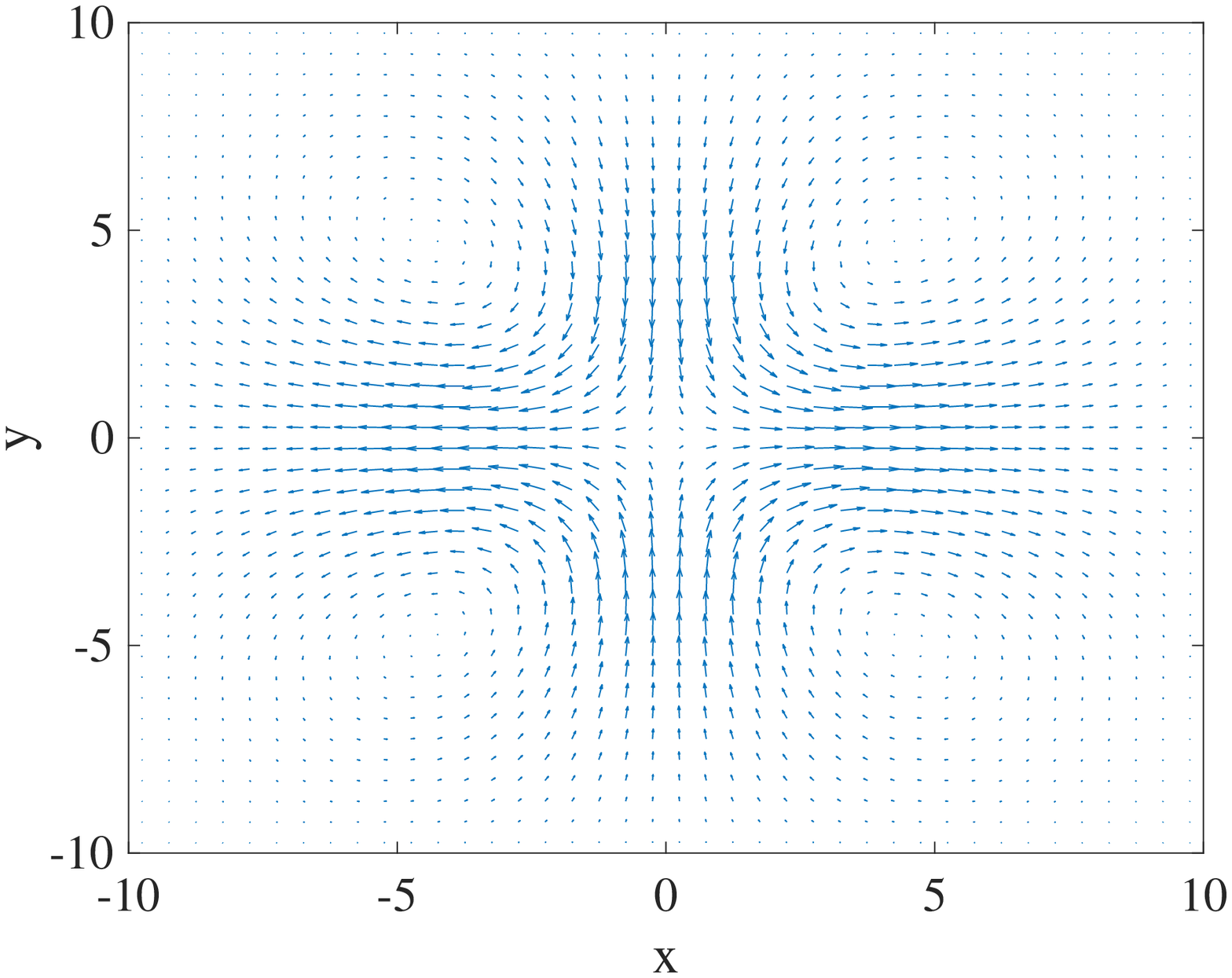}
            \end{minipage}
            }
            \centering \subfigure[]{
            \begin{minipage}[b]{0.3\textwidth}
               \centering
             \includegraphics[width=0.95\textwidth,height=0.9\textwidth]{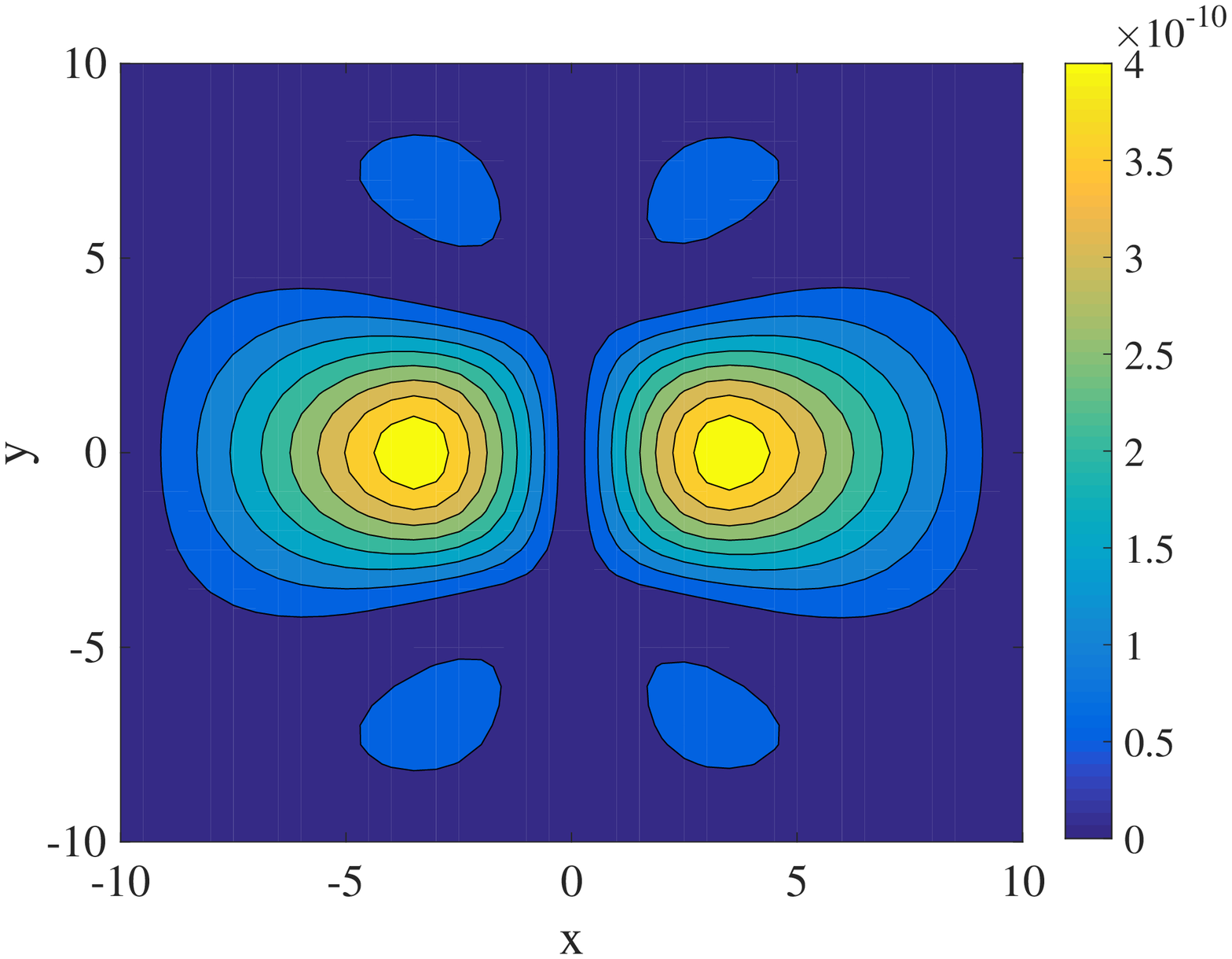}
            \end{minipage}
            }
            \centering \subfigure[]{
            \begin{minipage}[b]{0.3\textwidth}
            \centering
             \includegraphics[width=0.95\textwidth,height=0.9\textwidth]{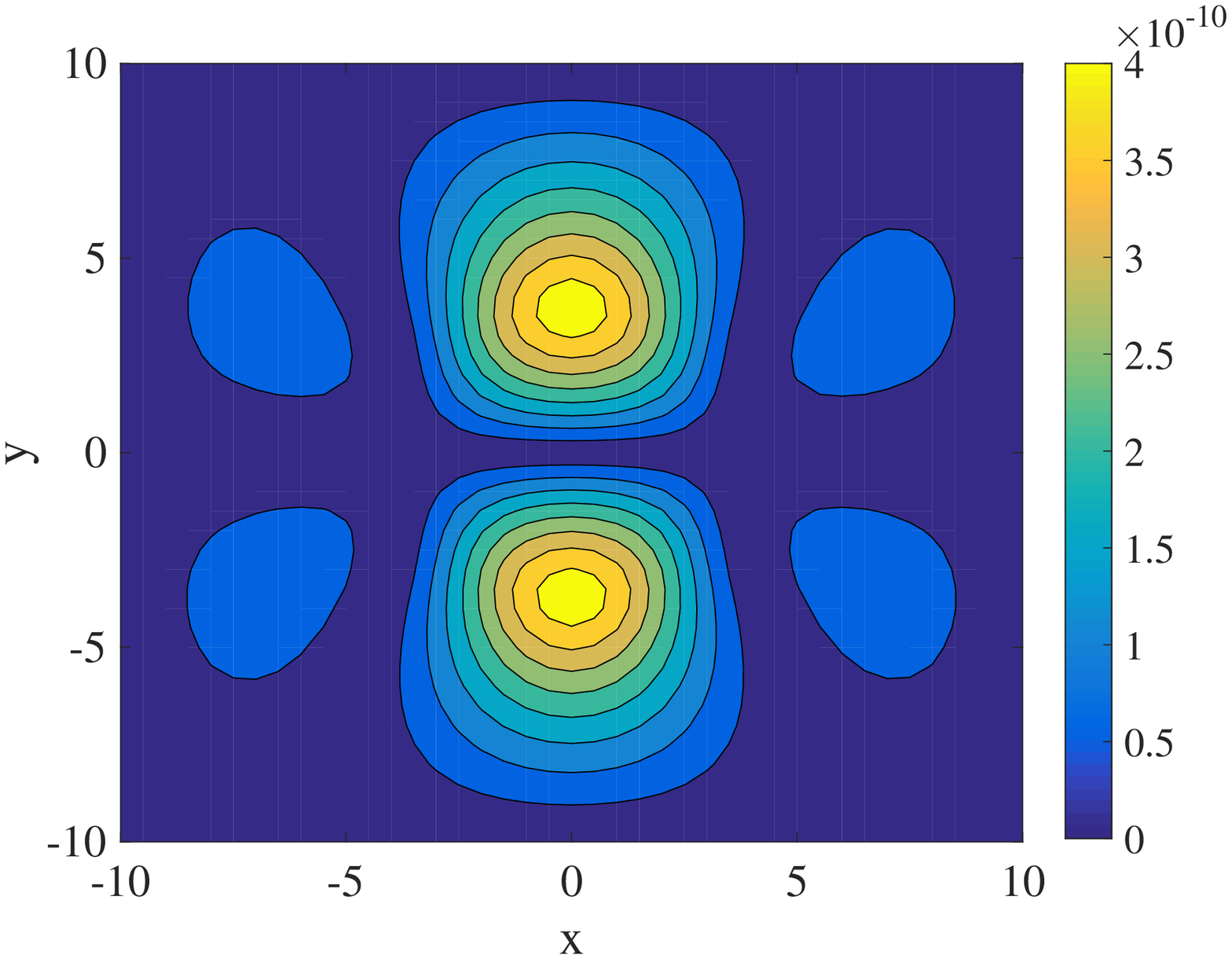}
            \end{minipage}
            }
           \caption{Example 3:   flow quivers (left column), magnitude contours of $x$-direction velocity component (center column), and magnitude contours of $y$-direction velocity component (right column)   at the 10th(top row), 30th(center row), and 120th(bottom row) time step  respectively.}
            \label{TwoEllipseCH4andnC10VelocityTemperature330K}
 \end{figure}

\section{Conclusions}

The NVT-based framework is  a latest alternative   setting that is preferred  over  the NPT-based framework to model multiphase fluid flow.  Based on the principles of  the NVT-based framework,  a mathematical model is proposed to describe the multi-component two-phase  flow  with partial miscibility. We combine  the first law of thermodynamics and the related physical relations to derive   the entropy balance equations, and then we derive a transport equation of the Helmholtz free energy density.  Furthermore, using the second law of thermodynamics, we derive a set of unified equations for both interfaces and bulk phases that can  describe the   partial miscibility of two fluids.  A feature of this model is that a  term involving  mass diffusions is naturally included in the momentum equation to ensure  consistency with thermodynamics.  We prove a relation between the pressure gradient  and   the gradients of  chemical potentials, which leads to   a new formulation of the  momentum conservation equation, showing that the gradients of  chemical potentials  become the primary driving force of the fluid motion. 

The total (free) energy   dissipation with time is proved for  the proposed model.  An energy-dissipation  numerical scheme is proposed based on a convex-concave  splitting of Helmholtz free energy density and a careful treatment of the coupling relations between molar densities and velocity  with a mathematical rigor.  Finally, we present numerical results to verify the effectiveness of the proposed method.

\section*{Appendix}
\begin{appendix}\label{appendix}

We describe  the detailed formulations of Helmholtz free energy density $f_{b}(\n)$ in \eqref{eqHelmholtzEnergy_a0}:
 \begin{eqnarray*}\label{eqHelmholtzEnergy_a0_01}
    f_b^{\textnormal{ideal}}(\n)= RT\sum_{i=1}^{M}n_i\(\ln n_i-1\),
\end{eqnarray*}
\begin{eqnarray*}\label{eqHelmholtzEnergy_a0_02}
    f_b^{\textnormal{repulsion}}(\n)=-nRT\ln\(1-bn\),
\end{eqnarray*}
\begin{eqnarray*}\label{eqHelmholtzEnergy_a0_03}
    f_b^{\textnormal{attraction}}(\n)= \frac{a(T)n}{2\sqrt{2}b}\ln\(\frac{1+(1-\sqrt{2})b n}{1+(1+\sqrt{2})b n}\),
\end{eqnarray*}
where  $n=\sum_{i=1}^Mn_i$ is the overall molar density,  and $R$ is the universal gas constant. Here,  $a$ and $b$ are the energy parameter and  the covolume, respectively, and  these parameters can be calculated as functions of  the mixture composition and the temperature.
We denote  by $T_{c_i}$ and $P_{c_i}$  the  $i$th component critical temperature and critical pressure, respectively.  For the $i$th component, let the reduced temperature be  $T_{r_i}=T/T_{c_i}$. 
The parameters $a_{i}$ and $b_{i}$ are calculated as
\begin{eqnarray*}
   a_{i}= 0.45724\frac{R^2T_{c_i}^2}{P_{c_i}}\[1+m_i(1-\sqrt{T_{r_i}})\]^2,~~~~b_{i}= 0.07780\frac{RT_{c_i}}{P_{c_i}}.
\end{eqnarray*}
 The coefficients $m_i$ are calculated  by the following formulas
\begin{eqnarray*}
 m_i=0.37464 + 1.54226\omega_i-  0.26992\omega_i^2 ,~~\omega_i\leq0.49,
\end{eqnarray*}
\begin{eqnarray*}
 m_i=0.379642+1.485030\omega_i-0.164423\omega_i^2 +0.016666 \omega_i^3,~~\omega_i>0.49,
\end{eqnarray*}
where $\omega_i$ is the acentric factor.
  We let the mole fraction of component $i$ be $y_i=n_i/n$. Then  $a(T)$ and $b$ are calculated by
\begin{eqnarray*}
   a =\sum_{i=1}^M\sum_{j=1}^M y_i y_j \(a_ia_j\)^{1/2}(1-k_{ij}),~~~~b =\sum_{i=1}^M  y_i b_{i},
\end{eqnarray*}
where $k_{ij}$ the given binary interaction coefficients for the energy parameters.

 We now show that for bulk homogeneous  fluids, the PR-EOS formulation \eqref{eqPREOSByMolarDens} is equivalent to  the pressure equation \eqref{eqgeneralpressure}.
 In fact, in this case, $f=f_b$ and the corresponding chemical potential of component $i$ can be calculated as
\begin{eqnarray*}\label{eqAppoxError}
 \mu_{i}&=&\frac{\partial  f_b(\n,T)}{\partial n_i}\nonumber\\
 &=& 
 RT\(\ln(n_{i})+\frac{b_{i}n}{1-bn}-\ln(1-bn)\)\nonumber\\
 &&+\frac{1}{2\sqrt{2}}\(\frac{2\sum_{j=1}^{M}n_{j}(a_{i}a_{j})^{1/2}(1-k_{ij})}{bn}-\frac{ab_{i}}{b^{2}}\)\ln\(\frac{1+(1-\sqrt{2})bn}{1+(1+\sqrt{2})bn}\)\nonumber\\
 &&+\frac{ab_{i}n}{b((\sqrt{2}-1)bn-1)(1+(1+\sqrt{2})bn)},
\end{eqnarray*}
and furthermore,  we deduce that
\begin{eqnarray*}\label{eqAppoxError}
\sum_{i=1}^M \mu_{i}n_i&=&RT\sum_{i=1}^M n_i\ln n_{i}+RT\frac{b n^2}{1-bn}-RTn\ln(1-bn)\nonumber\\
 &&+\frac{1}{2\sqrt{2}}\frac{an}{b}\ln\(\frac{1+(1-\sqrt{2})bn}{1+(1+\sqrt{2})bn}\)\nonumber\\
 &&+ \frac{a n^2}{((\sqrt{2}-1)bn-1)(1+(1+\sqrt{2})bn)}.
\end{eqnarray*}
Substituting the above equation into \eqref{eqgeneralpressure} yields  
\begin{eqnarray*}\label{eqPREOSpressureHED}
    p&=&\sum_{i=1}^M\mu_in_i-f_b\nonumber\\
    &=&RT\frac{b n^2}{1-bn}+RTn+ \frac{a n^2}{((\sqrt{2}-1)bn-1)(1+(1+\sqrt{2})bn)}\nonumber\\
    &=&\frac{nRT}{1-bn}-\frac{an^2}{1+2bn -b^2n^2 },
\end{eqnarray*}
 which is just  the PR-EOS formulation \eqref{eqPREOSByMolarDens}. 
 
 Finally, we  introduce the formulations of the  influence parameters. The pure component influence parameters $c_{i}$ are given by \cite{miqueu2004modelling}
\begin{eqnarray*}
 c_i=a_ib_i^{2/3}\[\alpha_i(1-T_{r_i})+\beta_i\],
\end{eqnarray*}
where  $\alpha_i$ and $\beta_i$ are the coefficients
 correlated merely with the acentric factor $\omega_i$ of the $i$th component
 by the following relations
\begin{eqnarray*}
 \alpha_i=-\frac{10^{-16}}{1.2326+1.3757\omega_i},~~~~\beta_i=\frac{10^{-16}}{0.9051+1.5410\omega_i}.
\end{eqnarray*}
The cross influence parameter   is generally  described  as the modified geometric mean of the pure component
influence parameters $c_i$ and $c_j$ by
\begin{eqnarray*}\label{eqCij}
 c_{ij}= (1-\beta_{ij})\sqrt{c_ic_j},
\end{eqnarray*}
where the parameters $\beta_{ij}$ are binary interaction coefficients chosen to satisfy the symmetry $c_{ij}=c_{ji}$ and $\beta_{ii}=0$.  In the numerical tests of this work, we take $\beta_{ij}=0.5$ for $i\neq j$. From the above formulations of the  influence parameters, it is generally  assumed that the influence parameters rely  on the temperature but  independent of the component  molar density.

\end{appendix}

\small

\end{document}